\documentclass{article}
\usepackage{graphicx} 
\usepackage{amsmath, amssymb, amsthm, tikz, setspace, subcaption, cite}

\usepackage[margin = 2cm]{geometry}
\usepackage[colorlinks = true]{hyperref}
\usepackage{setspace}

\newcommand{\rvline}{\hspace*{-\arraycolsep}\vline\hspace*{-\arraycolsep}}
\title{Structural and Spectral Properties of Prime Order Element Graph of Finite Abelian Groups}
\author{Tapa Manna\thanks{Email: \texttt{mannatapa24@gmail.com}}, Supriyo Dutta\thanks{Email: \texttt{dosupriyo@gmail.com}}, Baby Bhattacharya\thanks{Email: \texttt{babybhatt75@gmail.com}} \\ 
\small{Department of Mathematics}\\
\small{National Institute of Technology Agartala}\\
\small{Jirania, West Tripura, India}}
\date{}

\newtheorem{thm}{Theorem}[section]

\newtheorem{lemma}{Lemma}[section]

\newtheorem{definition}{Definition}[section]

\newtheorem{corollary}{Corollary}[section]

\DeclareMathOperator{\lcm}{lcm}
\DeclareMathOperator{\diag}{diag}

\begin{document}
	\maketitle
	\begin{abstract}
		Given a finite group $G$, the \emph{Prime Order Element (POE) Graph} $\Gamma(G)$ consists of the group elements as the vertices, and two vertices $x$ and $y$ are adjacent if and only if $o(xy)$ is prime. This paper presents a thorough structural and spectral analysis of the POE graphs associated with the finite Abelian groups of different types. The order of a finite Abelian group may be a prime or a product of primes, which influences the structure of POE graphs. The POE graph is connected when the order of the Abelian group is a square-free integer. The POE graphs of the other Abelian groups have multiple connected components. Some of these components are isomorphic to the POE graph of a lower-order group. We study various graph-theoretic properties of the components, including regularity and bipartiteness. Arranging the elements of the group in a number of particular orders, we observe the block structure in the adjacency matrix of POE graphs. It assists us in investigating the spectral properties of POE graphs. We explicitly derive the characteristic polynomials governing both integral and irrational eigenvalues, and compute the eigenvalues with multiplicity in terms of the structure of the graphs.

		
	\end{abstract}
	
	\tableofcontents
	
	\newpage 
	\section{Introduction}
	\label{introduction}
		
		The relation between the elements in a group has played a crucial role in group theory. This is a well-investigated topic over the last several decades, and the investigation enriches group theory as well as graph theory. In algebraic graph theory, defining graphs on groups is an intriguing field of research that provides many significant observations in the study of finite groups \cite{bera2018enhanced}, such as Nilpotency, cyclic subgroup structure, group exponent behavior, etc. The idea of Cayley graphs is an important outcome of that research \cite{magnus2004combinatorial, cayley1878desiderata}. There are multiple other ways to construct the graphs from groups in the literature, for instance, Power Graphs \cite{cameron2011power, kumar2021recent}, Coprime Graph \cite{banerjee2019new}, Gruenberg-Kegel Graphs \cite{maslova2016gruenberg}, Difference Graphs \cite{biswas2024difference}, Comaximal Subgroup Graphs \cite{akbari2017co, das2024co, das2025co, betz2022classifying}, Commuting Graphs \cite{akbari2006diameters}, Super Graphs \cite{arunkumar2022super, arunkumar2024super}, Join Graphs \cite{bahrami2019further}, and many more. Every one of those has its own unique significance. The interface of algebra and graph theory has significant applications in other branches of science and technology \cite{huang2018perfect, ma2020subgroup, wang2023applications, bretto2011cayley, jain2023construction}. But the primary goal of creating these graphs is to get as much information as possible about a group $G$ from the graph that corresponds to it. A comprehensive overview of several graphs originated from the groups may be found in \cite{cameron2021graphs}. In this article, we present the characteristics of the \textit{Prime Order Element} (POE) graphs of finite Abelian groups.
		
		Recall that a group $(G, \circ)$ consists of a non-empty set $G$ and a composition $\circ: G \times G \rightarrow G$, such that, $G$ is closed under $\circ$; the composition $\circ$ is associative; there exists an unique element $e\in G$, such that $e \circ x = x \circ e = x$, for all $x\in G$; also for all $x\in G$, there exists $y\in G$, such that $x \circ y = e$. All the groups considered in this article have finitely many elements. The order of an element $x \in G$ is the smallest positive integer $m$, such that $x^m = x \circ x \circ \dots \circ x (m\text{-times)} = e$. A graph is a combinatorial object $\Gamma = (V(\Gamma), E(\Gamma))$, where $V(\Gamma)$ is the set of vertices and $E(\Gamma) \subset V(\Gamma) \times V(\Gamma)$ is the set of edges.   Throughout the presentation, $p$ denotes an odd prime number. The POE graph of a group is defined as follows \cite{manna2025prime}:
		\begin{definition}\label{POE_graph_definition}
			The Prime Order Element Graph, in sort POE Graph, $\Gamma(G) = (V(\Gamma), E(\Gamma))$, of a group $G$ is a graph, such that, $V(\Gamma) = G$ and there is an edge $(x, y)$ between two distinct vertices $x$ and $y$ if order of $xy$ is a prime number.
		\end{definition}
		
		In our earlier articles \cite{manna2025prime, manna2024forbidden}, we defined POE graph of the finite groups. We explored the graph theoretic properties of POE graphs, including degree, regularity, completeness, and connectedness, etc. They are important for identifying the underlying group properties. In the article, we present the structural and spectral properties of the POE graphs of the finite Abelian groups. Find \autoref{Z_p_n_graphs}, \ref{Z_p_n_POE_graphs} \ref{Z_2_n_POE_graphs} etc. for a number of POE graphs depicted in this article. The following theorem \cite{gallian2021contemporary} presents an useful characteristic of the finite Abelian groups:

		\begin{thm}[\textbf{Fundamental Theorem of Finitely Generated Abelian Groups:}]  
			Every finitely generated abelian group $G$ is isomorphic to a direct sum of groups
			
			\[
			\mathbb{Z}^{r} \oplus \mathbb{Z}_{n_{1}} \oplus \mathbb{Z}_{n_{2}} \oplus \cdots \oplus \mathbb{Z}_{n_{k}},
			\]
			where $r \ge 0$ is a non-negative integer and $n_{1}, n_{2}, \dots, n_{k} \in \mathbb{N}$ satisfy $n_{1} \mid n_{2} \mid \cdots \mid n_{k}$.
		\end{thm}
		Putting $r = 0$ in the above theorem, we can say that all the finite Abelian groups are isomorphic to $\mathbb{Z}_{n_{1}} \oplus \mathbb{Z}_{n_{2}} \oplus \cdots \oplus \mathbb{Z}_{n_{k}}$. Order of the group will be $n = n_1 \times n_2 \times \dots n_k$. We can represent $n = p_1^{k_1} p_2^{k_2} \dots p_m^{k_m}$, where $2 \leq p_1 < p_2 < \dots < p_m$ are the prime numbers. Based on the prime factorization of the order of the finite Abelian groups we investigate the properties of the corresponding POE graph. Below we list our fundamental observations in this article:
		\begin{enumerate}
			\item 	
			$\Gamma(\mathbb{Z}_{p^n})$ is the union of $\Gamma(\mathbb{Z}_p)$ and $\frac{p^n-p}{2p}$ number of $(p-1)$-regular graphs each containing $2p$ vertices. 
			\item The integral spectrum of the graph $\Gamma((\mathbb{Z}_p)^n)$ is $\{[0]^{\frac{p^n-1}{2}},[-2]^{\frac{p^n-3}{2}}\}$ and the other non-integral spectrum are the roots of the polynomial $x^2-(p^n-3)x-(p^n-1)$.
			\item 
			$\Gamma(\mathbb{Z}_{2^{m_1}p_2^{m_2}\cdots p_k^{m_k}})$ is the union of the following connected components: 
			\begin{enumerate}
				\item 
				A component which is isomorphic to $\Gamma(\mathbb{Z}_{2p_2p_3\cdots p_k})$;
				\item 
				A component with the elements of order $4, 4p_i$ , $4p_ip_j$ for $i\neq j$, $\dots, 4\prod_{i=2}^k p_i$ where $i, j = 2,3,\cdots , k$ with a total $2\prod_{i=2}^k p_i$ number of elements;
				\item 
				A number of regular connected components each consisting of $4\prod_{i=2}^k p_i$ elements with deg$(\sum_{i=2}^kp_i)-k+2$. The total number of such components are $2^{n_1-2}p_2^{n_2-1}\dots p_k^{n_k-1}-1$.
			\end{enumerate}
			\item 
			If $G\cong {{\mathbb{Z}}_p}^n$, then the adjacency matrix $A(\Gamma(G))$ has eigenvalue $0$ with multiplicity $\frac{p^n-1}{2}$ and $-2$ with multiplicity $\frac{p^n-3}{2}$ and the remaining eigenvalues are the roots of the equation $x^2-(p^n-3)x-(p^n-1) = 0$.
			
			\item 
			The roots of the equations $(x^2-(p_2-1)x-1)=0$ and $(x^2-(p_2-5)x-(2p_2-5))=0$ are the irrational eigenvalues of $\Gamma(\mathbb{Z}_{2p_2})$.

			 \item For $k$ odd primes $p_1,p_2,\dots p_k$, $\Gamma(\mathbb{Z}_{p_1^{n_1}p_2^{n_2}\dots p_k^{n_k}})\cong {\Gamma({\mathbb{Z}_{p_1^{n_1}}}\times {\mathbb{Z}_{p_2^{n_2}}} \times \dots {\mathbb{Z}_{p_k^{n_k}}})}$ which is the union of $\Gamma(\mathbb{Z}_{p_1p_2\dots p_k})$ and $\frac{p_1^{n_1-1}p_2^{n_2-1}\dots p_k^{n_k-1}}{2}$ number of connected regular components each having $2p_1p_2\dots p_k$ vertices with regularity $p_1+p_2+\dots p_k -k$.
			
		\end{enumerate}
		
		The organization of this article is as follows: \autoref{Preliminary} introduces the fundamental definitions of group theory and graph theory that are crucial to our study. In \autoref{Properties_of_first_group}, \autoref{Properties_of_second_group}, \autoref{Properties_of_third_group}, \autoref{Properties_of_fourth_group},
		\autoref{Properties_of_fifth_group} and \autoref{Properties_of_sixth_group}, we discussed the structural and spectral properties of the POE graphs $\Gamma((\mathbb{Z}_{p})^n)$, $\Gamma(\mathbb{Z}_{p^n})$, $\Gamma(\mathbb{Z}_{2^n})$,
		$\Gamma({\mathbb{Z}_{2^{n_1}}}\times {\mathbb{Z}_{2^{n_2}}} \times \dots \times {\mathbb{Z}_{2^{n_k}}})$, $\Gamma(\mathbb{Z}_{2^{n_1}p_2^{n_2}})$ and $\Gamma(\mathbb{Z}_{p_1^{n_1}p_2^{n_2}})$, respectively. Then, we conclude the article.

	\section{Preliminary}
	\label{Preliminary}
		
		A group \cite{dummitbasic, gallian2021contemporary} $(G, \circ)$ consists of a set $G$ and a composition $\circ: G \times G \rightarrow G$, such that, $G$ is closed under $\circ$; the composition $\circ$ is associative; there exists an unique element $e\in G$, such that $e \circ a = a \circ e = a$, for all $a\in G$; also for all $a\in G$, there exists $b\in G$, such that $ab = e$. When there is no confusion about the composition $\circ$ we write $ab$ instead of $a \circ b$. The order of a group $G$ is denoted by $o(G)$ which is the number of elements in $G$.  A group $G$ is said to be a $p$-group if $o(G) = p^n$, where $p$ is a prime and $n$ is a natural number. We say $G$ is an Abelian group if for all $a, b \in G$, $ab = ba$. For any element $x \in G$ we denote $x^m = x \circ x \circ \dots \circ x(m$-times). The order of an element $x \in G$ is the smallest positive integer $m$, such that $x^m=e$. We denote $o(x) = m$. An elementary Abelian group is an Abelian $p$-group whose every non-identity element is of order $p$. Also, for any element $x \in G$, $o(x) | o(G)$. The following Lemma is much useful for subsequent calculations.
		\begin{lemma}\label{lcm_lemma}
			For any two elements $x$ and $y$ in an Abelian group $G$ we have $o(xy) = \l.c.m\{o(x),o(y)\}$.
		\end{lemma}  
		
		A cyclic group has an element $a$ which is called a generator of the group, such that $G = \{a^n : n \in \mathbb{Z}\}$. All the cyclic groups are the Abelian groups. The cyclic group $\mathbb{Z}_{p^n}$ contains only the elements of order $p^k,k=0,1,2,\dots n$,  which are $\phi(p^k)$ in number where $\phi$ is the Euler's phi function.
		
		A group isomorphism is a bijective mapping $f:G_1\rightarrow G_2$, where $G_1$ and $G_2$ be two groups, satisfying $f(xy)=f(x)f(y)$, for all $x, y \in G_1$. If such function $f$ exists we say $G_1$ and $G_2$ are isomorphic and we denote it as $G_1 \cong G_2$. Any group isomorphic to $ \mathbb{Z}_p \times \mathbb{Z}_p \times \dots \times \mathbb{Z}_p(n$-times) is called an elementary Abelian group, which is also denoted by ${\mathbb{Z}_p}^{(n)}$. Note that, $\circ({\mathbb{Z}_p}^{(n)}) = p^n$.
		
		The direct product of groups is defined as follows: 
		\begin{definition}
			The external direct product of the groups $(G_1, \circ_1), (G_2, \circ_2), \dots , (G_n, \circ_n)$ is $(G, \circ)$ such that $G = G_1 \times G_2 \times \dots \times G_n = \{(g_1,g_2,...,g_n): g_i\in G_i\}$, and for two elements $g^{(1)} = (g^{(1)}_1, g^{(1)}_2,...,g^{(1)}_n)$ and $g^{(2)} = (g^{(2)}_1, g^{(2)}_2,...,g^{(2)}_n) \in G$ the composition $\circ$ is defined by $g^{(1)} \circ g^{(2)} = (g^{(1)}_1 \circ_1 g^{(2)}_1, g^{(1)}_2 \circ_2 g^{(2)}_2,...,g^{(1)}_n \circ_n g^{(2)}_n)$. 
		\end{definition} 
		
		A graph \cite{west2001introduction} is a combinatorial object $\Gamma = (V(\Gamma), E(\Gamma))$, where $V(\Gamma)$ is a set of vertices and $E(\Gamma)$ is a set of edges. Throughout this presentation, all the graphs are simple graphs that do not have loops on the vertices, and multiple, weighted, directed edges between the vertices. The vertices $x$ and $y$ are said to be adjacent if there is an edge $(x, y) \in E(G)$. We denote $x \sim y$. We denote $x \nsim y$ to indicate $(x, y) \notin E(\Gamma)$. The degree of a vertex $v\in \Gamma(G)$ is denoted by $\deg(v)$ and defined by the number of vertices adjacent to $v$. A vertex $v \in \Gamma(G)$ is isolated if $\deg(v)=0$. A $r$-regular graph is a graph $\Gamma$ in which the degree of any vertex is $r$. We say $r$ is the regularity of the $r$-regular graph. A connected $2$-regular graph is a cycle graph. A 3-cycle has three vertices in it.
		
		A complete graph with $n$ vertices is denoted by $K_n = (V(K_n), E(K_n))$, which contains all the edges between any two distinct vertices. The complement of a graph $\Gamma$ with $n$ vertices is denoted by $\overline{\Gamma}$ with the vertex set $V(\Gamma)$ and edges set $E(\overline{\Gamma}) = E(K_n) - E(\Gamma)$. A path in a graph is a sequence of edges that connects a series of distinct vertices. A graph $\Gamma$ is connected if there is a path between any two vertices. A connected component ${\Gamma}_1$ of a graph $\Gamma$ is the maximal set of vertices and edges such that between any two vertices in $\Gamma_1$ there exists a path. The distance between two vertices $x,y$ in a connected graph $\Gamma$ is denoted by $d(x, y)$, which is the length of the shortest path between $x$ and $y$. 
		
		The adjacency matrix of a graph $\Gamma$ with $n$ vertices $v_1, v_2, \dots v_n$ is $A(\Gamma)_{n \times n} = (a_{ij})_{n \times n}$, where 
		\begin{equation}
			a_{i, j} = 
			\begin{cases}
				1, & \text{ if } v_i\sim v_j; \\
				0, & \text{otherwise}.
			\end{cases}
		\end{equation} 
		The spectrum of a matrix $P$ is denoted by $\Lambda(P)$, which is the multi-set of spectrum of $P$. The spectrum of $\Gamma$ is $\Lambda(A(\Gamma)) = \{[\lambda_1]^{m_1}, [\lambda_2]^{m_2}, \dots [\lambda_k]^{m_k}\}$, where $m_i$ denotes the multiplicity of the spectrum $\lambda_i$ in $[\lambda_i]^{m_i}$. The degree matrix of a graph $\Gamma$ is a diagonal matrix $D(\Gamma) = \operatorname{diag}\{\deg(v_i): i = 1, 2, \dots n\}$. 
		
		We say two graphs $\Gamma=(V(\Gamma),E(\Gamma))$ and $\Upsilon = (V(\Upsilon), E(\Upsilon))$ are isomorphic if there is a bijection $f : V(\Gamma) \to V(\Upsilon)$ such that $(x, y) \in E(\Gamma)$ if and only if $(f(x), f(y)) \in E(\Upsilon)$. The union of $\Gamma$ and $\Upsilon$ is a new graph $\Gamma \cup \Upsilon$ such that $V(\Gamma \cup \Upsilon) = V(\Gamma) \cup V(\Upsilon)$ and $E(\Gamma \cup \Upsilon) = E(\Gamma) \cup E(\Upsilon)$. We can arrange the vertices in $\Gamma \cup \Upsilon$ such that $A(\Gamma \cup \Upsilon) = \diag\{A(\Gamma), A(\Upsilon)\}$. If the connected components of a graph $\Gamma$ are $C^{(1)}, C^{(2)}, \dots C^{(l)}$ we write $\Gamma = C^{(1)} \cup C^{(2)} \cup \dots C^{(l)}$. Therefore, $A(\Gamma) = \diag\{A(C^{(1)}), A(C^{(2)}), \dots, A(C^{(l)})\}$. As the spectrum of a block diagonal matrix are the spectrum of the individual block matrices, we have
		\begin{equation}\label{block_diagonal_spectra}
			\Lambda(A(\Gamma)) = \Lambda(A(C^{(1)})) \cup \Lambda(A(C^{(2)})) \cup \dots \cup \Lambda(A(C^{(l)})).
		\end{equation}
		
		Let $G = (V(G), E(G))$ and $H = (V(H), E(H))$ be two graphs. We say $H$ is a subgraph of $G$ if $V(H) \subset V(G)$ and $E(G) \subset E(H)$. Let $C \subset V(G)$. The induced subgraph generated by $C$ is denoted by $\langle C \rangle = (V(\langle C \rangle), E(\langle C \rangle))$ such that $V(\langle C \rangle) = C$ and $E(\langle C \rangle) = \{(u, v): ~\text{if}~ u, v \in C ~\text{and}~ (u, v) \in E(G)\}$.
		
		This following lemma of matrix analysis is useful in this article:
		\begin{lemma} 
			\label{commuting_matrix_eigenvalues}
			Let $A$ and $B$ be two diagonalizable and commuting matrices of order $n$. The eigenvalues of $A$ and $B$ are $\lambda_1, \lambda_1, \dots \lambda_n$, and $\mu_1, \mu_2, \dots \mu_n$, respectively. Then the eigenvalues of $A + B$ are $\lambda_1 + \mu_{i_1}, \lambda_2 + \mu_{i_2}, \dots \lambda_n + \mu_{i_n}$, for some permutation
			$i_1, i_2, \dots i_n$ of $1, 2, \dots n$ \cite[Chapter 1]{horn2012matrix}.
		\end{lemma}

	\section{Properties of $\Gamma({\mathbb{Z}_p}^{(n)})$}
	\label{Properties_of_first_group}
	
		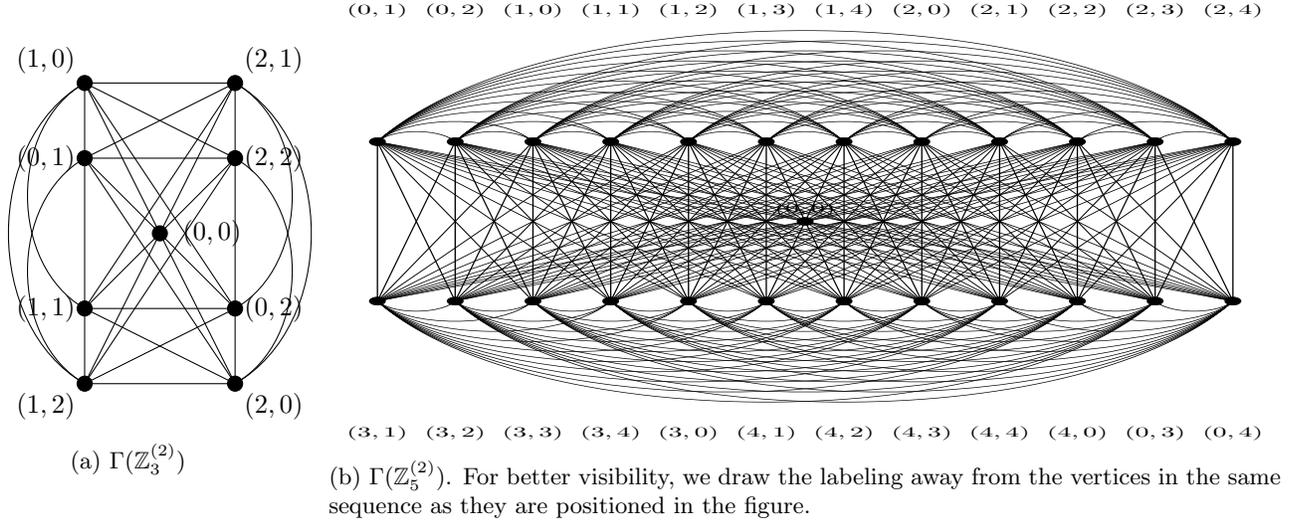
\begin{figure}
		\centering

		\begin{subfigure}[a]{.22\textwidth}
			\centering
			\begin{tikzpicture}
				\node [right] at (.18, 0) {${(0,0)}$};
				\draw [fill, black] (0, 0) circle[radius = 1mm];
				\node [left] at (-1, 1) {${(0,1)}$};
				\draw [fill, black] (-1, 1) circle[radius = 1mm];
				\node [above left] at (-1, 2) {${(1,0)}$};
				\draw [fill, black] (-1, 2) circle[radius = 1mm];
				\node [left] at (-1, -1) {${(1,1)}$};
				\draw [fill, black] (-1, -1) circle[radius = 1mm];
				\node [below left] at (-1, -2) {${(1,2)}$};
				\draw [fill, black] (-1, -2) circle[radius = 1mm];
				\node [above right] at (1, 2) {${(2,1)}$};
				\draw [fill, black] (1, 2) circle[radius = 1mm];
				\node [right] at (1, 1) {${(2,2)}$};
				\draw [fill, black] (1, 1) circle[radius = 1mm];
				\node [right] at (1, -1) {${(0,2)}$};
				\draw [fill, black] (1, -1) circle[radius = 1mm];
				\node [below right] at (1, -2) {${(2,0)}$};
				\draw [fill, black] (1, -2) circle[radius = 1mm];
				\draw [] (0,0) -- (-1,2);
				\draw [] (0,0) -- (-1,1);
				\draw [] (0,0) -- (-1,-1);
				\draw [] (0,0) -- (-1,-2);
				\draw [] (0,0) -- (1,2);
				\draw [] (0,0) -- (1,1);
				\draw [] (0,0) -- (1,-1);
				\draw [] (0,0) -- (1,-2);
				\draw [] (-1,2) -- (1,2);
				\draw [] (-1,2) -- (1,1);
				\draw [] (-1,2) -- (1,-1);
				\draw [] (-1,1) -- (1,2);
				\draw [] (-1,1) -- (1,1);
				\draw [] (-1,1) -- (1,-2);
				\draw [] (-1,-1) -- (1,2);
				\draw [] (-1,-1) -- (1,-1);
				\draw [] (-1,-1) -- (1,-2);
				\draw [] (-1,-2) -- (1,1);
				\draw [] (-1,-2) -- (1,-1);
				\draw [] (-1,-2) -- (1,-2);
				\draw [] (-1,2) -- (-1,1);
				\draw [] (-1,-1) -- (-1,-2);
				\path [-] (-1,2) edge[bend right=60] node {} (-1,-1);
				\path [-] (-1,2) edge[bend right=60] node {} (-1,-2);
				\draw [] (-1,1) -- (-1,-1);
				\path [-] (-1,1) edge[bend right=60] node {} (-1,-2);
				\draw [] (1,2) -- (1,1);
				\path [-] (1,2) edge[bend left=60] node {} (1,-1);
				\path [-] (1,2) edge[bend left=60] node {} (1,-2);
				\draw [] (1,1) -- (1,-1);
				\path [-] (1,1) edge[bend left=60] node {} (1,-2);
				\draw [] (1,-1) -- (1,-2);
			\end{tikzpicture}
			\caption{$\Gamma(\mathbb{Z}_3^{(2)})$}
		\end{subfigure}
		\hspace{.5cm}
		\begin{subfigure}[a]{.72\textwidth}
			\resizebox{\textwidth}{6cm}{
				\begin{tikzpicture}
					\node [above] at (-0.5,0) {${(0,0)}$};
					\draw [fill, black] (-0.5,0) circle[radius = 1mm];
					
					\node [above] at (-6,5) {${(0,1)}$};
					\draw [fill, black] (-6,2) circle[radius = 1mm];
					\node [above] at (-5,5) {${(0,2)}$};
					\draw [fill, black] (-5,2) circle[radius = 1mm];
					\node [above] at (-4,5) {${(1,0)}$};
					\draw [fill, black] (-4,2) circle[radius = 1mm];
					\node [above] at (-3,5) {${(1,1)}$};
					\draw [fill, black] (-3,2) circle[radius = 1mm];
					\node [above] at (-2,5) {${(1,2)}$};
					\draw [fill, black] (-2,2) circle[radius = 1mm];
					\node [above] at (-1,5) {${(1,3)}$};
					\draw [fill, black] (-1,2) circle[radius = 1mm];
					\node [above] at (0,5) {${(1,4)}$};
					\draw [fill, black] (0,2) circle[radius = 1mm];
					\node [above] at (1,5) {${(2,0)}$};
					\draw [fill, black] (1,2) circle[radius = 1mm];
					\node [above] at (2,5) {${(2,1)}$};
					\draw [fill, black] (2,2) circle[radius = 1mm];
					\node [above] at (3,5) {${(2,2)}$};
					\draw [fill, black] (3,2) circle[radius = 1mm];
					\node [above] at (4,5) {${(2,3)}$};
					\draw [fill, black] (4,2) circle[radius = 1mm];
					\node [above] at (5,5) {${(2,4)}$};
					\draw [fill, black] (5,2) circle[radius = 1mm];
					\node [below] at (-6,-5) {${(3,1)}$};
					\draw [fill, black] (-6,-2) circle[radius = 1mm];
					\node [below] at (-5,-5) {${(3,2)}$};
					\draw [fill, black] (-5,-2) circle[radius = 1mm];
					\node [below] at (-4,-5) {${(3,3)}$};
					\draw [fill, black] (-4,-2) circle[radius = 1mm];
					\node [below] at (-3,-5) {${(3,4)}$};
					\draw [fill, black] (-3,-2) circle[radius = 1mm];
					\node [below] at (-2,-5) {${(3,0)}$};
					\draw [fill, black] (-2,-2) circle[radius = 1mm];
					\node [below] at (-1,-5) {${(4,1)}$};
					\draw [fill, black] (-1,-2) circle[radius = 1mm];
					\node [below] at (0,-5) {${(4,2)}$};
					\draw [fill, black] (0,-2) circle[radius = 1mm];
					\node [below] at (1,-5) {${(4,3)}$};
					\draw [fill, black] (1,-2) circle[radius = 1mm];
					\node [below] at (2,-5) {${(4,4)}$};
					\draw [fill, black] (2,-2) circle[radius = 1mm];
					\node [below] at (3,-5) {${(4,0)}$};
					\draw [fill, black] (3,-2) circle[radius = 1mm];
					\node [below] at (4,-5) {${(0,3)}$};
					\draw [fill, black] (4,-2) circle[radius = 1mm];
					\node [below] at (5,-5) {${(0,4)}$};
					\draw [fill, black] (5,-2) circle[radius = 1mm];
					
					\draw [] (-0.5,0) -- (-6,2);
					\draw [] (-0.5,0) -- (-5,2);
					\draw [] (-0.5,0) -- (-4,2);
					\draw [] (-0.5,0) -- (-3,2);
					\draw [] (-0.5,0) -- (-2,2);
					\draw [] (-0.5,0) -- (-1,2);
					\draw [] (-0.5,0) -- (0,2);
					\draw [] (-0.5,0) -- (1,2);
					\draw [] (-0.5,0) -- (2,2);
					\draw [] (-0.5,0) -- (3,2);
					\draw [] (-0.5,0) -- (4,2);
					\draw [] (-0.5,0) -- (5,2);
					\draw [] (-0.5,0) -- (-6,-2);
					\draw [] (-0.5,0) -- (-5,-2);
					\draw [] (-0.5,0) -- (-4,-2);
					\draw [] (-0.5,0) -- (-3,-2);
					\draw [] (-0.5,0) -- (-2,-2);
					\draw [] (-0.5,0) -- (-1,-2);
					\draw [] (-0.5,0) -- (0,-2);
					\draw [] (-0.5,0) -- (1,-2);
					\draw [] (-0.5,0) -- (2,-2);
					\draw [] (-0.5,0) -- (3,-2);
					\draw [] (-0.5,0) -- (4,-2);
					\draw [] (-0.5,0) -- (5,-2);
					\draw [] (-6,2) -- (-6,-2);
					\draw [] (-6,2) -- (-5,-2);
					\draw [] (-6,2) -- (-4,-2);
					\draw [] (-6,2) -- (-3,-2);
					\draw [] (-6,2) -- (-2,-2);
					\draw [] (-6,2) -- (-1,-2);
					\draw [] (-6,2) -- (0,-2);
					\draw [] (-6,2) -- (1,-2);
					\draw [] (-6,2) -- (2,-2);
					\draw [] (-6,2) -- (3,-2);
					\draw [] (-6,2) -- (4,-2);
					\draw [] (-5,2) -- (-6,-2);
					\draw [] (-5,2) -- (-5,-2);
					\draw [] (-5,2) -- (-4,-2);
					\draw [] (-5,2) -- (-3,-2);
					\draw [] (-5,2) -- (-2,-2);
					\draw [] (-5,2) -- (-1,-2);
					\draw [] (-5,2) -- (0,-2);
					\draw [] (-5,2) -- (1,-2);
					\draw [] (-5,2) -- (2,-2);
					\draw [] (-5,2) -- (3,-2);
					\draw [] (-5,2) -- (5,-2);
					\draw [] (-4,2) -- (-6,-2);
					\draw [] (-4,2) -- (-5,-2);
					\draw [] (-4,2) -- (-4,-2);
					\draw [] (-4,2) -- (-3,-2);
					\draw [] (-4,2) -- (-2,-2);
					\draw [] (-4,2) -- (-1,-2);
					\draw [] (-4,2) -- (0,-2);
					\draw [] (-4,2) -- (1,-2);
					\draw [] (-4,2) -- (2,-2);
					\draw [] (-4,2) -- (4,-2);
					\draw [] (-4,2) -- (5,-2);
					\draw [] (-3,2) -- (-6,-2);
					\draw [] (-3,2) -- (-5,-2);
					\draw [] (-3,2) -- (-4,-2);
					\draw [] (-3,2) -- (-3,-2);
					\draw [] (-3,2) -- (-2,-2);
					\draw [] (-3,2) -- (-1,-2);
					\draw [] (-3,2) -- (0,-2);
					\draw [] (-3,2) -- (1,-2);
					\draw [] (-3,2) -- (3,-2);
					\draw [] (-3,2) -- (4,-2);
					\draw [] (-3,2) -- (5,-2);
					\draw [] (-2,2) -- (-6,-2);
					\draw [] (-2,2) -- (-5,-2);
					\draw [] (-2,2) -- (-4,-2);
					\draw [] (-2,2) -- (-3,-2);
					\draw [] (-2,2) -- (-2,-2);
					\draw [] (-2,2) -- (-1,-2);
					\draw [] (-2,2) -- (0,-2);
					\draw [] (-2,2) -- (2,-2);
					\draw [] (-2,2) -- (3,-2);
					\draw [] (-2,2) -- (4,-2);
					\draw [] (-2,2) -- (5,-2);
					\draw [] (-1,2) -- (-6,-2);
					\draw [] (-1,2) -- (-5,-2);
					\draw [] (-1,2) -- (-4,-2);
					\draw [] (-1,2) -- (-3,-2);
					\draw [] (-1,2) -- (-2,-2);
					\draw [] (-1,2) -- (-1,-2);
					\draw [] (-1,2) -- (1,-2);
					\draw [] (-1,2) -- (2,-2);
					\draw [] (-1,2) -- (3,-2);
					\draw [] (-1,2) -- (4,-2);
					\draw [] (-1,2) -- (5,-2);
					\draw [] (0,2) -- (-6,-2);
					\draw [] (0,2) -- (-5,-2);
					\draw [] (0,2) -- (-4,-2);
					\draw [] (0,2) -- (-3,-2);
					\draw [] (0,2) -- (-2,-2);
					\draw [] (0,2) -- (0,-2);
					\draw [] (0,2) -- (1,-2);
					\draw [] (0,2) -- (2,-2);
					\draw [] (0,2) -- (3,-2);
					\draw [] (0,2) -- (4,-2);
					\draw [] (0,2) -- (5,-2);
					\draw [] (1,2) -- (-6,-2);
					\draw [] (1,2) -- (-5,-2);
					\draw [] (1,2) -- (-4,-2);
					\draw [] (1,2) -- (-3,-2);
					\draw [] (1,2) -- (-1,-2);
					\draw [] (1,2) -- (0,-2);
					\draw [] (1,2) -- (1,-2);
					\draw [] (1,2) -- (2,-2);
					\draw [] (1,2) -- (3,-2);
					\draw [] (1,2) -- (4,-2);
					\draw [] (1,2) -- (5,-2);
					\draw [] (2,2) -- (-6,-2);
					\draw [] (2,2) -- (-5,-2);
					\draw [] (2,2) -- (-4,-2);
					\draw [] (2,2) -- (-2,-2);
					\draw [] (2,2) -- (-1,-2);
					\draw [] (2,2) -- (0,-2);
					\draw [] (2,2) -- (1,-2);
					\draw [] (2,2) -- (2,-2);
					\draw [] (2,2) -- (3,-2);
					\draw [] (2,2) -- (4,-2);
					\draw [] (2,2) -- (5,-2);
					\draw [] (3,2) -- (-6,-2);
					\draw [] (3,2) -- (-5,-2);
					\draw [] (3,2) -- (-3,-2);
					\draw [] (3,2) -- (-2,-2);
					\draw [] (3,2) -- (-1,-2);
					\draw [] (3,2) -- (0,-2);
					\draw [] (3,2) -- (1,-2);
					\draw [] (3,2) -- (2,-2);
					\draw [] (3,2) -- (3,-2);
					\draw [] (3,2) -- (4,-2);
					\draw [] (3,2) -- (5,-2);
					\draw [] (4,2) -- (-6,-2);
					\draw [] (4,2) -- (-4,-2);
					\draw [] (4,2) -- (-3,-2);
					\draw [] (4,2) -- (-2,-2);
					\draw [] (4,2) -- (-1,-2);
					\draw [] (4,2) -- (0,-2);
					\draw [] (4,2) -- (1,-2);
					\draw [] (4,2) -- (2,-2);
					\draw [] (4,2) -- (3,-2);
					\draw [] (4,2) -- (4,-2);
					\draw [] (4,2) -- (5,-2);
					\draw [] (5,2) -- (-5,-2);
					\draw [] (5,2) -- (-4,-2);
					\draw [] (5,2) -- (-3,-2);
					\draw [] (5,2) -- (-2,-2);
					\draw [] (5,2) -- (-1,-2);
					\draw [] (5,2) -- (0,-2);
					\draw [] (5,2) -- (1,-2);
					\draw [] (5,2) -- (2,-2);
					\draw [] (5,2) -- (3,-2);
					\draw [] (5,2) -- (4,-2);
					\draw [] (5,2) -- (5,-2);
					\path [-] (-6,2) edge[bend left=60] node {} (-5,2);
					\path [-] (-6,2) edge[bend left=60] node {} (-4,2);
					\path [-] (-6,2) edge[bend left=60] node {} (-3,2);
					\path [-] (-6,2) edge[bend left=60] node {} (-2,2);
					\path [-] (-6,2) edge[bend left=60] node {} (-1,2);
					\path [-] (-6,2) edge[bend left=60] node {} (0,2);
					\path [-] (-6,2) edge[bend left=60] node {} (1,2);
					\path [-] (-6,2) edge[bend left=60] node {} (2,2);
					\path [-] (-6,2) edge[bend left=60] node {} (3,2);
					\path [-] (-6,2) edge[bend left=60] node {} (4,2);
					\path [-] (-6,2) edge[bend left=60] node {} (5,2);
					\path [-] (-5,2) edge[bend left=60] node {} (-4,2);
					\path [-] (-5,2) edge[bend left=60] node {} (-3,2);
					\path [-] (-5,2) edge[bend left=60] node {} (-2,2);
					\path [-] (-5,2) edge[bend left=60] node {} (-1,2);
					\path [-] (-5,2) edge[bend left=60] node {} (0,2);
					\path [-] (-5,2) edge[bend left=60] node {} (1,2);
					\path [-] (-5,2) edge[bend left=60] node {} (2,2);
					\path [-] (-5,2) edge[bend left=60] node {} (3,2);
					\path [-] (-5,2) edge[bend left=60] node {} (4,2);
					\path [-] (-5,2) edge[bend left=60] node {} (5,2);
					\path [-] (-4,2) edge[bend left=60] node {} (-3,2);
					\path [-] (-4,2) edge[bend left=60] node {} (-2,2);
					\path [-] (-4,2) edge[bend left=60] node {} (-1,2);
					\path [-] (-4,2) edge[bend left=60] node {} (0,2);
					\path [-] (-4,2) edge[bend left=60] node {} (1,2);
					\path [-] (-4,2) edge[bend left=60] node {} (2,2);
					\path [-] (-4,2) edge[bend left=60] node {} (3,2);
					\path [-] (-4,2) edge[bend left=60] node {} (4,2);
					\path [-] (-4,2) edge[bend left=60] node {} (5,2);
					\path [-] (-3,2) edge[bend left=60] node {} (-2,2);
					\path [-] (-3,2) edge[bend left=60] node {} (-1,2);
					\path [-] (-3,2) edge[bend left=60] node {} (0,2);
					\path [-] (-3,2) edge[bend left=60] node {} (1,2);
					\path [-] (-3,2) edge[bend left=60] node {} (2,2);
					\path [-] (-3,2) edge[bend left=60] node {} (3,2);
					\path [-] (-3,2) edge[bend left=60] node {} (4,2);
					\path [-] (-3,2) edge[bend left=60] node {} (5,2);
					\path [-] (-2,2) edge[bend left=60] node {} (-1,2);
					\path [-] (-2,2) edge[bend left=60] node {} (0,2);
					\path [-] (-2,2) edge[bend left=60] node {} (1,2);
					\path [-] (-2,2) edge[bend left=60] node {} (2,2);
					\path [-] (-2,2) edge[bend left=60] node {} (3,2);
					\path [-] (-2,2) edge[bend left=60] node {} (4,2);
					\path [-] (-2,2) edge[bend left=60] node {} (5,2);
					\path [-] (-1,2) edge[bend left=60] node {} (0,2);
					\path [-] (-1,2) edge[bend left=60] node {} (1,2);
					\path [-] (-1,2) edge[bend left=60] node {} (2,2);
					\path [-] (-1,2) edge[bend left=60] node {} (3,2);
					\path [-] (-1,2) edge[bend left=60] node {} (4,2);
					\path [-] (-1,2) edge[bend left=60] node {} (5,2);
					\path [-] (0,2) edge[bend left=60] node {} (1,2);
					\path [-] (0,2) edge[bend left=60] node {} (2,2);
					\path [-] (0,2) edge[bend left=60] node {} (3,2);
					\path [-] (0,2) edge[bend left=60] node {} (4,2);
					\path [-] (0,2) edge[bend left=60] node {} (5,2);
					\path [-] (1,2) edge[bend left=60] node {} (2,2);
					\path [-] (1,2) edge[bend left=60] node {} (3,2);
					\path [-] (1,2) edge[bend left=60] node {} (4,2);
					\path [-] (1,2) edge[bend left=60] node {} (5,2);
					\path [-] (2,2) edge[bend left=60] node {} (3,2);
					\path [-] (2,2) edge[bend left=60] node {} (4,2);
					\path [-] (2,2) edge[bend left=60] node {} (5,2);
					\path [-] (3,2) edge[bend left=60] node {} (4,2);
					\path [-] (3,2) edge[bend left=60] node {} (5,2);
					\path [-] (4,2) edge[bend left=60] node {} (5,2);
					\path [-] (-6,-2) edge[bend right=60] node {} (-5,-2);
					\path [-] (-6,-2) edge[bend right=60] node {} (-4,-2);
					\path [-] (-6,-2) edge[bend right=60] node {} (-3,-2);
					\path [-] (-6,-2) edge[bend right=60] node {} (-2,-2);
					\path [-] (-6,-2) edge[bend right=60] node {} (-1,-2);
					\path [-] (-6,-2) edge[bend right=60] node {} (0,-2);
					\path [-] (-6,-2) edge[bend right=60] node {} (1,-2);
					\path [-] (-6,-2) edge[bend right=60] node {} (2,-2);
					\path [-] (-6,-2) edge[bend right=60] node {} (3,-2);
					\path [-] (-6,-2) edge[bend right=60] node {} (4,-2);
					\path [-] (-5,-2) edge[bend right=60] node {} (-4,-2);
					\path [-] (-5,-2) edge[bend right=60] node {} (-3,-2);
					\path [-] (-5,-2) edge[bend right=60] node {} (-2,-2);
					\path [-] (-5,-2) edge[bend right=60] node {} (-1,-2);
					\path [-] (-5,-2) edge[bend right=60] node {} (0,-2);
					\path [-] (-5,-2) edge[bend right=60] node {} (1,-2);
					\path [-] (-5,-2) edge[bend right=60] node {} (2,-2);
					\path [-] (-5,-2) edge[bend right=60] node {} (3,-2);
					\path [-] (-5,-2) edge[bend right=60] node {} (4,-2);
					\path [-] (-5,-2) edge[bend right=60] node {} (5,-2);
					\path [-] (-4,-2) edge[bend right=60] node {} (-3,-2);
					\path [-] (-4,-2) edge[bend right=60] node {} (-2,-2);
					\path [-] (-4,-2) edge[bend right=60] node {} (-1,-2);
					\path [-] (-4,-2) edge[bend right=60] node {} (0,-2);
					\path [-] (-4,-2) edge[bend right=60] node {} (1,-2);
					\path [-] (-4,-2) edge[bend right=60] node {} (2,-2);
					\path [-] (-4,-2) edge[bend right=60] node {} (3,-2);
					\path [-] (-4,-2) edge[bend right=60] node {} (4,-2);
					\path [-] (-4,-2) edge[bend right=60] node {} (5,-2);
					\path [-] (-3,-2) edge[bend right=60] node {} (-2,-2);
					\path [-] (-3,-2) edge[bend right=60] node {} (-1,-2);
					\path [-] (-3,-2) edge[bend right=60] node {} (0,-2);
					\path [-] (-3,-2) edge[bend right=60] node {} (1,-2);
					\path [-] (-3,-2) edge[bend right=60] node {} (2,-2);
					\path [-] (-3,-2) edge[bend right=60] node {} (3,-2);
					\path [-] (-3,-2) edge[bend right=60] node {} (4,-2);
					\path [-] (-3,-2) edge[bend right=60] node {} (5,-2);
					\path [-] (-2,-2) edge[bend right=60] node {} (-1,-2);
					\path [-] (-2,-2) edge[bend right=60] node {} (0,-2);
					\path [-] (-2,-2) edge[bend right=60] node {} (1,-2);
					\path [-] (-2,-2) edge[bend right=60] node {} (2,-2);
					\path [-] (-2,-2) edge[bend right=60] node {} (3,-2);
					\path [-] (-2,-2) edge[bend right=60] node {} (4,-2);
					\path [-] (-2,-2) edge[bend right=60] node {} (5,-2);
					\path [-] (-1,-2) edge[bend right=60] node {} (0,-2);
					\path [-] (-1,-2) edge[bend right=60] node {} (1,-2);
					\path [-] (-1,-2) edge[bend right=60] node {} (2,-2);
					\path [-] (-1,-2) edge[bend right=60] node {} (3,-2);
					\path [-] (-1,-2) edge[bend right=60] node {} (4,-2);
					\path [-] (-1,-2) edge[bend right=60] node {} (5,-2);
					\path [-] (0,-2) edge[bend right=60] node {} (1,-2);
					\path [-] (0,-2) edge[bend right=60] node {} (2,-2);
					\path [-] (0,-2) edge[bend right=60] node {} (3,-2);
					\path [-] (0,-2) edge[bend right=60] node {} (4,-2);
					\path [-] (0,-2) edge[bend right=60] node {} (5,-2);
					\path [-] (1,-2) edge[bend right=60] node {} (2,-2);
					\path [-] (1,-2) edge[bend right=60] node {} (3,-2);
					\path [-] (1,-2) edge[bend right=60] node {} (4,-2);
					\path [-] (1,-2) edge[bend right=60] node {} (5,-2);
					\path [-] (2,-2) edge[bend right=60] node {} (3,-2);
					\path [-] (2,-2) edge[bend right=60] node {} (4,-2);
					\path [-] (2,-2) edge[bend right=60] node {} (5,-2);
					\path [-] (3,-2) edge[bend right=60] node {} (4,-2);
					\path [-] (3,-2) edge[bend right=60] node {} (5,-2);
					\path [-] (4,-2) edge[bend right=60] node {} (5,-2);
				\end{tikzpicture}
			}
			\caption{$\Gamma({\mathbb{Z}_{5}^{(2)}})$. For better visibility, we draw the labeling away from the vertices in the same sequence as they are positioned in the figure.}
		\end{subfigure}
		\caption{POE graphs of $\mathbb{Z}_p^{(n)}$ for $n = 2$ and $p = 3$ and $5$. In both the graphs the vertex $(0, 0)$ represents the identity element of the graph. It is connected to all other vertices.}
		\label{Z_p_n_graphs}
	\end{figure}

		The group ${\mathbb{Z}_p}^{(n)} \equiv \mathbb{Z}_p \times \mathbb{Z}_p \times \dots \mathbb{Z}_p(n$-times) consists of the identity and $p^n-1$ elements of order $p$. The definition of POE graphs indicates that ${\mathbb{Z}_p}^{(n)}$ consists of $p^n$ vertices. As the order of any non-identity elements in ${\mathbb{Z}_p}^{(n)}$ is $p$, then the vertex corresponding to the identity element $e$ is connected to all other vertices. Hence, degree of $e$ in $\Gamma({\mathbb{Z}_p}^{(n)})$ is $(p^n - 1)$. If for any two elements $x, y \in {\mathbb{Z}_p}^{(n)}$ with  $x\neq y^{-1}$, we have $o(xy) = p$, that is, $x\sim y$. Therefore, the degree of any non-identity element in $\Gamma({\mathbb{Z}_p}^{(n)})$ is $(p^n - 2)$. The characteristics of $\Gamma({\mathbb{Z}_p})$ may be derived by putting $n = 1$ in $\Gamma({\mathbb{Z}_p}^{(n)})$. For $p = 3$ and $5$ and $n = 2$ we depict the POE graphs in \autoref{Z_p_n_graphs}.
			
		\begin{thm}{\label{First spectral theorem}}
			The integral spectrum of the graph $\Gamma({\mathbb{Z}_p}^{(n)})$ is $\{[0]^{\frac{p^n-1}{2}},[-2]^{\frac{p^n-3}{2}}\}$ and the other non-integral spectrum are the roots of the polynomial $x^2-(p^n-3)x-(p^n-1)$.
		\end{thm}
			
		\begin{proof}
			We arrange elements in ${\mathbb{Z}_p}^{(n)}$ as $e, x_1, x_2,\dots, x_{\frac{p^n-1}{2}},  x_{\frac{p^n-1}{2}}^{-1}, \dots, x_2^{-1}, x_1^{-1}$. According to the definition of POE graphs, mentioned in  \autoref{POE_graph_definition}, the identity element $e$ is adjacent to all the other elements of ${\mathbb{Z}_p}^{(n)}$. Also, any non-identity element is adjacent to all the elements except itself and its inverse. Therefore, the characteristic equation of $A(\Gamma({\mathbb{Z}_p}^{(n)}) = (a_{i, j})_{p^n \times p^n}$ can be represented by $\det[A(\Gamma({\mathbb{Z}_p}^{(n)})-{\lambda}I_{p^n}] = 0$,
			 $$ \text{or}~ \det
			 \left[\begin{array}{c | c c c c c c c }
			 	-\lambda & 1 & 1 & 1 & \dots & 1 & 1 & 1 \\
			 	\hline 
			 	1 & -\lambda  & 1 & 1 & \dots & 1 & 1 & 0\\
			 	1 & 1 & -\lambda & 1 & \dots  & 1 & 0 & 1 \\
			 	1 & 1 & 1 & -\lambda & \dots & 0 & 1 & 1 \\
			 	\vdots & \vdots & \vdots & \vdots & \ddots & \vdots & \vdots & \vdots \\ 
			 	1 & 1 & 1 & 0 & \dots & -\lambda & 1 & 1\\
			 	1 & 1 & 0 & 1 & \dots & 1 & -\lambda & 1\\ 
			 	1 & 0 & 1 & 1 & \dots  & 1 & 1 & -\lambda\\ 
			 \end{array}\right]_{p^n \times p^n} = 0.
			 $$ 
			 The matrix $A(\Gamma({\mathbb{Z}_p}^{(n)})-{\lambda}I_{p^n}$ is a block matrix consists of $4$ blocks which are given by $A=[\lambda]$, $B= [1 \hspace{5mm} 1 \hspace{5mm}1 \hspace{5mm}1\hspace{5mm} \dots \hspace{5mm}1\hspace{5mm} 1]$, $C= B^T$ and\\
			 \begin{equation}
			 	D=
			 	\begin{cases}
			 		-\lambda & \text{ for } a_{ii};\\
			 		0 & \text{ for } a_{ij},i+j=p^n; \\
			 		1 & \text{ otherwise }.\\
			 	\end{cases}
			 \end{equation}\\
			 Since, $\det(A) \neq 0$ and $\det(D)\neq 0$, then $\det[A(\Gamma({\mathbb{Z}_p}^{(n)})-{\lambda}I_{p^n}]$ can be expressed as  
			 \begin{equation} \label{first block-det equation}
			 	\det{\begin{bmatrix}
			 			A &  B\\
			 			C & D\\
			 	\end{bmatrix}} = \det(A).\det[D - CA^{-1}B] = (-\lambda).\det \left[D - \left(-\frac{1}{\lambda}\right) B^TB \right] = (-\lambda).det (P),
			 \end{equation}
			 where $P = (p_{i, j})_{(p^n - 1) \times (p^n - 1)}$, and
			 \begin{equation}
			 	p_{i, j} =
			 	\begin{cases}
			 		-\lambda+\frac{1}{\lambda} & \text{when}~ i = j ~\text{which are corresponding to } a_{i + 1, i + 1};\\
			 		\frac{1}{\lambda} & \text{ for } i \neq j, ~\text{and}~ i+j=p^n;\\
			 		1+\frac{1}{\lambda} & \text{ otherwise }.\\
			 	\end{cases}
			 \end{equation}\\
			 The matrix $P$ can be further divided into block matrices as $P = \begin{bmatrix}
			 	A_1 & B_1 \\
			 	B_1 & A_1 \\
			 \end{bmatrix}$, where $A_1 = (a'_{i,j})$, and $B_1=(b'_{i,j})$ are matrices of order $\frac{p^n - 1}{2}$ as well as
			 \begin{equation}
			 	 a'_{i,j}=
			 	\begin{cases}
			 		-\lambda+\frac{1}{\lambda} & \text{ for } i = j,\\
			 		1+\frac{1}{\lambda} & \text{ otherwise;}\\
			 	\end{cases}
			 \text{and}~
			 	 b'_{i,j}=
			 	\begin{cases}
			 		\frac{1}{\lambda} & \text{ for } a_{ij},i+j=\frac{p^n}{2},\\
			 		1+\frac{1}{\lambda} & \text{ otherwise.}\\
			 	\end{cases}
			 \end{equation}
			 Writing $A_1$ and $B_1$ in matrix form, we observe that they are commuting matrices. Therefore, 
			 \begin{equation*}
			 	\begin{split}
			 		& \det(P) = \det\begin{pmatrix}
			 			A_1 & \rvline & B_1 \\
			 			\hline
			 			B_1 & \rvline & A_1 \\
			 		\end{pmatrix} = \det(A_1+B_1)\det(A_1-B_1)\\
			 		& =\begin{vmatrix}
			 			(1-\lambda+\frac{2}{\lambda}) & (2+\frac{2}{\lambda}) & \dots & (2+\frac{2}{\lambda}) & (1+\frac{2}{\lambda})  \\
			 			(2+\frac{2}{\lambda}) & (1-\lambda+\frac{2}{\lambda}) & \dots & (1+\frac{2}{\lambda}) & (2+\frac{2}{\lambda})\\
			 			\vdots & \vdots & \ddots & \vdots & \vdots \\
			 			(1+\frac{2}{\lambda}) & (2+\frac{2}{\lambda}) & \dots & (2+\frac{2}{\lambda}) & (1-\lambda+\frac{2}{\lambda})\\
			 		\end{vmatrix} 
			 		\begin{vmatrix}
			 			(-1-\lambda) & 0 & \dots & 0 & 1  \\
			 			0 & (-1-\lambda) & \dots & 1 & 0\\
			 			\vdots & \vdots & \ddots & \vdots & \vdots \\
			 			1 & 0 & \dots & 0 & (-1-\lambda)\\
			 		\end{vmatrix}\\
			 		& = \begin{vmatrix}
			 			({\lambda}^2-2) & -2(\lambda+1) & -2(\lambda+1) & \dots & \dots & -2(\lambda+1)\\
			 			-2(\lambda+1) & ({\lambda}^2-2) & -2(\lambda+1) & \dots & \dots & -2(\lambda+1)\\
			 			\vdots & \vdots & \vdots & \ddots & \ddots & \vdots \\
			 			-2(\lambda+1) & -2(\lambda+1) & \dots & \dots & \dots & ({\lambda}^2-2)\\
			 		\end{vmatrix} [\because \det(MN) = \det(M) \det(N).]
			 	\end{split}
			 \end{equation*}
			 Now, the equation \eqref{first block-det equation} indicates that 
			 \begin{equation}
			 	\begin{split}
			 		& \det[A(\Gamma({\mathbb{Z}_p}^{(n)})-{\lambda}I_{p^n}] = 
			 		\det \begin{pmatrix}
			 			A & B \\
			 			C & D \\
			 		\end{pmatrix} \\
			 		& = (-\lambda) \begin{vmatrix}
			 			({\lambda}^2-2) & -2(\lambda+1) & -2(\lambda+1) & \dots & -2(\lambda+1)\\
			 			-2(\lambda+1) & ({\lambda}^2-2) & -2(\lambda+1) &  \dots & -2(\lambda+1)\\
			 			\vdots & \vdots & \vdots & \ddots & \vdots\\
			 			-2(\lambda+1) & -2(\lambda+1) & -2(\lambda+1) & \dots & ({\lambda}^2-2)\\
			 		\end{vmatrix}\\
			 		& =(-\lambda) \det[({\lambda}^2-2).I-2.(\lambda+1)(J-I)]\\
			 		& = (-\lambda)\det[({\lambda}^2-2).I-2.(\lambda+1).I-2.(\lambda+1).J]\\
			 		& =(-\lambda) \det[({\lambda}^2 + 2\lambda)I - 2(\lambda+1)J].
			 	\end{split}
			 \end{equation}
			 Here, $I$ is the identity matrix and 
			 $J$ is the matrix with all entries $1$ of order $\frac{p^n-1}{2}$. The spectrum of $({\lambda}^2+2\lambda)I$ are $({\lambda}^2+2\lambda)$ with multiplicity $\frac{p^n-1}{2}$. The spectrum of $2(\lambda+1)J$ are $2(p^n-1)(\lambda+1)$ with multiplicity $1$, and $0$ with multiplicity $\frac{p^n-1}{2}-1 = \frac{p^n-3}{2}$. They are commuting Hermitian matrices. Using \autoref{commuting_matrix_eigenvalues}, we state that the spectrum of $({\lambda}^2 + 2\lambda)I - 2(\lambda+1)J$ are the sum of corresponding spectrum. Simplifying we get
			 $$\det[({\lambda}^2+2.\lambda).I-2.(\lambda+1).J] = {\lambda}^{\frac{p^n-1}{2}}.(\lambda+2)^{\frac{p^n-3}{2}}.({\lambda}^2-(p^n-3){\lambda}-(p^n-1)).$$
			 
			 Combining all the cases, we observe that the spectrum of $A(\Gamma({\mathbb{Z}_p}^{(n)})$ are $\{[0]^{\frac{p^n-1}{2}},[-2]^{\frac{p^n-3}{2}}\}$, and the roots of the equation $x^2 - (p^n-3)x - (p^n-1) = 0$.
			\end{proof}
			
			\begin{corollary}\label{spectra_of_Z_p}
			  The graph $\Gamma(\mathbb{Z}_p)$ has spectrum $\{[0]^{\frac{p-1}{2}}, [-2]^{\frac{p-3}{2}}\}$ and the roots of the equation $x^2-(p-3)x-(p-1)=0$.
			\end{corollary}

	\section{Properties of $\Gamma(\mathbb{Z}_{p^n})$}
	\label{Properties_of_second_group}
	
	\begin{figure}
		\centering
		\begin{tikzpicture}
			\node [above] at (-1, 0.5) {${0}$};
			\draw [fill, black] (-1, 0.5) circle[radius = 1mm];
			\node [left] at (-2.5, -1) {${5}$};
			\draw [fill, black] (-2.5, -1) circle[radius = 1mm];
			\node [right] at (.5, -1) {${10}$};
			\draw [fill, black] (.5, -1) circle[radius = 1mm];
			\node [below] at (-2, -2.5) {${15}$};
			\draw [fill, black] (-2, -2.5) circle[radius = 1mm];
			\node [below] at (0, -2.5) {${20}$};
			\draw [fill, black] (0, -2.5) circle[radius = 1mm];
			\draw [] (-1, 0.5) -- (-2.5, -1);
			\draw [] (-1, 0.5) -- (.5, -1);
			\draw [] (-1, 0.5) -- (-2, -2.5);
			\draw [] (-1, 0.5) -- (0, -2.5);
			\draw [] (-2.5, -1) -- (-2, -2.5);
			\draw [] (-.5, -1) -- (0, -2.5);
			\draw [] (-2.5, -1) -- (.5, -1);
			\draw [] (0, -2.5) -- (.5, -1);
			\draw [] (-2, -2.5) -- (0, -2.5);
			
			\node [left] at (2, -1) {${6}$};
			\draw [fill, black] (2, -1) circle[radius = 1mm];
			\node [left] at (3, 0.5) {${4}$};
			\draw [fill, black] (3, 0.5) circle[radius = 1mm];
			\node [left] at (3, -0.5) {${9}$};
			\draw [fill, black] (3, -0.5) circle[radius = 1mm];
			\node [left] at (3, -1.5) {${14}$};
			\draw [fill, black] (3, -1.5) circle[radius = 1mm];
			\node [left] at (3, -2.5) {${24}$};
			\draw [fill, black] (3, -2.5) circle[radius = 1mm];
			\node [right] at (7, -1) {${19}$};
			\draw [fill, black] (7, -1) circle[radius = 1mm];
			\node [right] at (6, 0.5) {${21}$};
			\draw [fill, black] (6, 0.5) circle[radius = 1mm];
			\node [right] at (6, -0.5) {${16}$};
			\draw [fill, black] (6, -0.5) circle[radius = 1mm];
			\node [right] at (6, -1.5) {${11}$};
			\draw [fill, black] (6, -1.5) circle[radius = 1mm];
			\node [right] at (6, -2.5) {${1}$};
			\draw [fill, black] (6, -2.5) circle[radius = 1mm];
			\draw [] (2,-1) -- (3,-0.5);
			\draw [] (2,-1) -- (3,0.5);
			\draw [] (2,-1) -- (3,-1.5);
			\draw [] (2,-1) -- (3,-2.5);
			\draw [] (7,-1) -- (6,0.5);
			\draw [] (7,-1) -- (6,-0.5);
			\draw [] (7,-1) -- (6,-1.5);
			\draw [] (7,-1) -- (6,-2.5);
			\draw [] (3,0.5) -- (6,-0.5);
			\draw [] (3,0.5) -- (6,-1.5);
			\draw [] (3,0.5) -- (6,-2.5);
			\draw [] (3,-0.5) -- (6,0.5);
			\draw [] (3,-0.5) -- (6,-2.5);
			\draw [] (3,-0.5) -- (6,-1.5);
			\draw [] (3,-1.5) -- (6,0.5);
			\draw [] (3,-1.5) -- (6,-0.5);
			\draw [] (3,-1.5) -- (6,-2.5);
			\draw [] (3,-2.5) -- (6,0.5);
			\draw [] (3,-2.5) -- (6,-0.5);
			\draw [] (3,-2.5) -- (6,-1.5);
			
			\node [left] at (8, -1) {${2}$};
			\draw [fill, black] (8, -1) circle[radius = 1mm];
			\node [right] at (13, -1) {${23}$};
			\draw [fill, black] (13, -1) circle[radius = 1mm];
			\node [left] at (9, 0.5) {${3}$};
			\draw [fill, black] (9, 0.5) circle[radius = 1mm];
			\node [left] at (9, -0.5) {${8}$};
			\draw [fill, black] (9, -0.5) circle[radius = 1mm];
			\node [left] at (9, -1.5) {${13}$};
			\draw [fill, black] (9, -1.5) circle[radius = 1mm];
			\node [left] at (9, -2.5) {${18}$};
			\draw [fill, black] (9, -2.5) circle[radius = 1mm];
			\node [right] at (12, 0.5) {${22}$};
			\draw [fill, black] (12, 0.5) circle[radius = 1mm];
			\node [right] at (12, -0.5) {${17}$};
			\draw [fill, black] (12, -0.5) circle[radius = 1mm];
			\node [right] at (12, -1.5) {${12}$};
			\draw [fill, black] (12, -1.5) circle[radius = 1mm];
			\node [right] at (12, -2.5) {${7}$};
			\draw [fill, black] (12, -2.5) circle[radius = 1mm];
			\draw [] (8,-1) -- (9,0.5);
			\draw [] (8,-1) -- (9,-0.5);
			\draw [] (8,-1) -- (9,-1.5);
			\draw [] (8,-1) -- (9,-2.5);
			\draw [] (13,-1) -- (12,0.5);
			\draw [] (13,-1) -- (12,-0.5);
			\draw [] (13,-1) -- (12,-1.5);
			\draw [] (13,-1) -- (12,-2.5);
			\draw [] (9,0.5) -- (12,-0.5);
			\draw [] (9,0.5) -- (12,-1.5);
			\draw [] (9,0.5) -- (12,-2.5);
			\draw [] (9,-0.5) -- (12,0.5);
			\draw [] (9,-0.5) -- (12,-1.5);
			\draw [] (9,-0.5) -- (12,-2.5);
			\draw [] (9,-1.5) -- (12,0.5);
			\draw [] (9,-1.5) -- (12,-0.5);
			\draw [] (9,-1.5) -- (12,-2.5);
			\draw [] (9,-2.5) -- (12,0.5);
			\draw [] (9,-2.5) -- (12,-0.5);
			\draw [] (9,-2.5) -- (12,-1.5);
		\end{tikzpicture}
		\caption{The POE graph $\Gamma({\mathbb{Z}_{5^2}})$, which has three components. The leftmost component is isomorphic to $\Gamma({\mathbb{Z}})$. The other two components are are isomorphic to each other.}
		\label{Z_p_n_POE_graphs}
	\end{figure}
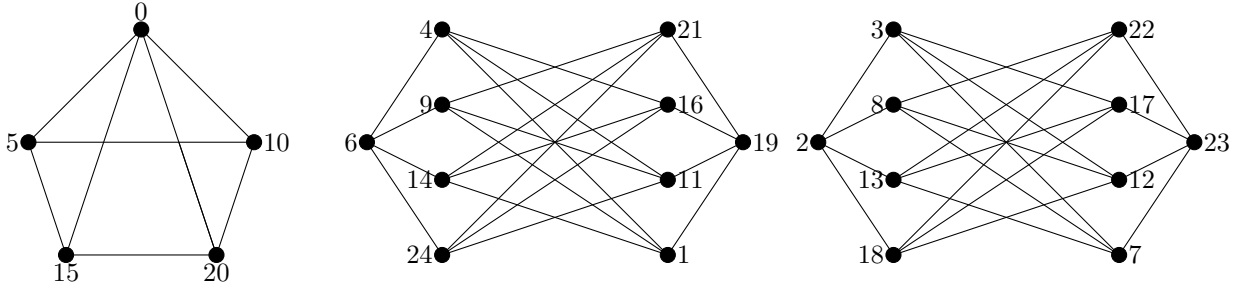
			
	The POE graphs $\Gamma(\mathbb{Z}_{p^n})$ are disconnected when $n \geq 2$. For $n = 1$ we get  $\Gamma(\mathbb{Z}_{p})$, which is a connected graph. Also, we can prove that $\Gamma(\mathbb{Z}_p)$ is always a connected component of $\Gamma(\mathbb{Z}_{p^n})$. An example of the POE graphs of these graphs is drawn in \autoref{Z_p_n_POE_graphs}. Some other significant properties of these are established below.
	
		\begin{lemma}\label{distance-lemma}
			Let $G$ be an Abelian $p$-group with $p > 2$ and with at least two distinct elements of prime orders $y_1$ and $y_2$ with $y_1 y_2 \neq e$. Given any element $x \in G$ the distance between $x$ and $x^{-1}$ in the POE graph $\Gamma(G)$ is
			\begin{equation}
				d(x, x^{-1}) = 
				\begin{cases}
					2, & \text{ if }~ $o(x) = p$ ;\\
					3, & \text{otherwise}.
				\end{cases}
			\end{equation}
		\end{lemma}

		\begin{proof}
			As $G$ is an Abelian $p$-group with $p >2$, there is no element of order $2$ in $G$. Given any element $x \neq e$ in $G$, $x x^{-1} = e$. Hence, $o(x x^{-1}) = 1$. Therefore, there is no edge between $x$ and $x^{-1}$.
			
			Let $o(x)=p$, then $o(x \circ e) = o(x) = p$. Therefore, $x\sim e$. Also, $o(x^{-1})=p$ implies $x^{-1}\sim e$. Thus, there exists a path $x\sim e \sim x^{-1}$ that is $d(x,x^{-1})=2$.
			
			Let $x$ be an element in $G$ such that $o(x) \neq p$. Therefore, $o(x) = p^k$, where $k \geq 2$. Note that, $o(x \circ e) = o(x) = p^k$. Hence, there is no edge between $x$ and $e$. Given that $G$ has at least two distinct elements $y_1$ and $y_2$ with same prime orders, such that, $o(y_1)=o(y_2)=p$ and $y_1y_2\neq e$. Then, $o(y_1y_2)=p$. Also, $x.x^{-1}y_1 = y_1$ and $x^{-1}.xy_2=y_2$. This implies that there always exists a path of length $3$ between $x$ and $x^{-1}$ which is $x\sim x^{-1}y_1\sim xy_2 \sim x^{-1}$, as $G$ is an Abelian group. 
			
			Let there exists a path between $x$ and $x^{-1}$ of length $2$, say $x\sim y\sim x^{-1}$. Note that, $y$ is not the identity element. Under this assumption, $\circ(xy) = \circ(x^{-1}y) = p$. It indicates that, $\circ(y^2) = 1$  or $p$ as $y^2 = yxx^{-1}y$. If $\circ(y^2) = 1$, then $\circ(y) = 1$ or $2$ which is a contradiction as $y \neq e$ and $G$ has no element of order $2$. If $\circ(y^2)=p$, then $\circ(y)=p$, which is again a contradiction because by \autoref{lcm_lemma} we have $\lcm(o(x), o(y)) = p^k \neq p$. Therefore, there is no path between $x$ and $x^{-1}$ whose length is $2$. Hence, $d(x,x^{-1})=3$.
		\end{proof}
	
		It is noteworthy that the \autoref{distance-lemma} indicated for the fact that $x$ and $x^{-1}$ always belong to same connected component of the POE graph $\Gamma(G)$, where $x$ belongs to the Abelian group $G$.
		
		In the group $\mathbb{Z}_{p^n}$, the elements of order $p$ are $p^{n - 1}, 2p^{n - 1}, \dots (p - 1)p^{n - 1}$. Therefore, following  \autoref{distance-lemma}, we state that any element $x$ and $x^{-1}$ always belong to same connected component and $d(x, x^{-1}) \leq 3$. But, $\Gamma(\mathbb{Z}_{p^n})$ is not a connected graph.
		
		\begin{lemma}\label{degree-lemma}
			In the POE graph $\Gamma(\mathbb{Z}_p)$,
			$$\deg(v) = \begin{cases}
				p - 1, & \text{ if }~ v = 0; \\ 
				p - 2, & \text{otherwise}.
			\end{cases}$$
		\end{lemma}
		\begin{proof}
			The additive group $\mathbb{Z}_p$ contains $p-1$ elements of prime order, which we collect in a set $S = \{x_1, x_2, \dots x_{p-1}\}$. Clearly, $\mathbb{Z}_p=S\cup \{e\}$. Note that, the identity $e$ is adjacent to all the elements of $S$ that is $\deg(e) = p-1$. For any two non-identity elements $x_i,x_j\in S$, $o(x_ix_j)=p$, where $x_ix_j\neq e$. For $i=j,o(x_i^2)=p$. Since, we define the POE graph as a simple graph we do not consider the self-loop in it. Thus, an element $x_i$ of order $p$ from $S$ is adjacent to all the elements except $x_i^{-1}$ and $x_i$ itself. Hence, $\deg(x_i)=p-2$.
		\end{proof}
		
		\begin{lemma}\label{same-order-lemma}
			In an Abelian $p$-group $G$, if two non-identity elements $x$ and $y$ are adjacent in the POE graph $\Gamma(G)$ then $o(x) = o(y)$.
		\end{lemma}
		\begin{proof}
			As $G$ is an Abelian $p$-group, $o(G) = p^n$. As $x, y \in G$, $o(x)$ and $o(y)$ divides $p^n$. Thus, $o(x) = p^a$ and $o(y) = p^b$, where we assume $a < b$ for simplicity. It is given that $x\sim y$ in $\Gamma(G)$. Therefore, $o(xy) = p$ or $(xy)^p = e$ or $x^p=(y^{-1})^p$ or $(x^p)^{p^{a-1}}=((y^{-1})^p)^{p^{a-1}}$ or $x^{p^a}=(y^{-1})^{p^a}$. Hence, $(y^{-1})^{p^a}=e$ or $o(y^{-1}) = p^a = o(y)$; which is a contradiction. Consequently, $\circ(x)=\circ(y)$.
		\end{proof}
		
		The above Lemma indicates that in a connected component of $\Gamma(G)$ all the elements have equal order. But it does not indicate that all the elements of equal order belong to same connected component. For example, in \autoref{Z_p_n_POE_graphs} all the elements of order $25$ are distributed into two connected components.
		
		Following the above Lemma, in the POE graph $\Gamma(\mathbb{Z}_{p^n})$, if $x\sim y$, then $o(x) = o(y)$.
		
		\begin{lemma}\label{connected-lemma}
			The graph $\Gamma(\mathbb{Z}_{p^n})$ is not a connected graph.
		\end{lemma} 
		
		\begin{proof}
			The group $\mathbb{Z}_{p^n}$ contains the elements of order $1,p,p^2,p^3, \dots p^n$. In the graph $\Gamma(\mathbb{Z}_{p^n})$, the identity is always adjacent to only the elements of order $p$. Also, the elements of order $p$ is not adjacent to any element of order $p^k,(k\geq 2)$ since for any two elements $x,y$ in an Abelian group, $o(xy) = \lcm(o(x),o(y))$, using \autoref{lcm_lemma}. Clearly, identity with the elements of order $p$ makes a separate connected component from the elements of order $p^k$ with $k \geq 2$ in the graph $\Gamma(\mathbb{Z}_{p^n})$. Thus, the graph becomes disconnected. 
		\end{proof}
		
		The proof of \autoref{connected-lemma} indicates that a connected component of $\Gamma(\mathbb{Z}_{p^n})$ is isomorphic to $\Gamma(\mathbb{Z}_p)$. Also, \autoref{same-order-lemma} indicates that a connected component can contain the elements of equal order. We denote the connected components of elements of different orders as $C^{(1)}, C^{(2)}, \dots$, which are not isomorphic to $\Gamma(\mathbb{Z}_p)$. Therefore $\Gamma(\mathbb{Z}_{p^n}) = \Gamma(\mathbb{Z}_p) \cup_{i = 1}^{\frac{p^{n - 1} - 1}{2}} C^{(i)}$.
		
		\begin{lemma}\label{degree-lemma-2}
			For any element $x\in \mathbb{Z}_{p^n}$, for $o(x)=p^k$ with $k\geq 2$, the degree of the vertex $x$ in $\Gamma(\mathbb{Z}_{p^n})$ is $p-1$.
		\end{lemma}
		\begin{proof}
			Let $x\in C^{(l)}$ where $C^{(l)}$ be any connected component of elements of order $p^k$ with $k\geq 2$. Also, $S=\{x_1,x_2,\dots,x_{p-1}\}$ is the set of all elements of order $p$ in the group $\mathbb{Z}_{p^n}$. Thus, $o(x)=p^k$ and $o(x^2)=p^k$. If $x\sim y$, then $o(xy) = p$. We obtain, $xy=x_i$ for any one of  $i = 1, 2, \dots, (p - 1)$. Since, for every element $x_i$ there exists a unique $y$ such that $xy=x_i$, then the degree of each element $x$ of order $p^k(k\geq 2)$ is $p-1$.
		\end{proof}
		
		\begin{lemma}\label{triangle-lemma}
			Let $C^{(l)}$ be a connected component of elements of order $p^k$ with $k\geq 2$, then there is no $3$-cycle in $C^{(l)}$.
		\end{lemma}
		\begin{proof}
			Let there be a $3$-cycle in $C^{(l)}$ which is $x_1\sim x_2\sim x_3\sim x_1$. Then, $\circ(x_1x_2)=\circ(x_2x_3)=\circ(x_3x_1)=p$ that is $\circ(x_1x_2x_2^{-1}x_3^{-1}x_3x_1) = o(x_1^2) =1$ or  $p$. It indicates $\circ(x_1)=1 \mbox{ or } 2 \mbox{ or } p$ which is a contradiction. Therefore, there is no $3$-cycle in $C^{(l)}$.
		\end{proof}
		
		\begin{lemma}\label{adjacency-lemma}
			For $x,y_1,y_2 \in C^{(l)}$, if $x \sim y_1$, $x\sim y_2$, $x^{-1}\sim y_1^{-1}$, and $x^{-1}\sim y_2^{-1}$ hold then $y_1\sim y_2^{-1}$.
		\end{lemma}
		\begin{proof}
			As $x \sim y_1$ and $x\sim y_2$ we have $\circ(xy_1)=p$ and $\circ(xy_2)=p$. Also, $\circ(y_1x)=p$ and $\circ(x^{-1}y_2^{-1})=p$ that is $\circ(y_1xx^{-1}y_2^{-1}) = \circ(y_1y_2^{-1}) =1 \mbox{ or } p$. If $\circ(y_1y_2^{-1})=1$, then $y_1=y_2$, which is a contradiction. Hence, $\circ(y_1y_2^{-1})=p$. Consequently, we have, $y_1\sim y_2^{-1}$.
		\end{proof}
		\begin{lemma}\label{bipartite-lemma}
			Each connected component $C^{(l)}$ is bipartite.
		\end{lemma}
		\begin{proof}
			Let $A_l$ contains an odd cycle of length $m$ such that $x_1\sim x_2\sim \dots \sim x_m$. Thus, $o(x_1x_2)=p,o(x_2x_3)=p,\dots , o(x_mx_1)=p$. This implies that
			\begin{equation}
				o(x_1x_2x_2^{-1}x_3^{-1}x_3x_4x_4^{-1}x_5^{-1}
					\dots x_{m-1}^{-1}x_m^{-1}x_mx_1)= 1 \mbox{ or } p.
			\end{equation}
			It implies, $o(x_1^2) = 1$ or $p$, that is $o(x_1)=1$, or $2$, or $p$, which leads us to a contradiction.
			Therefore, the connected component $C^{(l)}$ contains no odd cycle. Any graph having no odd cycle is bipartite.
		\end{proof}
		
			\begin{thm}\label{structure theorem 1}
			The POE Graph $\Gamma(\mathbb{Z}_{p^n})$ is the union of an isomorphic copy of $\Gamma(\mathbb{Z}_p)$ and $\frac{p^n-p}{2p}$ number of $(p-1)$-regular bipartite graphs each containing $2p$ number of vertices, where $p$ in an odd prime and $n > 2$.
		\end{thm}
		\begin{proof}
			Let $S=\{x_1,x_2,\dots x_{p-1}\}$ be the set of all elements of order $p$ in $\mathbb{Z}_{p^n}$. 
			
			Now, by \autoref{connected-lemma} and \autoref{same-order-lemma}, the graph $\Gamma(\mathbb{Z}_{p^n})$ is disconnected and it is the union of $\Gamma(\mathbb{Z}_{p})$ and $C^{(1)},C^{(2)},\dots$ where $C^{(l)}$'s are the connected components of the elements of equal order, which is  $p^k$ for some $k\geq 2$. But, \autoref{same-order-lemma} does not claim that all the elements of order $p^k$ is in one connected component.  Now, from \autoref{degree-lemma-2} we say that the degree of each element of order $p^k$ with $k\geq 2$ is $p-1$. Therefore, the connected components $C^{(1)},C^{(2)},\dots$ are $(p-1)$-regular. The connected component isomorphic to $\Gamma(\mathbb{Z}_p)$ contains $p$ elements. Thus, the other components contains $p^n-p$ number of elements in total. \autoref{degree-lemma} indicates that the elements $x$ and $x^{-1}$ belongs to the same connected components. Let $x\in C^{(l)}$ be an element and $o(x)=p^k$ for some $k\geq 2$. \autoref{degree-lemma-2} suggests that there exist $p-1$ elements adjacent to $x$. Let $x\sim y_i$ for $i=1,2,\dots,p-1$, where $xy_i=x_i$. Also, $x^{-1}\sim y_i^{-1}$ for $i=1,2,\dots p-1$ and $x^{-1}y_i^{-1}=x_i$. It is proved in \autoref{triangle-lemma} that there is no 3-cycle in $C^{(l)}$. Hence, $y_i\nsim y_j$ and $y_i^{-1}\nsim y_j^{-1}$ for $i\neq j$. By \autoref{adjacency-lemma}, $y_i\sim y_j^{-1}$, for $i\neq j$. Combining all these, we write the adjacency relations in $\Gamma(\mathbb{Z}_{p^n})$ as follows:
			$$ \begin{cases}
				x\sim y_i, i=1,2,\dots, p-1 ;\\
				x^{-1}\sim y_i^{-1}, i=1,2,\dots,p-1 ;\\
				y_i\sim x \mbox{ and } y_i\sim y_j^{-1}, i\neq j ;\\
				y_i^{-1}\sim x^{-1} \mbox{ and } y_i^{-1}\sim y_j^{-1}, i\neq j.\\
			\end{cases}$$
			From the above relation between elements, it is clear that the connected component $C^{(l)}$ contains $x,x^{-1},y_i,y_i^{-1}$ for $i=1,2,\dots p-1$. Therefore, every component $C^{(l)}$ contains $2p$ elements. Also, it is proved in \autoref{bipartite-lemma}, that each connected component $C^{(l)}$ is bipartite. Combining all these we get the required result. 
		\end{proof}

		\begin{thm}\label{2nd spectral theorem}
			The spectrum of $\Gamma(\mathbb{Z}_{p^n})$ are $\{[0]^{\frac{p-1}{2}}, [-2]^{\frac{p-3}{2}}, [p-1]^{\frac{p^n-p}{2p}}, [1-p]^{\frac{p^n-p}{2p}}, [1]^{(p-1)\frac{(p^n-p)}{2p}}$, $[-1]^{(p-1)\frac{(p^n-p)}{2p}}\}$ and rest of two irrational spectrum are the roots of the equation $x^2-(p-3)x-(p-1)=0$.  
		\end{thm}

		\begin{proof}
			\autoref{connected-lemma} indicates that the graph $\Gamma(\mathbb{Z}_{p^n})$ is the union of $\Gamma(\mathbb{Z})$ and $\frac{p^n-p}{2p}$ regular connected components $C^{(l)}$ with regularity $(p - 1)$. Each $C^{(l)}$ consists of $2p$ vertices. The spectrum of $\Gamma(\mathbb{Z}_{p^n})$ is the union of the spectrum of all the connected components. The spectrum of $\Gamma(\mathbb{Z}_p)$ can be determined by \autoref{First spectral theorem}.
			
			Let $C^{(l)}$ be a regular connected component consists of $2p$ vertices and regularity $p-1$. \autoref{structure theorem 1} proves that all the connected component $C^{(l)}$ are bipartite graphs with $2p$ vertices of degree $p-1$. The vertices of $C^{(l)}$ can be expressed as $V_1\cup V_2$, where $V_1=\{x_i: i=1,2,\cdots,p\}$ and $V_2=\{x_i^{-1}: i=1,2,\cdots,p\}$. The adjacency matrix for the component $C^{(l)}$ is $A(C^{(l)}) = (a_{i,j}^{(l)})_{2p \times 2p}$, where
			\begin{equation}
				a_{i,j}^{(l)} = 
				\begin{cases}
					1,  & \mbox{ if } i=1,2,\cdots,p;  j=p+1,\cdots,2p; \mbox{ and } i\neq j ; \\
					1,  & \mbox{ if }i=p+1,p+2,\cdots,2p; \mbox{ and } j=1,2,\cdots,p ; \\
					0, & \mbox{otherwise.}
				\end{cases}
			\end{equation}
			The Characteristic polynomial of $C^{(l)}$ is\\
			\begin{center}
				$|A(C^{(l)})-{\lambda}I| = \det\begin{bmatrix}
					-\lambda & 0 & \dots & 0 & 0 & \rvline & 0 & 1 & \dots & 1 & 1\\
					0 & -\lambda & \dots & 0 & 0 & \rvline & 1 & 0  & \dots & 1 & 1\\
					\vdots & \vdots & \ddots & \vdots & \vdots & \rvline & \vdots & \vdots & \ddots & \vdots & \vdots\\
					0 & 0 & \dots & -\lambda & 0 & \rvline & 1 & 1 & \dots & 0 & 1\\
					0 & 0 & \dots & 0 & -\lambda & \rvline & 1 & 1  & \dots & 1 & 0\\
					\hline
					0 & 1  & \dots & 1 & 1 & \rvline & -\lambda & 0 & \dots & 0 & 0\\
					1 & 0  & \dots & 1 & 1 & \rvline & 0 & -\lambda & \dots & 0 & 0\\
					\vdots & \vdots & \ddots & \vdots & \vdots & \rvline & \vdots & \vdots & \ddots & \vdots & \vdots\\
					1 & 1 & \dots & 0 & 1 & \rvline & 0 & 0 & \dots & -\lambda & 0\\
					1 & 1 & \dots & 1 & 0 & \rvline & 0 & 0 & \dots & 0 & -\lambda\\
				\end{bmatrix}$.
			\end{center}
			Therefore, $|A(C^{(l)})-{\lambda}I|$ can be expressed as $\det \begin{bmatrix}
				A & B\\
				B & A
			\end{bmatrix}$
			where $A = -\lambda I$ and $B = J - I$ are the commuting matrices of order $p$.

			Thus, 
			\begin{equation}
				|A_i-{\lambda}I|= |A+B|\cdot |A-B|=\begin{vmatrix}
					-\lambda & 1 & \dots & 1 & 1\\
					1 & -\lambda & \dots & 1 & 1\\
					\vdots & \vdots & \ddots & \vdots & \vdots \\
					1 & 1 & \dots & -\lambda & 1 \\
					1 & 1 & \dots & 1 & -\lambda 
				\end{vmatrix} \cdot \begin{vmatrix}
					-\lambda & -1 & \dots & -1 & -1\\
					-1 & -\lambda & \dots & -1 & -1\\
					\vdots & \vdots & \ddots & \vdots & \vdots \\
					-1 & -1 & \dots & -\lambda & -1 \\
					-1 & -1 & \dots & -1 & -\lambda  
				\end{vmatrix}.
			\end{equation}
			Note that, the set of spectrum of $A$ and $B$ are $\Lambda(A) = \{p,0,0,\cdots, 0(p-1 \mbox{ times })\}$, and $\Lambda(B) = \{(-\lambda-1),(-\lambda-1),\cdots, (-\lambda-1) (p \mbox{ times })\}$. Applying \autoref{commuting_matrix_eigenvalues} we have the set of spectrum of $A + B = J - (\lambda + 1) I$.

			which is $\Lambda(A + B) = \{(-\lambda+p-1),(-\lambda-1),(-\lambda-1),\cdots,(-\lambda-1)\}$. Similarly, the set of spectrum of $A-B$ is $\Lambda(A - B) = \{(-\lambda-p+1),(-\lambda+1),(-\lambda+1),\cdots,(-\lambda+1)\}$. Therefore, the spectrum of $A_i$ are $(p-1),(1-p), 1(p-1 \mbox{ times }),-1(p-1 \mbox{ times })$. 
			
			As a component of $\Gamma(\mathbb{Z}_{p^n})$ is isomorphic to $\Gamma(\mathbb{Z}_{p})$, applying \autoref{spectra_of_Z_p} we say that the graph $\Gamma(\mathbb{Z}_{p^n})$ has spectrum $0$ with multiplicity $\frac{p-1}{2}$, $-2$ with multiplicity $\frac{p-3}{2}$, and the roots of $x^2-(p-3)x-(p-1)$. The remaining spectrum of $\Gamma(\mathbb{Z}_{p^n})$ are determined by \autoref{structure theorem 1} which comes from the spectrum of $A(C^{(l)})$. In combination, the other spectrum of $\Gamma(\mathbb{Z}_{p^n})$ are $\{[0]^{\frac{p-1}{2}}, [-2]^{\frac{p-3}{2}}, [p-1]^{\frac{p^n-p}{2p}}, [1-p]^{\frac{p^n-p}{2p}}, [1]^{(p-1)\frac{(p^n-p)}{2p}}$, $[-1]^{(p-1)\frac{(p^n-p)}{2p}}\}$, and the remaining spectrum are the roots of the equation $x^2-(p-3)x-(p-1)=0$. Combining all these, we get the required result.
		\end{proof}

	\section{Properties of $\Gamma(\mathbb{Z}_{2^n})$}
	\label{Properties_of_third_group}

		The group $\mathbb{Z}_{2^n} = \{0, 1, 2, \dots (2^n - 1)\}$ contains only one element of order $2$, which is $2^{n-1}$. Also, $\mathbb{Z}_{2^n}$ is a cyclic group of order $2^n$. The POE graphs corresponding this group are unions of $K_2$ graphs and isolated vertices. Two examples of them are depicted in \autoref{Z_2_n_POE_graphs}. Now, we have the following properties:
		
		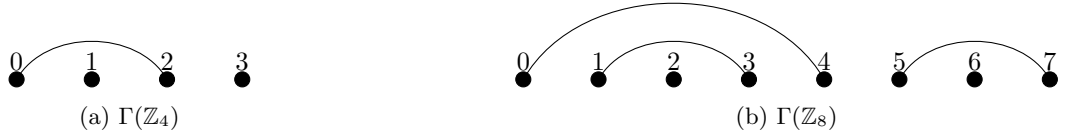
\begin{figure}
			\centering
			\begin{subfigure}[b]{.49\textwidth}
				\centering
				\begin{tikzpicture}
					\node [above] at (0, 0) {${0}$};
					\draw [fill, black] (0, 0) circle[radius = 1mm];
					\node [above] at (1, 0) {${1}$};
					\draw [fill, black] (1, 0) circle[radius = 1mm];
					\node [above] at (2, 0) {${2}$};
					\draw [fill, black] (2, 0) circle[radius = 1mm];
					\node [above] at (3, 0) {${3}$};
					\draw [fill, black] (3, 0) circle[radius = 1mm];
					\path [-] (0,0) edge[bend left=60] node {} (2,0);
				\end{tikzpicture}
				\caption{$\Gamma({\mathbb{Z}_4})$}
			\end{subfigure} 
			\begin{subfigure}[b]{.49\textwidth}
				\centering
				\begin{tikzpicture}
					\node [above] at (5, 0) {${0}$};
					\draw [fill, black] (5, 0) circle[radius = 1mm];
					\node [above] at (6, 0) {${1}$};
					\draw [fill, black] (6, 0) circle[radius = 1mm];
					\node [above] at (7, 0) {${2}$};
					\draw [fill, black] (7, 0) circle[radius = 1mm];
					\node [above] at (8, 0) {${3}$};
					\draw [fill, black] (8, 0) circle[radius = 1mm];
					\node [above] at (9, 0) {${4}$};
					\draw [fill, black] (9, 0) circle[radius = 1mm];
					\node [above] at (10, 0) {${5}$};
					\draw [fill, black] (10, 0) circle[radius = 1mm];
					\node [above] at (11, 0) {${6}$};
					\draw [fill, black] (11, 0) circle[radius = 1mm];
					\node [above] at (12, 0) {${7}$};
					\draw [fill, black] (12, 0) circle[radius = 1mm];
					\path [-] (5,0) edge[bend left=60] node {} (9,0);
					\path [-] (6,0) edge[bend left=60] node {} (8,0);
					\path [-] (10,0) edge[bend left=60] node {} (12,0);
				\end{tikzpicture}
				\caption{$\Gamma({\mathbb{Z}_8})$}
			\end{subfigure} 
			\caption{POE graphs associated to $\mathbb{Z}_4$ and $\mathbb{Z}_8$. They are union of $K_2$ graphs and isolated vertices.}
			\label{Z_2_n_POE_graphs}
		\end{figure}
		
		\begin{lemma}{\label{order 2^n}}
			The graph $\Gamma(\mathbb{Z}_{2^n})$ is the union of $\frac{2^n-2}{2}$ different $K_2$ graphs and two isolated vertices corresponding to the elements of order $4$.
		\end{lemma}
		
		\begin{proof}
			The group $\mathbb{Z}_{2^n}$ has only one element of order $2$ which is $2^{n-1}$. There is no other element in $\mathbb{Z}_{2^n}$ of prime order. Also, there are two items of order $4$ in $\mathbb{Z}_{2^n}$, which are $2^{n - 2}$ and $2^2$. 
			
			If $x$ be an element of order $4$ then there is no element $y \neq x$ such that $xy = 2^{n - 1}$. Hence, $x$ is not adjacent to any element. Therefore, there are two isolated vertices correspond to the two items of order $4$.
			
			As $2^{n-1}$ is its self-inverse, it is adjacent to the identity $e$. No other element is adjacent to $2^{n-1}$ or $e$. Now, the composition of any two elements $x, y\in G$ is prime if and only if  $xy=2^{n-1}$. For each $x\in G$, there exists $y\in G$ such that $xy=2^{n-1}$ or $y=2^{n-1}x^{-1}$. Since, the inverse of each element is unique, then for every element $x \neq 2^{n-1} \in G$, there exists a unique $y$ such that $xy=2^{n-1}$. Therefore, there is an edge between $x$ and $y$. Also, they are not adjacent to any other element. Eventually, there are $\frac{2^n-2}{2} = 2^{n-1}-1$ pairs of $x$ and $y$ which forms $(2^{n-1}-1)$ $K_2$ graphs. Hence, we have the result.
		\end{proof}
		
		We can rearrange the isolated vertices and the vertices corresponding to the other elements of the group as follows:
		$$x_1, x_1^{-1}, x_2, x_2^{-1},x_3,x_3^{-1},\dots,x_{2^{n-1}-1},x_{2^{n-1}-1}^{-1},2^{n-2},2^2.$$
		It produces the adjacency matrix of $\Gamma(\mathbb{Z}_{2^n})$ in the form $A({\Gamma}(\mathbb{Z}_{2^n})) = (a_{i, j})_{2^n \times 2^n}$, where
		\begin{equation}
			a_{i, j} = 
			\begin{cases}
				1, & \text{ for } i = j+1, \text{or}~ j=i+1, ~\text{and}~ i, j \notin \{2^{n}-1, 2^n\};\\
				0, & \text{ otherwise.}\\
			\end{cases}
		\end{equation}
		Now we have the following result regarding the spectra of $\Gamma(\mathbb{Z}_{2^n})$.
		\begin{thm}
			If $G\cong \mathbb{Z}_{2^n}$ then the spectrum of $A(\Gamma(G))$ is $\{[0]^2,[1]^{2^{n-1}-1},[-1]^{2^{n-1}-1}\}$. 
		\end{thm}
		
		\begin{proof}
			The POE graph $A(G)$ is an union of $\frac{2^n-2}{2}$ different $K_2$ graphs and two isolated vertices. Equation \eqref{block_diagonal_spectra} indicates that the spectra of $A({\Gamma}(G))$ is the union of the spectrum of $K_2$ graphs with multiplicity $\frac{2^n-2}{2}$ and $0$ with multiplicity $2$. Thus, the spectrum of $A(\Gamma(G))$ are$\{[0]^2,[1]^{2^{n-1}-1},[-1]^{2^{n-1}-1}\}$.
		\end{proof}	
%
%

\section{Properties of ${\Gamma(\mathbb{Z}_{2^{n_1}} \times \mathbb{Z}_{2^{n_2}}\times \dots \times \mathbb{Z}_{2^{n_k}}})$}
\label{Properties_of_fourth_group}

The group $\mathbb{Z}_{2^{n_1}} \times \mathbb{Z}_{2^{n_2}}\times \dots \times \mathbb{Z}_{2^{n_k}}$ is direct product of the groups $\mathbb{Z}_{2^{n}}$ for different values of $n$. The POE graphs $\Gamma({\mathbb{Z}_{2^{n_1}} \times \mathbb{Z}_{2^{n_2}}\times \dots \times \mathbb{Z}_{2^{n_k}}})$ are disconnected when at least one of the $n_i$ are different from $1$. A POE graph of this class is depicted in \autoref{Z_2_n_i_product_POE_Graph}. We consider the following properties of these graphs:

 \begin{figure}
	\centering
	\begin{tikzpicture}
		\node [above] at (0, 0) {${(0,0)}$};
		\draw [fill, black] (0, 0) circle[radius = 1mm];
		\node [below] at (0, -2) {${(1,0)}$};
		\draw [fill, black] (0, -2) circle[radius = 1mm];
		\node [above] at (2, 0) {${(0,4)}$};
		\draw [fill, black] (2, 0) circle[radius = 1mm];
		\node [below] at (2, -2) {${(2,4)}$};
		\draw [fill, black] (2, -2) circle[radius = 1mm];
		
		\draw [] (0,0) -- (2,-2);
		\draw [] (0,0) -- (0,-2);
		\draw [] (0,0) -- (2,0);
		\draw [] (2,0) -- (0,-2);
		\draw [] (2,0) -- (2,-2);
		\draw [] (0,-2) -- (2,-2);
		
		\node [above] at (4, 0) {${(1,2)}$};
		\draw [fill, black] (4, 0) circle[radius = 1mm];
		\node [above] at (6, 0) {${(0,6)}$};
		\draw [fill, black] (6, 0) circle[radius = 1mm];
		\node [below] at (4, -2) {${(0,2)}$};
		\draw [fill, black] (4, -2) circle[radius = 1mm];
		\node [below] at (6, -2) {${(1,6)}$};
		\draw [fill, black] (6, -2) circle[radius = 1mm];
		\draw [] (4,0) -- (6,0);
		\draw [] (4,-2) -- (6,-2);
		\draw [] (4,0) -- (4,-2);
		\draw [] (6,0) -- (6,-2);	
		
		\node [left] at (8, -1) {${(1,7)}$};
		\draw [fill, black] (8, -1) circle[radius = 1mm];
		\node [above left] at (10, -1) {${(0,1)}$};
		\draw [fill, black] (10, -1) circle[radius = 1mm];
		\node [left] at (10, 0) {${(0,3)}$};
		\draw [fill, black] (10, 0) circle[radius = 1mm];
		\node [left] at (10, -2) {${(1,3)}$};
		\draw [fill, black] (10, -2) circle[radius = 1mm];
		\node [right] at (12, 0) {${(0,5)}$};
		\draw [fill, black] (12, 0) circle[radius = 1mm];
		\node [above right] at (12, -1) {${(0,7)}$};
		\draw [fill, black] (12, -1) circle[radius = 1mm];
		\node [right] at (12, -2) {${(1,5)}$};
		\draw [fill, black] (12, -2) circle[radius = 1mm];
		\node [right] at (14, -1) {${(1,1)}$};
		\draw [fill, black] (14, -1) circle[radius = 1mm];
		\draw [] (8,-1) -- (10,-1);
		\draw [] (10,0) -- (10, -1);
		\draw [] (10, -2) -- (10, -1);
		\draw [] (12,0) -- (12,-1);
		\draw [] (12,-2) -- (12,-1);
		\draw [] (12,-1) -- (14,-1);
		\draw [] (10,0) -- (12,-2);
		\draw [] (12,0) -- (10,-2);
		
		\path [-] (8,-1) edge[bend left=60] node {} (12,0);
		\path [-] (8,-1) edge[bend right=60] node {} (12,-2);
		\path [-] (14,-1) edge[bend right=60] node {} (10,0);
		\path [-] (14,-1) edge[bend left=60] node {} (10,-2);
	\end{tikzpicture}
	\caption{The POE graph $\Gamma({\mathbb{Z}_2} \times {\mathbb{Z}_8})$. The graph is disconnected. Note that none of the connected components are isomorphic to $\Gamma(\mathbb{Z}_2)$ or $\Gamma(\mathbb{Z}_2)$, which are depicted in \autoref{Z_2_n_POE_graphs}.}
	\label{Z_2_n_i_product_POE_Graph}
\end{figure}
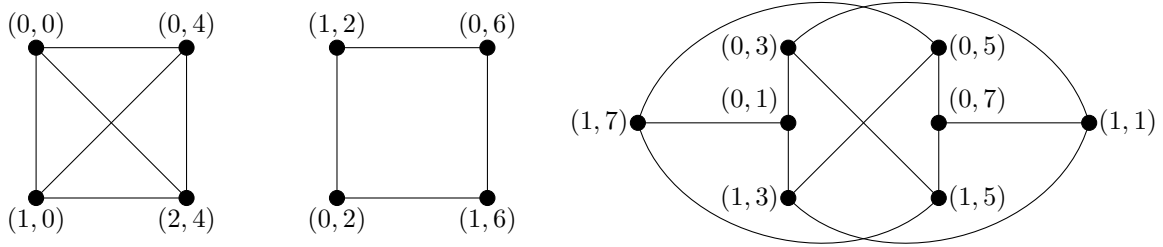

\begin{lemma}{\label{Zero 2-group}}
	The graph $\Gamma({\mathbb{Z}_{2^{n_1}} \times \mathbb{Z}_{2^{n_2}}\times \dots \times \mathbb{Z}_{2^{n_k}}})$ is a disconnected graph and one connected component is isomorphic to $\Gamma({\mathbb{Z}_2^k})$.
\end{lemma}
\begin{proof}
	Recall that the group $\mathbb{Z}_{2^n}$ has only one element of order $2$ which is $2^{n-1}$. Therefore, the group $\mathbb{Z}_{2^{n_1}} \times \mathbb{Z}_{2^{n_2}}\times$ $\dots \times \mathbb{Z}_{2^{n_k}}$ contains $(2^{k} - 1)$ elements of order $2$ which are of the form $(\varepsilon_1 2^{n_1-1},\dots,\varepsilon_k 2^{n_k-1})$, where $\varepsilon_i\in\{0,1\}$. Let $C^{(0)}$ be the component containing the identity element. Since, the identity element is adjacent to the elements of prime order, then all these elements belong to $C^{(0)}$. Using \autoref{lcm_lemma} we conclude that no element of higher order belong to $C^{(0)}$. Define a mapping $\varphi: C^{(0)}\longrightarrow {\mathbb{Z}_2^k}$ by $\varphi\big((\varepsilon_1 2^{n_1-1},\dots,\varepsilon_k 2^{n_k-1})\big)
	=
	(\varepsilon_1,\dots,\varepsilon_k)$. One can easily prove that $\varphi$ is a bijective function and hence $C^{(0)}\cong {\mathbb{Z}_2^k}$.
\end{proof}

\begin{lemma}{\label{First 2-group}}
	The number of elements of order $2^t$ for $t\geq 2$ in the group ${\mathbb{Z}_{2^{n_1}} \times \mathbb{Z}_{2^{n_2}}\times \dots \times \mathbb{Z}_{2^{n_k}}}$ is $2^c - 2^d$, where $c = {\sum_{i=1}^{k}}min\{n_i,2^t\}$ and $d = {\sum_{i=1}^{k}}min\{n_i,2^{t-1}\}$.
\end{lemma}

\begin{proof}
	 Let $x=(x_1,x_2,\dots, x_k)$ be an element of order $2^t$ for $t\geq 2$ in the group ${\mathbb{Z}_{2^{n_1}} \times \mathbb{Z}_{2^{n_2}}\times \dots \times \mathbb{Z}_{2^{n_k}}}$. Now, one can observe that  $o(x)=2^t(t\geq 2)$ if and only if $o(x_i)\leq 2^t (i=1,2,\dots, k)$ and for at least one $x_i, o(x_i)=2^t$. In the group $\mathbb{Z}_{2^{n_i}}$, 
	\begin{enumerate}
		\item if $n_i=t$, all elements already have order $\leq 2^t(t\geq 2)$. 
		\item if $n_i\geq t$, the elements of order $\leq 2^t$ that are the elements of order either $1,2, 2^2, \dots $ or $2^t$ that is the total number of elements in a subgroup of order $2^t$ of the group $\mathbb{Z}_{2^{n_i}}$.
	\end{enumerate} 
	
	Therefore, the number of choices for each co-ordinate of $x=(x_1,x_2,\dots , x_k)$ is $$
	\begin{cases}
		2^{n_i}, & n_i\leq t, \\
		2^t, & n_i \geq t.
	\end{cases}
	$$
	Hence, the total number of elements of order $2^t(t\geq 2)$ in ${\mathbb{Z}_{2^{n_1}} \times \mathbb{Z}_{2^{n_2}}\times \dots \times \mathbb{Z}_{2^{n_k}}}$ is $$\left (\prod_{i=1}^{k} \left (\begin{cases}
		2^{n_i}, & n_i \leq t, \\
		2^t, & n_i \geq t.
	\end{cases}\right )\right )- \left (\prod_{i=1}^{k} \left (\begin{cases}
	2^{n_i}, & n_i \leq t-1, \\
	2^{t-1}, & n_i \geq t-1.
	\end{cases}\right )\right ) = 2^{\sum_{i=1}^{k} min(n_i,t)} -2^{\sum_{i=1}^{k} min(n_i,t-1)}.$$ 
\end{proof}

\begin{corollary}{\label{Second 2-group}}
	The number of elements of order $4$ is ${2^{\sum_{i=1}^{k}min\{n_i,2\}}-2^k}$.
\end{corollary}

\begin{lemma}{\label{Third 2-group}}
	The number of connected component containing the elements of order $4$ is $\frac{2^{\sum_{i=1}^{k}min\{n_i,2\}}-2^k}{2^k}$, where ${2^{\sum_{i=1}^{k}min\{n_i,2\}}-2^k}$ is the total number of elements of order $4$ in the group.
\end{lemma}

\begin{proof}
	From the above \autoref{Second 2-group}, the number of elements of order $4$ in the group is ${2^{\sum_{i=1}^{k}min\{n_i,2\}}-2^k}$. Also, the number of order $2$ elements in the group is $2^k-1$. Let $x$ be an element of order $4$ in the group. Then, $$x^4=e \implies (x^2)^2=e.$$ This implies that, $x\circ x$ gives an element of order $2$. Clearly, the degree of each element of order $4$ is $2^{k}-2$. Since in an Abelian p-group, $o(xy)=l.c.m\{o(x),o(y)\}$, the elements of order $4$ can be adjacent with only the element of order $4$. Now, using the same method as \autoref{structure theorem 1}, each connected component of elements of order $4$ contains $2^k$ number of elements. Hence the number of connected components of elements of order $4$ in the graph $\Gamma({\mathbb{Z}_{2^{n_1}} \times \mathbb{Z}_{2^{n_2}}\times \dots \times \mathbb{Z}_{2^{n_k}}})$ is $\frac{2^{\sum_{i=1}^{k}min\{n_i,2\}}-2^k}{2^k}$.
 \end{proof}
 
 \begin{lemma}{\label{Forth 2-group}}
 	The number of connected component containing the elements of order $2^t(t\geq 3)$ is $\frac{2^{\sum_{i=1}^{k}}min\{n_i,2^t\}-2^{\sum_{i=1}^{k}}min\{n_i,2^{t-1}\}}{2^{k+1}}$ where $2^{\sum_{i=1}^{k}}min\{n_i,2^t\}-2^{\sum_{i=1}^{k}}min\{n_i,2^{t-1}\}$ is the total number of elements of order $2^t(t\geq 3)$ in the group.
 \end{lemma}
 
 \begin{proof}
 	Following the \autoref{First 2-group}, the number of elements of order $2^t(t\geq 3)$ is $2^{\sum_{i=1}^{k}}min\{n_i,2^t\}-2^{\sum_{i=1}^{k}}min\{n_i,2^{t-1}\}$. If $x$ be an element in the group such that $o(x)=2^t(t\geq 3)$, clearly, $o(x^2)\neq 2$. Hence the degree of each element of order $2^t(t\geq 3)$ is $2^k-1$. Since in an Abelian p-group, $o(xy)=l.c.m.\{o(x),o(y)\}$, every connected component contains the elements of same order. Now, by using the same method as \autoref{structure theorem 1}, each connected component of elements of order $2^t(t\geq 3)$ contains $2^{k+1}$ number of elements. Hence, the number of connected components of elements of order $2^t(t\geq 3)$ in the graph $\Gamma({\mathbb{Z}_{2^{n_1}} \times \mathbb{Z}_{2^{n_2}}\times \dots \times \mathbb{Z}_{2^{n_k}}})$ is $\frac{2^{\sum_{i=1}^{k}}min\{n_i,2^t\}-2^{\sum_{i=1}^{k}}min\{n_i,2^{t-1}\}}{2^{k+1}}$.  
 \end{proof}
 
 \begin{thm}
 	The graph $\Gamma({\mathbb{Z}_{2^{n_1}} \times \mathbb{Z}_{2^{n_2}}\times \dots \times \mathbb{Z}_{2^{n_k}}})$ is a disconnected graph cotaining the following connected components:
 	\begin{enumerate}
 		\item One connected component isomorphic to $\Gamma({\mathbb{Z}_2^k})$.
 		\item The number of connected component containing the elements of order $4$ is $\frac{2^{\sum_{i=1}^{k}min\{n_i,2\}}-2^k}{2^k}$, where ${2^{\sum_{i=1}^{k}min\{n_i,2\}}-2^k}$ is the total number of elements of order $4$ in the group.
 		\item The number of connected component containing the elements of order $2^t(t\geq 3)$ is $\frac{2^{\sum_{i=1}^{k}}min\{n_i,2^t\}-2^{\sum_{i=1}^{k}}min\{n_i,2^{t-1}\}}{2^{k+1}}$, where $2^{\sum_{i=1}^{k}}min\{n_i,2^t\}-2^{\sum_{i=1}^{k}}min\{n_i,2^{t-1}\}$ is the total number of elements of order $2^t(t\geq 3)$ in the group.
 	\end{enumerate}
 \end{thm} 
 
 \begin{proof}
 	Following the above \autoref{Zero 2-group}, \autoref{First 2-group}, \autoref{Third 2-group} and \autoref{Forth 2-group}, the theorem is proved. 
 \end{proof}

\section{Properties of $\Gamma(\mathbb{Z}_{2^{n_1}{p_2}^{n_2}})$}
\label{Properties_of_fifth_group}

The group $\mathbb{Z}_{2^{n_1}{p_2}^{n_2}}$ contains the non-identity elements of following orders:

\begin{itemize}
	\item  $2p_2$;
	\item $2^{k_1}$ for $1 \leq k_1 \leq n_1$;
	\item $p_2^{k_2}$ for $1 \leq k_2 \leq n_2$;
	\item $2^{k_1}p_2$ for $2 \leq k_1 \leq n_1$;
	\item $2p_2^{k_2}$ for $2 \leq k_2 \leq n_2$;
	\item $2^{k_1}p_2^{k_2}$ for $2 \leq k_1 \leq n_1-1$, and $2 \leq k_2 \leq n_2-1$.
\end{itemize}
The corresponding graphs are disconnected when any of $n_1$ or $n_2 \neq 1$. An example of is depicted in \autoref{Z_2_2_3_2_POE_Graph}.
 
\begin{figure}
	\centering
	\begin{tikzpicture}
		\node [above] at (1, 2) {${(0,0)}$};
		\draw [fill, black] (1, 2) circle[radius = 1mm];
		\node [above left] at (0, 1) {${(0,3)}$};
		\draw [fill, black] (0, 1) circle[radius = 1mm];
		\node [above right] at (2, 1) {${(0,6)}$};
		\draw [fill, black] (2, 1) circle[radius = 1mm];
		\node [below left] at (0, -1) {${(2,6)}$};
		\draw [fill, black] (0, -1) circle[radius = 1mm];
		\node [below right] at (2, -1) {${(2,3)}$};
		\draw [fill, black] (2, -1) circle[radius = 1mm];
		\node [below] at (1, -2) {${(2,0)}$};
		\draw [fill, black] (1, -2) circle[radius = 1mm];
		\draw [] (1, 2) -- (0, 1);
		\draw [] (1, 2) -- (2, 1);
		\draw [] (0, 1) -- (0, -1);
		\draw [] (2, 1) -- (2, -1);
		\draw [] (0, -1) -- (1, -2);
		\draw [] (2, -1) -- (1, -2);
		\draw [] (1, 2) -- (1, -2);

		\node [left] at (3.5, 0) {${(1,0)}$};
		\draw [fill, black] (3.5, 0) circle[radius = 1mm];
		\node [above] at (4, 1) {${(3,3)}$};
		\draw [fill, black] (4, 1) circle[radius = 1mm];
		\node [below] at (4, -1) {${(3,6)}$};
		\draw [fill, black] (4, -1) circle[radius = 1mm];
		\node [above] at (5, 1) {${(1,3)}$};
		\draw [fill, black] (5, 1) circle[radius = 1mm];
		\node [below] at (5, -1) {${(1,6)}$};
		\draw [fill, black] (5, -1) circle[radius = 1mm];
		\node [right] at (5.5, 0) {${(3,0)}$};
		\draw [fill, black] (5.5, 0) circle[radius = 1mm];
		\draw [] (3.5,0) -- (4,1);
		\draw [] (3.5,0) -- (4,-1);
		\draw [] (5.5,0) -- (5,1);
		\draw [] (5.5,0) -- (5,-1);
		\draw [] (4,1) -- (5,1);
		\draw [] (4,-1) -- (5,-1);
		\draw [] (4,1) -- (4,-1);
		\draw [] (5,1) -- (5,-1);
		
		\node [left] at (-5, 0) {${(3,8)}$};
		\draw [fill, black] (-5, 0) circle[radius = 1mm];
		\node [left] at (-5, 1) {${(1,7)}$};
		\draw [fill, black] (-5, 1) circle[radius = 1mm];
		\node [left] at (-5, 2) {${(1,4)}$};
		\draw [fill, black] (-5, 2) circle[radius = 1mm];
		\node [left] at (-5, -1) {${(3,4)}$};
		\draw [fill, black] (-5, -1) circle[radius = 1mm];
		\node [left] at (-5, -2) {${(3,7)}$};
		\draw [fill, black] (-5, -2) circle[radius = 1mm];
		\node [left] at (-5, -3) {${(1,8)}$};
		\draw [fill, black] (-5, -3) circle[radius = 1mm];
		
		\node [right] at (-3, 0) {${(1,1)}$};
		\draw [fill, black] (-3, 0) circle[radius = 1mm];
		\node [right] at (-3, 1) {${(3,2)}$};
		\draw [fill, black] (-3, 1) circle[radius = 1mm];
		\node [right] at (-3, 2) {${(3,5)}$};
		\draw [fill, black] (-3, 2) circle[radius = 1mm];
		\node [right] at (-3, -1) {${(1,5)}$};
		\draw [fill, black] (-3, -1) circle[radius = 1mm];
		\node [right] at (-3, -2) {${(1,2)}$};
		\draw [fill, black] (-3, -2) circle[radius = 1mm];
		\node [right] at (-3, -3) {${(3,1)}$};
		\draw [fill, black] (-3, -3) circle[radius = 1mm];
		
		\draw [] (-5,2) -- (-3,1);
		\draw [] (-5,2) -- (-3,-1);
		\path [-] (-5,2) edge[bend right=60] node {} (-5,0);
		\draw [] (-5,1) -- (-3,2);
		\draw [] (-5,1) -- (-5,0);
		\draw [] (-5,1) -- (-3,-2);
		\draw [] (-5,0) -- (-3,-3);
		\draw [] (-5,-1) -- (-3,-2);
		\draw [] (-5,-1) -- (-3,2);
		\path [-] (-5,-1) edge[bend right=60] node {} (-5,-3);
		\draw [] (-5,-2) -- (-5,-3);
		\draw [] (-5,-2) -- (-3,-1);
		\draw [] (-5,-2) -- (-3,1);
		\draw [] (-5,-3) -- (-3,0);
		\draw [] (-3,1) -- (-3,0);
		\draw [] (-3,-2) -- (-3,-3);
		\path [-] (-3,2) edge[bend left=60] node {} (-3,0);
		\path [-] (-3,-1) edge[bend left=60] node {} (-3,-3);
		
		\node [left] at (-10, 0) {${(0,8)}$};
		\draw [fill, black] (-10, 0) circle[radius = 1mm];
		\node [left] at (-10, 1) {${(0,7)}$};
		\draw [fill, black] (-10, 1) circle[radius = 1mm];
		\node [left] at (-10, 2) {${(0,4)}$};
		\draw [fill, black] (-10, 2) circle[radius = 1mm];
		\node [left] at (-10, -1) {${(2,4)}$};
		\draw [fill, black] (-10, -1) circle[radius = 1mm];
		\node [left] at (-10, -2) {${(2,7)}$};
		\draw [fill, black] (-10, -2) circle[radius = 1mm];
		\node [left] at (-10, -3) {${(2,8)}$};
		\draw [fill, black] (-10, -3) circle[radius = 1mm];
		
		\node [right] at (-8, 0) {${(0,1)}$};
		\draw [fill, black] (-8, 0) circle[radius = 1mm];
		\node [right] at (-8, 1) {${(0,2)}$};
		\draw [fill, black] (-8, 1) circle[radius = 1mm];
		\node [right] at (-8, 2) {${(0,5)}$};
		\draw [fill, black] (-8, 2) circle[radius = 1mm];
		\node [right] at (-8, -1) {${(2,5)}$};
		\draw [fill, black] (-8, -1) circle[radius = 1mm];
		\node [right] at (-8, -2) {${(2,2)}$};
		\draw [fill, black] (-8, -2) circle[radius = 1mm];
		\node [right] at (-8, -3) {${(2,1)}$};
		\draw [fill, black] (-8, -3) circle[radius = 1mm];
		
		\draw [] (-10,1) -- (-10,0);
		\draw [] (-8,1) -- (-8,0);
		\draw [] (-10,-2) -- (-10,-3);
		\draw [] (-8,-2) -- (-8,-3);
		\path [-] (-10,2) edge[bend right=60] node {} (-10,0);
		\draw [] (-10,2) -- (-8,1);
		\draw [] (-10,2) -- (-8,-1);
		\path [-] (-8,2) edge[bend left=60] node {} (-8,0);
		\draw [] (-8,2) -- (-10,1);
		\draw [] (-8,2) -- (-10,-1);
		\draw [] (-10,1) -- (-8,-2);
		\draw [] (-8,1) -- (-10,-2);
		\draw [] (-10,0) -- (-8,-3);
		\draw [] (-8,0) -- (-10,-3);
		\draw [] (-10,-1) -- (-8,-2);
		\draw [] (-8,-1) -- (-10,-2);
		\path [-] (-8,-1) edge[bend left=60] node {} (-8,-3);
		\path [-] (-10,-1) edge[bend right=60] node {} (-10,-3);
	\end{tikzpicture}
	\caption{The POE graph $\Gamma({\mathbb{Z}_{2^23^2}})$ is a graph with $4$ connected components. }
	\label{Z_2_2_3_2_POE_Graph}
\end{figure}
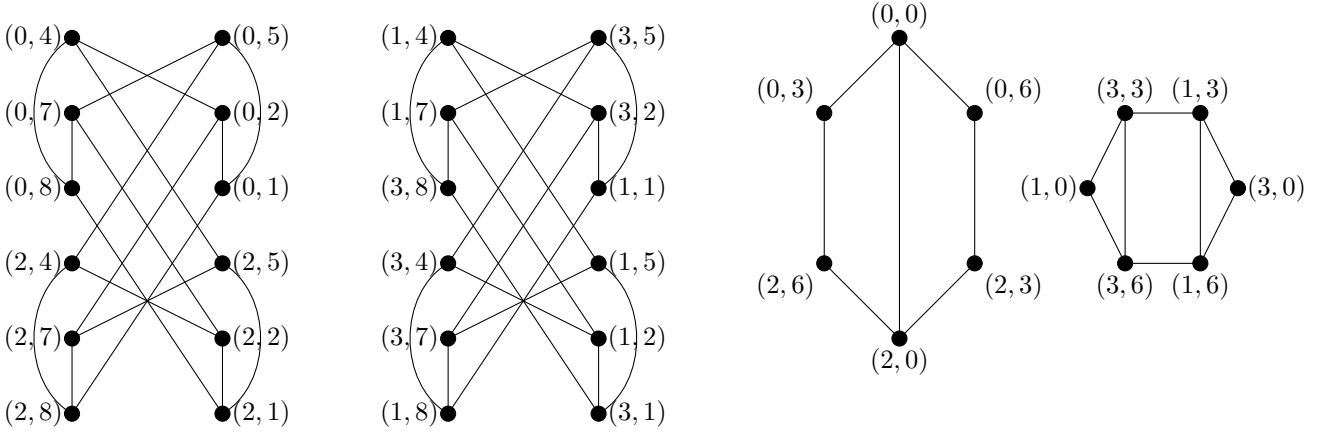

\begin{lemma}\label{Disconnected lemma 2}
	The graph $\Gamma(\mathbb{Z}_{2^{n_1}{p_2}^{n_2}})$ is disconnected and one connected component is isomorphic to $\Gamma(\mathbb{Z}_{2p_2})$.
\end{lemma}

\begin{proof}
	The group $\mathbb{Z}_{2^{n_1}{p_2}^{n_2}}$ is isomorphic to $\mathbb{Z}_{2^{n_1}}\times \mathbb{Z}_{p_2^{n_2}}$. The vertices of the graph $\Gamma(\mathbb{Z}_{2^{n_1}}\times \mathbb{Z}_{p_2^{n_2}})$ can be labeled by the elements in $\{(x,y):x\in \mathbb{Z}_{2^{n_1}},y\in \mathbb{Z}_{p_2^{n_2}}\}$. There are two types of elements of prime order in the group $\mathbb{Z}_{2^{n_1}}\times \mathbb{Z}_{p_2^{n_2}}$, which are the followings:
	\begin{enumerate}
		\item 
			The elements of order $2$. They are of the form $(x_1,e_2)$, where $o(x_1)=2$ in $\mathbb{Z}_{2^{n_1}}$ that is $o(x_1,e_2)=2$ in $\mathbb{Z}_{2^{n_1}}\times \mathbb{Z}_{p_2^{n_2}}$.
		\item 
			The elements of order $p_2$. They are of the form $(e_1,y_1)$, where $o(y_1)=p_2$ in $\mathbb{Z}_{p_2^{n_2}}$ that is $o(e_1,y_2)=p_2$ in $\mathbb{Z}_{2^{n_1}}\times \mathbb{Z}_{p_2^{n_2}}$.
	\end{enumerate}
	Let $C^{(0)}$ be the connected component of $\Gamma(\mathbb{Z}_{2^{n_1}p_2^{n_2}})$ that contains the identity element $(e_1,e_2)$. Clearly, $(e_1,e_2)$ is adjacent to the elements of order $2$ and $p_2$. Also, the elements of order $2p_2$ is of the form $(x_1,y_1)$, where $o(x_1)=2$ in $\mathbb{Z}_{2^{n_1}}$ and $o(y_1)=p_2$ in $\mathbb{Z}_{p_2^{n_2}}$. Note that, there is a path $(e_1,y_1)\sim (e_1,e_2)\sim (x_1,e_2)\sim (x_1,y_1)$ in the graph $\Gamma(\mathbb{Z}_{2^{n_1}}\times \mathbb{Z}_{p_2^{n_2}})$, where $o(e_1,y_1)=p_2$, $o(x_1,e_2)=2$, $o(x_1,y_1)=2p_2$. Therefore, the elements of order $2,p_2,2p_2$ always belong to $C^0$.
	
	Assume that $C^0$ contains an element $(x_2,y_2)$ of higher order such that $o(x_2)=2^{k_1}$ for $k_1\geq 2$ or $o(y_2)=p_2^{k_2}$ for $k_2\geq 2$. Then, at least any one of the elements of order $2$ or $p_2$ or $2p_2$ is adjacent to the element $(x_2,y_2)$. Now, using \autoref{lcm_lemma}, we have the following cases:
	\begin{itemize}
		\item Let $(x_2,y_2)\sim (x_1,e_2)$, where $o(x_1)=2$ in $\mathbb{Z}_{2^{n_1}}$. Then $o(x_2x_1)=2^{k_1}(k_1\geq 2)$, which implies $o((x_2,y_2)(x_1,e_2))=2^{k_1}p_2^{k_2}$ with $k_1\geq 2 \mbox{ or } k_2\geq 2$. This is a contradiction.
		\item Let $(x_2,y_2)\sim (e_1,y_1)$, where $o(y_1)=p_2$ in $\mathbb{Z}_{p_2^{n_2}}$. Then $o(y_2y_1)=p_2^{k_2}(k_2\geq 2)$, which implies $o((x_2,y_2)(e_1,y_1))=2^{k_1}p_2^{k_2}$ with $k_1\geq 2 \mbox{ or } k_2\geq 2$. This is a contradiction.
		\item Let $(x_2,y_2)\sim (x_1,y_1)$, where $o(x_1,y_1)=2p_2$. Then $o((x_2,y_2)(x_1,y_1))=2^{k_1}p_2^{k_2}$ with $k_1\geq 2 \mbox{ or } k_2\geq 2$. This is again a contradiction. 
	\end{itemize} 
	Therefore, in any case, $C^0$ can not contain an element $(x_2,y_2)$ of higher order. Thus, the graph $\Gamma(\mathbb{Z}_{2^{n_1}{p_2}^{n_2}})$ is disconnected.

 		Define a map $f:V(\Gamma(\mathbb{Z}_{2p_2})) \longrightarrow V(C^{(0)})$ by \[ f(a,b)=\big(a2^{n_1-1},\, bp_2^{n_2-1}\big), \quad a\in\mathbb{Z}_2,\ b\in\mathbb{Z}_{p_2}. \] For $f(a,b)=f(a',b')$, we have
 		$$(a-a')2^{n_1-1}\equiv 0 \pmod{2^{n_1}} \quad \text{and} \quad (b-b')p_2^{n_2-1}\equiv 0 \pmod{p_2^{n_2}}.$$ 
 		It implies that $a,a'\in\mathbb{Z}_2$ and $b,b'\in\mathbb{Z}_{p_2}$, and $a=a'$ and $b=b'$. Therefore, $f$ is injective. Also, for some $k\in\{1,2,\dots,p_2-1\}$, each vertex of $C^{(0)}$ has the form 
 		$$(0,0),\ (2^{n_1-1},0),\ (0,kp_2^{n_2-1}),\ (2^{n_1-1},kp_2^{n_2-1}).$$ 
 		Each of these vertices represents the image under $f$ of $(0,0)$, $(1,0)$, $(0,k)$, and $(1,k)$. Therefore, $f$ is surjective. Hence, $f$ is a bijection. Additionally, for any $(a,b),(a',b')\in V(\Gamma(\mathbb{Z}_{2p_2}))$, \[ o\big((a+a',\,b+b')\big)\text{is prime} \iff o\big(f(a,b)f(a',b')\big)\text{is prime}. \] Consequently, $f$ acts as a graph isomorphism, and $\Gamma(\mathbb{Z}_{2p_2}) \cong C^{(0)}$.
 	\end{proof}

\begin{lemma}\label{order 4 element and its inverse}
	In the graph $\Gamma(\mathbb{Z}_{2^{n_1}p_2^{n_2}})$, an element of order $4$ and its inverse belong to the same connected component. 
\end{lemma}

\begin{proof}
	Let $(x,e_2)$ be an element of order $4$, where $o(x)=4$ in $\mathbb{Z}_{2^{n_1}}$. The inverse of the element is $(x^{-1},e_2)$.
	 Since $(x,e_2)\sim (x^{-1},y)\sim (x,y)\sim (x^{-1},e_2)$ is a path, where $\circ(y)=p_2$, then both the elements are order $4$ and its inverse are in the same connected component.
\end{proof}

\begin{lemma}\label{order 4 element adjacency}
	The elements of order $4$ are adjacent only to the elements of order $4p_2$.
\end{lemma}

\begin{proof}
	The group contains only $\phi(4)=2$ elements of order $4$ and they are not adjacent to each other as they are inverse to each other. Let us denote the component contained with the elements of order $4$ as $C^{(1)}$. Let $\circ(x_1,y_1)=4$, it is adjacent to an element $(x_2,y_2)$ if and only if $x_1x_2=e_1$ and $\circ(y_1y_2)=p_2$. Since $y_1=e_2$, then $(x_2,y_2)=(x_1^{-1},y_2)$, where $\circ(y_2)=p_2$. Thus, the element of order $4$ is adjacent to the elements of order $4p_2$ only.
\end{proof}

\begin{lemma}\label{order $4p_2$ adjacency}
	Any neighbor of an elements of order $4p_2$ is an elements of order either $4$ or $4p_2$.
\end{lemma}

\begin{proof}
	Let $\circ(x_1,y_1)=4p_2$. Then, $\circ(x_1)=4$ and $\circ(y_1)=p_2$. Now, $(x_1,y_1)\sim (x_2,y_2)$ if and only if $\circ(x_1x_2)=2$ and $y_1y_2=e_2$ or $x_1x_2=e_2$ and $\circ(y_1y_2)=p_2$. That is either $(x_2,y_2)=(x_1,y_1^{-1})$; or $(x_2,y_2) = (x_1^{-1},e_2)$. Also, for any element $(x_1,y_1)$ of order $4p_2$, it is adjacent to $(x_1^{-1},e_2)$. Therefore, all the elements of order $4p_2$ and $4$ are in one component.
\end{proof}

\begin{lemma}\label{only elements of two orders}
	All the elements of order $4$ and $4p_2$ creates a component and it does not contain any other element.
\end{lemma}

\begin{proof}
	Let the component consists of the elements of order $4$ and $4p_2$ be $C^{(1)}$. Clearly, the elements of prime order, identity and the elements of order $2p_2$ are not in $C^{(1)}$. Also, the elements of order $2^{k_1}(k_1\geq 3)$, $p_2^{k_2}(k_2\geq 2)$, $2^{k_1}p_2(k_1\geq 3)$, $2p_2^{k_2}(k_2\geq 2)$ and $2^{k_1}p_2^{k_2}(k_1,k_2\geq 2)$ cannot be adjacent to the elements of order $4$ and $4p_2$ by the adjacency characterization of $\Gamma(\mathbb{Z}_{2^{n_1}})$ and $\Gamma(\mathbb{Z}_{p_2^{k_2}})$ by \autoref{same-order-lemma}. Hence, we have the result.
\end{proof}

\begin{lemma}\label{adjacency 3}
	The elements of order $2^{k_1}$ are adjacent to only the elements of order $2^{k_1}$ and $2^{k_1}p_2$ for $k_1\geq 3$.
\end{lemma}

\begin{proof}
	The elements of order $2^{k_1}$, where $k_1\geq 3$ is of the form $(x,e_2)$ where $\circ(x)=2^{k_1}$. If $(x,e_2)\sim (x_1,y_1)$ and the composition generates an element of order $2$, then $\circ(x_1)=2^{k_1}$ and $y_1=e_2$. If the composition of $(x,e_2)$ and $(x_1,y_1)$ gives an element of order $p_2$, then $xx_1=e_1$ and $\circ(y_1)=p_2$ that is $\circ(x_1)=2^{k_1}$ and $\circ(y_1)=p_2$. Thus, the elements of order $2^{k_1}$ can be adjacent to the elements of order $2^{k_1}$ and $2^{k_1}p_2$.
\end{proof}

\begin{definition}
	We define an edge between two elements $x$ and $y$ in the POE graph is called a $p$-edge if $o(xy) = p$.
\end{definition} 

\begin{lemma}\label{adjacency 4}
	There are either $0$ or exactly four elements of order \( 2^{k_1} \) in a connected component of $\Gamma(\mathbb{Z}_{2^{n_1}{p_2}^{n_2}})$. Moreover, a component containing the elements of order \( 2^{k_1} \) contains \( 4(p_2 - 1) \) elements of order \( 2^{k_1} p_2 \), for $k \geq 3$.
\end{lemma}

\begin{proof}
	Let $C^{(2)}$ be a connected component containing an element of order $2^{k_1}$, where $k_1\geq 3$. Then, it is of the form $(x,e_2)$, where $\circ(x)=2^{k_1}$. Since $(x,e_2)\sim (x^{-1},y_1)\sim (x,y_2)\sim (x^{-1},e_2)$ with $\circ(y_1)=\circ(y_2)=p_2$ and $y_1y_2\neq e_2$, an element of order $2^{k_1}$ and its inverse are in the same component. By the adjacency condition of $\Gamma(\mathbb{Z}_{2^{m_1}})$, mentioned in \autoref{order 2^n}, the elements $(x,e_2)$ and $(x^{-1}, e_2)$ are adjacent to another element of order $2^{k_1}$. Thus, $C^{(2)}$ contains at least $4$ elements of order $2^{k_1}$, which are $(x_1,e_2), (x_2,e_2), (x_1^{-1},e_2), (x_2^{-1},e_2)$. Each of these $4$ elements form $p_2$-edges with $(p_2 - 1)$ elements as follows:
	\begin{center}
		\begin{itemize}
			\item  
				$(x_1,e_2)\sim (x_1^{-1},y_i)$ ;
			\item 
				$(x_2,e_2)\sim (x_2^{-1},y_i)$;
			\item 
				$(x_1^{-1},e_2)\sim (x_1,y_i)$;
			\item 
				$(x_2^{-1},e_2)\sim (x_2,y_i)$.
		\end{itemize}
	\end{center}
	Here, $i=1,2,\cdots (p_2-1)$ and $\circ(y_i)=p_2$. Now, we can conclude that in $C^{(2)}$, we have at least $4+4(p_2-1) = 4p_2$ elements. 
	
	Now, we are going to prove that there is exactly $4p_2$ elements in $C^{(2)}$. Consider the adjacency relations of $(x_1^{-1},y_i),(x_2^{-1},y_i),(x_1,y_i),(x_2,y_i)$ with the other elements, which are as follows:
	\begin{center}
		\begin{itemize}
			\item $(x_1^{-1},y_i)\sim (x_1,y_j)$ for $i\neq j$;
			\item $(x_2^{-1},y_i)\sim (x_2,y_j)$ for $i\neq j$;
			\item $(x_1,y_i)\sim (x_1^{-1},y_j)$ for $i\neq j$;
			\item $(x_2,y_i)\sim (x_2^{-1},y_j)$ for $i\neq j$.
		\end{itemize}
	\end{center}
	Each of the above adjacency relation is determined by a $p_2$-edge between the elements. Clearly, the number of elements of order $2^{k_1}$ and and the elements adjacent to them which are of order $2^{k_1}p_2$ is $4p_2$ with degree $p_2$. If any one element is added in $C^{(2)}$, the degree of the at least one element will be increased which is not possible. Thus, the component $C^{(2)}$ contains exactly $4$ elements of order $2^{k_1}(k_1\geq 3)$ with a total $4p_2$ number of elements.
\end{proof}

\begin{lemma}\label{adjacency 5}
	 A connected component having an element of order \( p_2^{k_2} \),  where \( k_2 \geq 2 \) contains only elements of order \( 2p_2^{k_2} \). The total number of elements in that component is $4p_2$ .
\end{lemma}

\begin{proof}
	Let $C^{(3)}$ be a connected component which contains an element $(e_1,y)$ of order $p_2^{k_2}$. Obviously, $\circ(y)=p_2^{k_2}$. Also, $(e_1,y)\sim (x_1,y_1)$ if and only if either $\circ(x_1)=2$ and $yy_1=e_2$ or $x_1=e_1$ and $\circ(yy_1)=p_2$. Thus, by the adjacency condition of $\Gamma(\mathbb{Z}_{p_2^{n_2}})$, $\circ(x_1,y_1)=p_2^{k_2} \mbox{ or } 2p_2^{k_2}$. Let $(x_2,y_2)\in C^{(3)}$ such that $\circ(x_2,y_2)=2p_2^{k_2}$ then $\circ(x_2)=2$ and $\circ(y_2)=p_2^{k_2}$. Therefore, $(x_2,y_2)\sim (x_3,y_3)$ if and only if any one of the following conditions are satisfied:
	\begin{center}
		\begin{itemize}
			\item 
				$\circ(x_2x_3)=2$ and $y_2y_3=e_2$ that is $x_3=e_1$ and $\circ(y_3)=p_2^{k_2}$.
			\item 
				$x_2x_3 = e_2$ and $\circ(y_2y_3)=p_2$ that is $\circ(x_3)=2$, and $\circ(y_3)=p_2^{k_2}$.
		\end{itemize}
	\end{center} 
	Under this situation, we can conclude that the elements of order $2p_2^{k_2}$ are adjacent to only with the elements of order $p_2^{k_2}$ and $2p_2^{k_2}$.
	
	Now, come to the second part of the proof. By the adjacency relations of $\Gamma(\mathbb{Z}_{p_2^{n_2}})$, the elements of order $p_2^{k_2}$ can form a $p_2$-edge with the elements of order $p_2^{k_2}$ only. Thus, \autoref{2nd spectral theorem} indicates that $C^{(3)}$ contains $2p_2$ elements of order $p_2^{k_2}$. Let us distribute the elements of order $p_2^{k_2}$ into two sets, such that
	\begin{itemize}
		\item $K_1=\{(e_1,y_i): i=1,2,\cdots,p_2\}$;
		\item $K_2=\{(e_1,y_i^{-1}): i=1,2,\cdots,p_2\}$,
	\end{itemize}
	where $\circ(y_i)=p_2^{k_2}$ and $(e_1,y_i)\sim (e_1,y_j^{-1})$ for $i\neq j$. There are $(p_2-1)$ $p_2$-edges in the induced subgraph $\langle K_1\cup K_2 \rangle$. Now, the elements of $K_1$ and $K_2$ form the $2$-edges which can be classified as follows:
	\begin{itemize}
		\item 
			$L_1=\{(x,y_i^{-1}): (e_1,y_i)\sim (x,y_i^{-1}) \mbox{ for } i=1,2,\cdots,p_2,\circ(x)=2\}$.
		\item 
			$L_2=\{(x,y_i): (e_1,y_i^{-1})\sim (x,y_i) \mbox{ for } i=1,2,\cdots,p_2,\circ(x)=2\}$.
	\end{itemize}
	Recall that degree of each element in the connected component is $p_2$. The degree of a vertex is at most the total number of prime order element in the underlying group. Now, $(x,y_i)\sim (x,y_j^{-1})(i\neq j)$ and they form $p_2$-edge among them. Thus, the degree of each of the total $4p_2$ elements is $p_2$ which is the maximum degree for the graph. If any element is added in the component, the degree of at least one of those $4p_2$ elements will be increased which is not possible in this case. 
\end{proof}

Now, we are left with the components of the elements of order $2^{k_1}p_2^{k_2}$ for $k_1,k_2\geq 2$.

\begin{lemma}\label{adjacency 6}
	The elements of order $2^{k_1}p_2^{k_2}(k_1,k_2\geq 2)$ can only be adjacent to the elements of order $2^{k_1}p_2^{k_2}$ only where $k_1,k_2\geq 2$.
\end{lemma}

\begin{proof}
	Let $C^{(4)}$ be a connected component with the elements of order $2^{k_1}p_2^{k_2}$, where $k_1, k_2 \geq 2$. If $(x_1,y_1)\in C^{(4)}$ then $\circ(x_1)=2^{k_1}$ and $\circ(y_1)=p_2^{k_2}$. Therefore, $(x_1,y_1)\sim (x_2,y_2)$ if and only if one of the following holds:
	\begin{itemize}
		\item 
			$\circ(x_1x_2)=2$ and $y_1y_2=e_2$ that is $\circ(x_2)=2^{k_1}$ and $\circ(y_2)=p_2^{k_2}$;
		\item 
			$x_1x_2=e_1$ and $\circ(y_1y_2)=p_2$ that is $\circ(x_2)=2^{k_1}$ and $\circ(y_2)=p_2^{k_2}$.
	\end{itemize}
	Thus, the elements of order $2^{k_1}p_2^{k_2}$ can only be adjacent to the element of order $2^{k_1}p_2^{k_2}$.
\end{proof}

\begin{lemma}\label{only elements}
	The connected component $C^{(4)}$ contains $4p_2$ elements.
\end{lemma}

\begin{proof}
	Let $(x_1,y_1)\in C^{(4)}$ such that $\circ(x_1,y_1)=2^{k_1}p_2^{k_2}$. By the adjacency relations of $\Gamma(\mathbb{Z}_{p_2^{n_2}})$, the vertex $(x_1,y_1)$ is adjacent to $(p_2-1)$ elements by $p_2$-edges which are of the form 
	$$K_1=\{(x_1^{-1},y_i): i=2,3,\cdots,p_2\}.$$
	Similarly, $(x_1^{-1},y_1^{-1})\in C^{(4)}$, is adjacent to $(p_2-1)$ elements to form $p_2$-edges which are
	$$K_2=\{(x_1,y_i^{-1}):i=2,3,\cdots,p_2\}.$$
	Now, the set $K_1\cup K_2\cup\{(x_1,y_1),(x_1^{-1},y_1^{-1})\}$ consisting $2p_2$ elements. Each element in the set having $(p_2-1)$ number of $p_2$-edges among themselves.
	
	The group contains one element of order $2$ and each element $x$ form a $p$-edge with an unique element $y$ in the group. So, for each $x$, there is a unique $y$ such that $o(xy)=2$. Thus, each element in $C^{(4)}$ has a $2$-edge. For the $2$-edge, $(x_1,y_1)\sim (x_2,y_1^{-1})$ and $(x_1^{-1},y_1^{-1})\sim (x_2^{-1},y_1)$, where $\circ(x_1x_2)=e_1$. The elements in $K_1$ forms $2$-edge with the set of elements
	\begin{center}
		$L_1=\{(x_2,y_i): (x_1,y_i^{-1})\sim(x_2,y_i), i=2,3,\cdots, p_2 \}$.
	\end{center}
	Similarly, The elements in $K_2$ forms $2$-edge with the set of elements
	\begin{center}
		$L_2=\{(x_2^{-1},y_i^{-1}): (x_1^{-1},y_i)\sim(x_2^{-1},y_i^{-1}), i=2,3,\cdots, p_2 \}$.
	\end{center}
	Now, we have another set consisting of $2p_2$ elements $L_1\cup L_2 \cup \{(x_2,y_1^{-1}),(x_2^{-1},y_1)\}$. The elements in $L_1$ and $L_2$ form $p_2$-edges among themselves in the following way
	\begin{itemize}
		\item 
			$(x_2,y_i)\sim (x_2^{-1},y_j^{-1})$ for $i\neq j$;
		\item $(x_2^{-1},y_i^{-1})\sim (x_2,y_j)$ for $i\neq j$.
	\end{itemize}
	Also, $(x_2,y_i)\sim (x_2^{-1},y_j^{-1})$ and $(x_2^{-1},y_i^{-1})\sim (x_2,y_j)$ for $i\neq j$. 
	
	Thus, all $4p_2$ distinct elements have the degree $p_2$ each. If any more element is added, the degree of at least one of the elements of those $4p_2$ elements will be increased which is impossible.
\end{proof}

	\begin{thm}
	The POE graph $\Gamma(\mathbb{Z}_{2^{n_1}p_2^{n_2}})$ is the union of the following connected components: 
	\begin{enumerate}
		\item 
			A connected component with $2p_2$ vertices which is isomorphic to $\Gamma(\mathbb{Z}_{2p_2})$.
		\item 
			A connected component with $2p_2$ vertices containing the elements of order $4$ and the elements of order $4p_2$.
		\item The graph consists of $2^{n_1-2}p_2^{n_2-1}-1$ regular connected components with $4p_2$ vertices in each with regularity $p_2$.
	\end{enumerate}
	
\end{thm}

\begin{proof}
	\autoref{Disconnected lemma 2} indicates that there is a connected component of the graph which is isomorphic to  $\Gamma(\mathbb{Z}_{2p_2})$. Also, using \autoref{order 4 element and its inverse}, \ref{order 4 element adjacency}, \ref{order $4p_2$ adjacency}, and \ref{only elements of two orders}, we say $\Gamma(\mathbb{Z}_{2^{m_1}p_2^{m_2}})$ has one connected component with $2p_2$ vertices. The elements in it are of order $4$ and $4p_2$. Form the \autoref{adjacency 3}, \ref{adjacency 4}, \ref{adjacency 5}, \ref{adjacency 6}, \ref{only elements}, it is clear that all the other connected components contain $4p_2$ vertices. 
	
	Let $(x,y)$ be a vertex in the graph $\Gamma(\mathbb{Z}_{2^{n_1}p_2^{n_2}})$, such that $o(x,y)\neq 2 \mbox{ or } p_2 \mbox{ or } 4$. Then, $o((x,y)^2) \neq p_2 \mbox{ or } 2$, therefore, there exist unique $(x_1,y_1)$ for every $(x,y)$, such that $o((x,y)(x_1,y_1))=2 \mbox{ or } p_2$. There is one element of order $2$ and $(p_2 - 1)$ elements or order $p_2$. Thus, except the elements of order $p_2,4$ and $2$, every element has degree $p_2$. Therefore, except the component isomorphic to $\Gamma(\mathbb{Z}_{2p_2})$ and the component containing the elements of order $4$, all other components are $p_2$-regular and each of them contains $4p_2$ vertices.   
\end{proof}


We use the idea of equitable partition on graphs to discuss the irrational spectrum of $\Gamma(\mathbb{Z}_{2p_2})$. The idea is defined as follows:
\begin{definition}{\label{equitable_partitions}}
	A partition $\mathcal{P} = \{P_0, P_1, \dots, P_{m-1}\}$ of the vertices of a graph $\Gamma$ is called \emph{equitable} if, for every pair of (not necessarily distinct) indices $i,j \in \{0,1,\dots,m-1\}$, there exists a non-negative integer $a_{ij}$ such that each vertex $v \in P_i$ has exactly $a_{ij}$ neighbors in $P_j$, regardless of the choice of $v$.
	
	The matrix $A({\Gamma}/{\mathcal{P}}) = (a_{ij})_{m \times m}$ is called the \emph{partition matrix} or \emph{quotient matrix}.	
\end{definition}

\begin{thm}\cite{godsil2013algebraic}[Chapter 9, Theorem 9.3.3]\label{equitable_partition_theorem}
	``If $\pi$ is an equitable partition of a graph $X$, then the characteristic polynomial $A(X/\pi)$ of the quotient matrix of the graph $X$ divides the characteristic polynomial $A(X)$ of the graph."
\end{thm}

\begin{thm}
	The roots of the equations $(x^2-(p_2-1)x-1)=0$ and $(x^2-(p_2-5)x-(2p_2-5))=0$ are the irrational spectrum of $\Gamma(\mathbb{Z}_{2p_2})$.
\end{thm}

\begin{proof}
	We divide the vertices of $\Gamma(\mathbb{Z}_{2p_2})$ into $4$ partitions as $V_1=\{e\}$,
	$V_2=\{x^{p_2}\}$,
	$V_3=\{x^{2k} : 1 \le k \le p_2-1\}$,
	and
	$V_4=\{x^{2k+1} : 0 \le k \le p_2-1\}$. The elements $a_{i,j}$ as in the \autoref{equitable_partitions} are as follows: 
	\begin{table}[h]
		\centering
		\begin{tabular}{c|cccc}
			\hline
			Partitions & $V_1$ & $V_2$ & $V_3$ & $V_4$ \\
			\hline
			$V_1$ & $0$ & $1$ & $p_2-1$ & $0$ \\
			$V_2$ & $1$ & $0$ & $0$ & $p_2-1$ \\
			$V_3$ & $1$ & $0$ & $p_2-3$ & $1$ \\
			$V_4$ & $0$ & $1$ & $1$ & $p_2-3$ \\
			\hline
		\end{tabular}
	\end{table}

Now, the characteristic polynomial for quotient matrix $A({\Gamma}/{\mathcal{P}})$ with the block division is 
$$\begin{vmatrix}
	-x && 1 && \vline && p_2-1 && 0 \\
	1 && -x && \vline && 0 && p_2-1 \\
	\hline
	1 && 0 && \vline && p_2-3-x && 1 \\
	0 && 1 && \vline && 1 && p_2-3-x		
\end{vmatrix} = 0.$$
We denote it as $\det \begin{bmatrix}
 	A && \vline && B\\
 	\hline 
 	C && \vline && D
 \end{bmatrix}=0$, where $A=\begin{bmatrix}
 -x && 1 \\
 1 && -x
 \end{bmatrix}$, $B=\begin{bmatrix}
 p_2-1 && 0\\
 0 && p_2-1
 \end{bmatrix}$, $C=\begin{bmatrix}
 1 && 0 \\
 0 && 1
 \end{bmatrix}$ and $D=\begin{bmatrix}
 p_2-3-x && 1\\
 1 && p_2-3-x
 \end{bmatrix}$.
Since $C$ and $D$ commute to each other, then,
	$\det \begin{bmatrix}
		A && \vline && B\\
		\hline 
		C && \vline && D\\
	\end{bmatrix} = \det(AD-BC)$, where
	$AD =\begin{bmatrix}
		x^2-(p_2-3)x+1 && p_2-3-2x\\
		p_2-3-2x && x^2-(p_2-3)x+1
		\end{bmatrix}$ and $BC = \begin{bmatrix}
		p_2-1 && 0\\
		0 && p_2-1
		\end{bmatrix}$.
Therefore,
\begin{align*}
	& \begin{vmatrix}
		-x && 1 && \vline && (p_2-1) && 0 \\
		1 && -x && \vline && 0 && (p_2-1) \\
		\hline
		1 && 0 && \vline && p_2-3-x && 1 \\
		0 && 1 && \vline && 1 && p_2-3-x		
	\end{vmatrix} = \begin{vmatrix}
		x^2-(p_2-3)x-p_2+2 && p_2-3-2x\\
		p_2-3-2x && x^2-(p_2-3)x-p_2+2
	\end{vmatrix}\\
	= &  (x^2-(p_2-3)x-p_2+2)^2-(p_2-3-2x)^2 \\
	= & (x^2-(p_2-3)x-p_2+2+p_2-3-2x)(x^2-(p_2-3)x-p_2+2-p_2+3+2x) \\
	= & (x^2-(p_2-1)x-1)(x^2-(p_2-5)x-(2p_2-5)).  
\end{align*}
Following the \autoref{equitable_partition_theorem}, the roots of the equations $(x^2-(p_2-1)x-1)=0$ and $(x^2-(p_2-5)x-(2p_2-5))=0$ are the irrational spectrum of the graph.
\end{proof}

\begin{thm}
	The POE graph $\Gamma(\mathbb{Z}_{2^{n_1}p_2^{n_2}\dots p_k^{n_k}})$ is the union of following connected components:
	\begin{enumerate}
		\item A connected component with $2p_2p_3\dots p_k$ vertices which is isomorphic to $\Gamma(\mathbb{Z}_{2p_2p_3\dots p_k})$.
		\item A connected component with $2p_2p_3\dots p_k$ vertices containing the elements of order $4$ and the elements of order $4p_2$,$4p_3\dots 4p_k$.
		\item The graph consists of $2^{n_1-2}p_2^{n_2-1}\dots p_k^{n_k-1}-1$ regular connected components with $4p_2p_3\dots p_k$ vertices in each with regularity $p_2+p_3+\dots p_k$.
	\end{enumerate}
\end{thm}

\begin{proof}
	The group $\mathbb{Z}_{2^{n_1}p_2^{n_2}\dots p_k^{n_k}} \cong {\mathbb{Z}_{2^{n_1}}} \times {\mathbb{Z}_{{p_2}^{n_2}}} \dots \times {\mathbb{Z}_{{p_k}^{n_k}}}$. In the graph $\Gamma({\mathbb{Z}_{2^{n_1}}} \times {\mathbb{Z}_{{p_2}^{n_2}}} \dots \times {\mathbb{Z}_{{p_k}^{n_k}}})$, identity, the element of order $2$ and the elements of other prime orders belong to one connected component. Since there is always a path $(e_1,e_2,\dots, e_k)\sim (x,e_2,e_3,\dots , e_k)\sim (x,y_2,e_3,\dots,e_k)\sim (x,y_2^{-1},y_3,e_4,\dots, e_k)\sim \dots \sim (x,y_2,y_3^{-1},\dots, y_k)$ if $k$ is odd and $(e_1,e_2,\dots, e_k)\sim (x,e_2,e_3,\dots , e_k)\sim (x,y_2,e_3,\dots,e_k)\sim (x,y_2^{-1},y_3,e_4,\dots, e_k)\sim \dots \sim (x,y_2^{-1},\dots, y_k)$, the elements of order $2^{\eta_1} p_2^{\eta_2} \dots p_k^{\eta_k}$ where $\eta_i = 0$ or $1$ belong to the same connected component. Now, by applying the same method as \autoref{Disconnected lemma 2}, the graph is disconnected and one connected component is isomorphic to $\Gamma(\mathbb{Z}_{2p_2p_3\dots p_k})$. Also, using the same technique and applying mathematical induction, from the \autoref{order 4 element and its inverse}, \autoref{order 4 element adjacency},\autoref{order $4p_2$ adjacency},\autoref{only elements of two orders},\autoref{adjacency 3},\autoref{adjacency 4},\autoref{adjacency 5},\autoref{adjacency 6},\autoref{only elements} and generalizing them we can have the required proof.
\end{proof}
		
\section{Properties of $\Gamma(\mathbb{Z}_{{p_1^{n_1}}{p_2^{n_2}}})$}
\label{Properties_of_sixth_group}

We can prove that the group $\mathbb{Z}_{p_1^{n_1}p_2^{n_2}}\cong \mathbb{Z}_{p_1^{n_1}} \times \mathbb{Z}_{p_2^{n_2}}$. The elements of $\mathbb{Z}_{p_1^{n_1}} \times \mathbb{Z}_{p_2^{n_2}}$ can be expressed as $\{(x,y): x\in \mathbb{Z}_{p_1^{n_1}}, y\in \mathbb{Z}_{p_2^{n_2}}\}$. The group $\mathbb{Z}_{p_1^{n_1}p_2^{n_2}}$ contains the elements of following orders:
\begin{itemize}
	\item 
	Elements of the form $(x,e_2)$, where $o(x)=p_1^{k_1}$ for $k_1 = 1, 2, \dots, (n_1 - 1)$ in $\mathbb{Z}_{p_1^{n_1}}$. They are of order $p_1^{k_1}$.
	\item 
	Elements of the form $(e_1,y)$, where $o(y)=p_2^{k_2}$ for $k_2=1,2,\dots, (n_2-1)$ in $\mathbb{Z}_{p_2^{n_2}}$. They are of order $p_2^{k_2}$.
	\item Elements of the form $(x,y)$, where $o(x)=p_1^{k_1}$ for $k_1=0,1,2\dots, (n_1-1)$ or $o(y)=p_2^{k_2}$ for $k_2=0,1,2\dots, (n_2-1)$. They are of order $p_1^{k_1}p_2^{k_2}$.
\end{itemize} 

The group contains $(p_1-1)$ elements of order $p_1$ and $(p_2-1)$ elements of order $p_2$. Hence, the identity $(e_1,e_2)$ is adjacent to all the $(p_1+p_2-2)$ elements. Two vertices \((x_1,y_1) \sim (x_2,y_2)\) in \(\Gamma(\mathbb{Z}_{p_1^{n_1}} \times \mathbb{Z}_{p_2^{n_2}})\) are adjacent if and only if any of the following conditions holds:
\begin{itemize}
	\item \(\circ(x_1 x_2) = p_1\) and \(y_1 y_2 = e_2\); or
	\item \(\circ(y_1 y_2) = p_2\) and \(x_1 x_2 = e_1\).
\end{itemize}

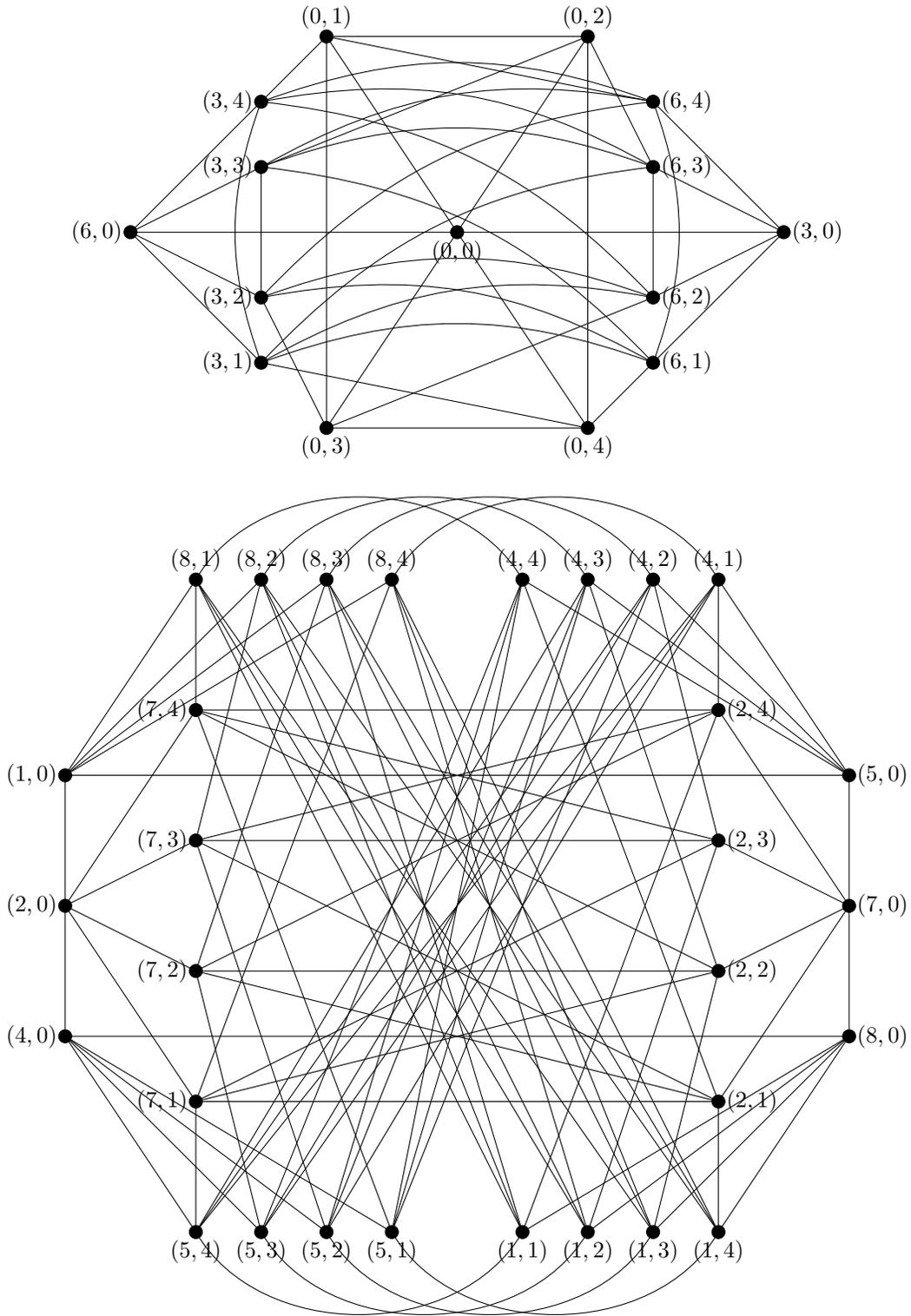
\begin{figure}
	\begin{subfigure}[b]{0.99\textwidth}
		\centering
		\begin{tikzpicture}
			\node [below] at (0, 0) {${(0,0)}$};
			\draw [fill, black] (0,0) circle[radius = 1mm];
			\node [right] at (5, 0) {${(3,0)}$};
			\draw [fill, black] (5,0) circle[radius = 1mm];
			\node [left] at (-5, 0) {${(6,0)}$};
			\draw [fill, black] (-5,0) circle[radius = 1mm];
			\node [above] at (-2, 3) {${(0,1)}$};
			\draw [fill, black] (-2,3) circle[radius = 1mm];
			\node [below] at (-2, -3) {${(0,3)}$};
			\draw [fill, black] (-2, -3) circle[radius = 1mm];
			\node [above] at (2, 3) {${(0,2)}$};
			\draw [fill, black] (2, 3) circle[radius = 1mm];
			\node [below] at (2, -3) {${(0,4)}$};
			\draw [fill, black] (2, -3) circle[radius = 1mm];
			\node [right] at (3, 2) {${(6,4)}$};
			\draw [fill, black] (3,2) circle[radius = 1mm];
			\node [right] at (3, 1) {${(6,3)}$};
			\draw [fill, black] (3,1) circle[radius = 1mm];
			\node [right] at (3, -1) {${(6,2)}$};
			\draw [fill, black] (3,-1) circle[radius = 1mm];
			\node [right] at (3, -2) {${(6,1)}$};
			\draw [fill, black] (3, -2) circle[radius = 1mm];
			
			\node [left] at (-3, 2) {${(3,4)}$};
			\draw [fill, black] (-3,2) circle[radius = 1mm];
			\node [left] at (-3, 1) {${(3,3)}$};
			\draw [fill, black] (-3,1) circle[radius = 1mm];
			\node [left] at (-3, -1) {${(3,2)}$};
			\draw [fill, black] (-3,-1) circle[radius = 1mm];
			\node [left] at (-3,-2) {${(3,1)}$};
			\draw [fill, black] (-3,-2) circle[radius = 1mm];
			\draw [] (0,0) -- (-2,3);
			\draw [] (0,0) -- (2,3);
			\draw [] (0,0) -- (-2,-3);
			\draw [] (0,0) -- (2,-3);
			\draw [] (0,0) -- (5,0);
			\draw [] (0,0) -- (-5,0);
			\path [-] (-3,-2) edge[bend left=20] node {} (3,1);
			\path [-] (-3,-2) edge[bend left=20] node {} (3,-1);
			\path [-] (-3,-2) edge[bend left=20] node {} (3,-2);
			
			\path [-] (-3,-1) edge[bend left=20] node {} (3,2);
			\path [-] (-3,-1) edge[bend left=20] node {} (3,-1);
			\path [-] (-3,-1) edge[bend left=20] node {} (3,-2);
			
			\path [-] (-3,1) edge[bend left=20] node {} (3,2);
			\path [-] (-3,1) edge[bend left=20] node {} (3,1);
			\path [-] (-3,1) edge[bend left=20] node {} (3,-2);
			
			\path [-] (-3,2) edge[bend left=20] node {} (3,2);
			\path [-] (-3,2) edge[bend left=20] node {} (3,1);
			\path [-] (-3,2) edge[bend left=20] node {} (3,-1);
			
			\draw [] (-2,3) -- (2,3);
			\draw [] (-2,3) -- (-2,-3);
			\draw [] (2,3) -- (2,-3);
			\draw [] (2,-3) -- (-2,-3);
			
			\draw [] (-5,0) -- (-3,-2);
			\draw [] (-5,0) -- (-3,-1);
			\draw [] (-5,0) -- (-3,1);
			\draw [] (-5,0) -- (-3,2);
			
			\draw [] (5,0) -- (3,2);
			\draw [] (5,0) -- (3,1);
			\draw [] (5,0) -- (3,-1);
			\draw [] (5,0) -- (3,-2);
			
			\draw [] (-2,3) -- (3,2);
			\draw [] (-2,3) -- (-3,2);
			\draw [] (2,-3) -- (3,-2);
			\draw [] (2,-3) -- (-3,-2);
			
			\draw [] (2,3) -- (3,1);
			\draw [] (2,3) -- (-3,1);
			
			\draw [] (-2,-3) -- (3,-1);
			\draw [] (-2,-3) -- (-3,-1);
			
			\path [-] (-3,2) edge[bend right=20] node {} (-3,-2);
			\path [-] (3,2) edge[bend left=20] node {} (3,-2);
			\draw [] (-3,1) -- (-3,-1);
			\draw [] (3,1) -- (3,-1);
			
		\end{tikzpicture}
	\end{subfigure} 
	
	\begin{subfigure}[b]{0.99\textwidth}
		\centering
		\begin{tikzpicture}
			\node [left] at (-4,2) {${(7,4)}$};
			\draw [fill, black] (-4,2) circle[radius = 1mm];
			\node [left] at (-4,0) {${(7,3)}$};
			\draw [fill, black] (-4,0) circle[radius = 1mm];
			\node [left] at (-4,-2) {${(7,2)}$};
			\draw [fill, black] (-4,-2) circle[radius = 1mm];
			\node [left] at (-4,-4) {${(7,1)}$};
			\draw [fill, black] (-4,-4) circle[radius = 1mm];
			\node [right] at (4,2) {${(2,4)}$};
			\draw [fill, black] (4,2) circle[radius = 1mm];
			\node [right] at (4,0) {${(2,3)}$};
			\draw [fill, black] (4,0) circle[radius = 1mm];
			\node [right] at (4,-2) {${(2,2)}$};
			\draw [fill, black] (4,-2) circle[radius = 1mm];
			\node [right] at (4,-4) {${(2,1)}$};
			\draw [fill, black] (4,-4) circle[radius = 1mm];
			\node [above] at (-4,4) {${(8,1)}$};
			\draw [fill, black] (-4,4) circle[radius = 1mm];
			\node [above] at (-3,4) {${(8,2)}$};
			\draw [fill, black] (-3,4) circle[radius = 1mm];
			\node [above] at (-2,4) {${(8,3)}$};
			\draw [fill, black] (-2,4) circle[radius = 1mm];
			\node [above] at (-1,4) {${(8,4)}$};
			\draw [fill, black] (-1,4) circle[radius = 1mm];
			\node [above] at (1,4) {${(4,4)}$};
			\draw [fill, black] (1,4) circle[radius = 1mm];
			\node [above] at (2,4) {${(4,3)}$};
			\draw [fill, black] (2,4) circle[radius = 1mm];
			\node [above] at (3,4) {${(4,2)}$};
			\draw [fill, black] (3,4) circle[radius = 1mm];
			\node [above] at (4,4) {${(4,1)}$};
			\draw [fill, black] (4,4) circle[radius = 1mm];
			\node [below] at (-4,-6) {${(5,4)}$};
			\draw [fill, black] (-4,-6) circle[radius = 1mm];
			\node [below] at (-3,-6) {${(5,3)}$};
			\draw [fill, black] (-3,-6) circle[radius = 1mm];
			\node [below] at (-2,-6) {${(5,2)}$};
			\draw [fill, black] (-2,-6) circle[radius = 1mm];
			\node [below] at (-1,-6) {${(5,1)}$};
			\draw [fill, black] (-1,-6) circle[radius = 1mm];
			\node [below] at (1,-6) {${(1,1)}$};
			\draw [fill, black] (1,-6) circle[radius = 1mm];
			\node [below] at (2,-6) {${(1,2)}$};
			\draw [fill, black] (2,-6) circle[radius = 1mm];
			\node [below] at (3,-6) {${(1,3)}$};
			\draw [fill, black] (3,-6) circle[radius = 1mm];
			\node [below] at (4,-6) {${(1,4)}$};
			\draw [fill, black] (4,-6) circle[radius = 1mm];
			\node [left] at (-6,1) {${(1,0)}$};
			\draw [fill, black] (-6,1) circle[radius = 1mm];
			\node [left] at (-6,-1) {${(2,0)}$};
			\draw [fill, black] (-6,-1) circle[radius = 1mm];
			\node [left] at (-6,-3) {${(4,0)}$};
			\draw [fill, black] (-6,-3) circle[radius = 1mm];
			\node [right] at (6,1) {${(5,0)}$};
			\draw [fill, black] (6,1) circle[radius = 1mm];
			\node [right] at (6,-1) {${(7,0)}$};
			\draw [fill, black] (6,-1) circle[radius = 1mm];
			\node [right] at (6,-3) {${(8,0)}$};
			\draw [fill, black] (6,-3) circle[radius = 1mm];
			
			\draw [] (-6,-1) -- (-6,1);
			\draw [] (-6,-1) -- (-6,-3);
			\draw [] (-6,-1) -- (-4,2);
			\draw [] (-6,-1) -- (-4,0);
			\draw [] (-6,-1) -- (-4,-2);
			\draw [] (-6,-1) -- (-4,-4);
			\draw [] (-6,1) -- (-4,4);
			\draw [] (-6,1) -- (-3,4);
			\draw [] (-6,1) -- (-2,4);
			\draw [] (-6,1) -- (-1,4);
			\draw [] (-6,1) -- (6,1);
			\draw [] (-6,-3) -- (6,-3);
			\draw [] (-6,-3) -- (-4,-6);
			\draw [] (-6,-3) -- (-3,-6);
			\draw [] (-6,-3) -- (-2,-6);
			\draw [] (-6,-3) -- (-1,-6);
			\draw [] (6,-1) -- (6,-3);
			\draw [] (6,1) -- (1,4);
			\draw [] (6,1) -- (2,4);
			\draw [] (6,1) -- (3,4);
			\draw [] (6,1) -- (4,4);
			\draw [] (6,-1) -- (6,1);
			\draw [] (6,-1) -- (4,2);
			\draw [] (6,-1) -- (4,0);
			\draw [] (6,-1) -- (4,-2);
			\draw [] (6,-1) -- (4,-4);
			\draw [] (6,-3) -- (1,-6);
			\draw [] (6,-3) -- (2,-6);
			\draw [] (6,-3) -- (3,-6);
			\draw [] (6,-3) -- (4,-6);
			\draw [] (-4,2) -- (4,2);
			\draw [] (-4,2) -- (4,0);
			\draw [] (-4,2) -- (4,-2);
			\draw [] (-4,2) -- (-4,4);
			\draw [] (-4,2) -- (-1,-6);
			\draw [] (-4,0) -- (4,0);
			\draw [] (-4,0) -- (4,2);
			\draw [] (-4,0) -- (4,-4);
			\draw [] (-4,0) -- (-3,4);
			\draw [] (-4,0) -- (-2,-6);
			\draw [] (-4,-2) -- (4,2);
			\draw [] (-4,-2) -- (4,-2);
			\draw [] (-4,-2) -- (4,-4);
			\draw [] (-4,-2) -- (-2,4);
			\draw [] (-4,-2) -- (-3,-6);
			\draw [] (-4,-4) -- (4,0);
			\draw [] (-4,-4) -- (4,-2);
			\draw [] (-4,-4) -- (4,-4);
			\draw [] (-4,-4) -- (-1,4);
			\draw [] (-4,-4) -- (-4,-6);
			\draw [] (4,2) -- (4,4);
			\draw [] (4,2) -- (1,-6);
			\draw [] (4,0) -- (2,-6);
			\draw [] (4,0) -- (3,4);
			\draw [] (4,-2) -- (3,-6);
			\draw [] (4,-2) -- (2,4);
			\draw [] (4,-4) -- (4,-6);
			\draw [] (4,-4) -- (1,4);
			\draw [] (-4,4) -- (1,-6);
			\draw [] (-4,4) -- (2,-6);
			\draw [] (-4,4) -- (3,-6);
			\path [-] (-4,4) edge[bend left=60] node {} (1,4);
			
			\draw [] (-3,4) -- (1,-6);
			\draw [] (-3,4) -- (2,-6);
			\draw [] (-3,4) -- (4,-6);
			\path [-] (-3,4) edge[bend left=60] node {} (2,4);
			
			\draw [] (-2,4) -- (1,-6);
			\draw [] (-2,4) -- (4,-6);
			\draw [] (-2,4) -- (3,-6);
			\path [-] (-2,4) edge[bend left=60] node {} (3,4);
			
			\draw [] (-1,4) -- (2,-6);
			\draw [] (-1,4) -- (4,-6);
			\draw [] (-1,4) -- (3,-6);
			\path [-] (-1,4) edge[bend left=60] node {} (4,4);
			
			\draw [] (1,4) -- (-3,-6);
			\draw [] (1,4) -- (-2,-6);
			\draw [] (1,4) -- (-1,-6);
			
			\draw [] (2,4) -- (-4,-6);
			\draw [] (2,4) -- (-2,-6);
			\draw [] (2,4) -- (-1,-6);
			
			\draw [] (3,4) -- (-3,-6);
			\draw [] (3,4) -- (-4,-6);
			\draw [] (3,4) -- (-1,-6);
			
			\draw [] (4,4) -- (-3,-6);
			\draw [] (4,4) -- (-2,-6);
			\draw [] (4,4) -- (-4,-6);
			
			\path [-] (-4,-6) edge[bend right=60] node {} (1,-6);
			\path [-] (-3,-6) edge[bend right=60] node {} (2,-6);
			\path [-] (-2,-6) edge[bend right=60] node {} (3,-6);
			\path [-] (-1,-6) edge[bend right=60] node {} (4,-6);

		\end{tikzpicture}
	\end{subfigure}   
	\caption{The POE graph $\Gamma({\mathbb{Z}_{3^25}})$ consists of the above two connected components. The component at the top is isomorphic to $\Gamma({\mathbb{Z}_{15}})$, as discussed in \autoref{Disconnected_Lemma_2}.}
\end{figure}

\begin{lemma}\label{Disconnected_Lemma_2}
	The graph  $\Gamma(\mathbb{Z}_{p_1^{n_1}p_2^{n_2}})$ is disconnected and one connected component is isomorphic to $\Gamma(\mathbb{Z}_{p_1p_2})$.
\end{lemma}

\begin{proof}
Let $o(x_1,e_2)=p_1$ and $o(e_1,y_1)=p_2$, then using \autoref{lcm_lemma} we have $o(x_1,y_1)=p_1p_2$ and $o(x_1^{-1},y_1)=o(x_1,y_1^{-1})=p_1p_2$. All the elements of order $p_1p_2$ are of the form $\{(x,y):o(x)=p_1,o(y)=p_2\}$. Therefore, maintaining the adjacency relation of POE graphs we have $(x_1,e_2)\sim (x_1^{-1},y_1)$ and $(e_1,y_1)\sim (x_1,y_1^{-1})$. Since  $(e_1,e_2)\sim (x_1,e_2)\sim (x_1^{-1},y_1)$ and $(e_1,e_2)\sim (e_1,y_1)\sim (x_1,y_1^{-1})$, all the elements of order $p_1p_2$ are in the same connected component containing the identity $(e_1, e_2)$, which we marked as $C^{(0)}$.
  
Let $C^{(0)}$ contains an element $(x_2,y_2)$ such that $o(x_2,y_2)=p_1^{k_1}p_2^{k_2}$ with $k_1\geq 2$ or $k_2\geq 2$, then $(x_2,y_2)$ is adjacent to any one of the elements of order $p_1$, $p_2$ or $p_1p_2$. Now, applying the \autoref{lcm_lemma} we have the following cases:
\begin{itemize}
	\item 
		If $(x_2,y_2)\sim (x_1,e_2)$, where $o(x_1,e_2)=p_1$, then $o((x_2,y_2)(x_1,e_2))=p_1^{k_1}p_2^{k_2}$, which is not a prime.
 	\item 
 		If $(x_2,y_2)\sim (e_1,y_1)$, where $o(e_1,y_1)=p_2$, then $o((x_2,y_2)(e_1,y_1))=p_1^{k_1}p_2^{k_2}$, which is not a prime.
 	\item 
 		If $(x_2,y_2)\sim (x_1,y_1)$, where $o(x_1,y_1)=p_1p_2$, then $o((x_2,y_2)(x_1,y_1))=p_1^{k_1}p_2^{k_2}$, which is not a prime.
\end{itemize} 
Each of the above cases contradicts to the definition of POE graphs. Therefore, $C^{(0)}$ has no element of order other than $p_1,p_2$ and $p_1p_2$. Hence, the graph $\Gamma(\mathbb{Z}_{p_1^{n_1}}\times \mathbb{Z}_{p_2^{n_2}})$ has multiple components. Therefore, it is disconnected.
 
Now, we prove that the connected component $C^{(0)}$ is isomorphic to $\Gamma(\mathbb{Z}_{p_1p_2})$. We can express the group $\mathbb{Z}_{p_1p_2}$ as $\mathbb{Z}_{p_1}\times \mathbb{Z}_{p_2}$ and the elements of the groups are $(e_1,e_2)$, elements of order $p_1,p_2$ and $p_1p_2$ which are of the same form as they are in $C^{(0)}$. Also, the number of elements of every order is identical in both $C^{(0)}$ and $\Gamma(\mathbb{Z}_{p_1}\times \mathbb{Z}_{p_2})$. By the definition of POE graph, two elements of $C^{(0)}$ and the graph $\Gamma(\mathbb{Z}_{p_1}\times \mathbb{Z}_{p_2})$  are adjacent in the same manner. Hence, the connected component $C$ is isomorphic to $\Gamma(\mathbb{Z}_{p_1}\times \mathbb{Z}_{p_2})$ that is $\Gamma(\mathbb{Z}_{p_1p_2})$ following the same manner as \autoref{Disconnected lemma 2}. \end{proof}

\begin{lemma}\label{adjacency_lemma_3}
	In the graph  $\Gamma(\mathbb{Z}_{p_1^{n_1}p_2^{n_2}})$, the elements of orders $p_1^{k_1}$ are adjacent to the elements of order $p_1^{k_1}$ and $p_1^{k_1}p_2$. Also,the elements of order $p_2^{k_2}$ is adjacent to the elements of order $p_2^{k_2}$ and $p_1p_2^{k_2}$, where $k_1, k_2 = 2, 3, \dots, (n_2-1)$. 
\end{lemma}

\begin{proof}
	The elements of order $p_1^{k_1}$ of the group $\mathbb{Z}_{p_1^{k_1}}\times \mathbb{Z}_{p_2^{k_2}} = \{(x,y): x\in \mathbb{Z}_{p_1^{n_1}},y\in \mathbb{Z}_{p_2^{k_2}}\}$ can be expressed as follows:
	$$\{(x,e_2): o(x)=p_1^{k_1} ~\text{where}~ k_1 = 2, 3, \dots, (n_1-1)\}.$$
	If $o(x_1,y_1) = p_1^{k_1}$ then $(x_1,y_1) = (x_1,e_2)$ with $o(x_1)=p_1^{k_1}$. If $(x_1,y_1)\sim (x_2,y_2)$, then following the definition of POE graphs, we have either $o((x_1,y_1)(x_2,y_2))=p_1$ or $p_2$. Now, we have the following cases:
	
	\textbf{Case - I:} Let $o((x_1,y_1)(x_2,y_2))=p_1$. Then,
	\begin{equation}
		\begin{split}
			& o((x_1,y_1)(x_2,y_2))=p_1\\
			\implies & o(x_1x_2,y_1y_2)=p_1\\
			\implies & o(x_1x_2)=p_1, y_1y_2=e_2.\\
		\end{split}
	\end{equation}
	Recall that, we considered the elements of order $p_1^{k_1}$ and $o(x_1,y_1)=p_1^{k_1}$. Also, we know that in the element $(x_1,y_1)$, the first component $x_1$ is coming from the group $\mathbb{Z}_{p_1^{n_1}}$ and and the second component is coming from the group $\mathbb{Z}_{p_2^{n_2}}$. Clearly, if $o(x_1,y_1)=p_1^{k_1}$, then $o(x_1)=p_1^{k_1}$ and $o(y_1)=1$ that is $y_1=e_2$. Also, we have $y_2=e_2$, thus, $(x_2,y_2)=(x_2,e_2)$. Now, the order of $((x_1,y_1)(x_2,y_2))$ depends on $o(x_1x_2)$, only. As $(x_1,e_2)$ and $(x_2,e_2)$ can be treated as the elements of the group $\mathbb{Z}_{p_1^{n_1}}$, following the \autoref{same-order-lemma}, $o(x_2,e_2)=p_1^{k_1}$.	Thus, $o(x_1,y_1)=o(x_2,y_2)=p_1^{k_1}$ for $k_1=2,3,\dots, (n_1-1)$.
	
	\textbf{Case - II:} Let $o((x_1,y_1)(x_2,y_2))=p_2.$
	
	\begin{equation}
		\begin{split}
			& o((x_1,y_1)(x_2,y_2))=p_2\\
			\implies & o(x_1x_2,y_1y_2)=p_2\\
			\implies & x_1x_2=e_1,o( y_1y_2)=p_2.\\
		\end{split}
	\end{equation}
	It is assumed that $o(x_1)=p_1^{k_1}$, that is $x_1 \neq e_1$. Clearly, $x_2=x_1^{-1}$ and $o(y_1y_2)=p_2$. Since $y_1=e_2$, $o(y_2)=p_2$. Thus, $o(x_2,y_2)=p_1^{k_1}p_2$ for $k_1=2,3,\dots,n_1-1$. 
	
	Combining the above two cases, we have the result. In a similar manner, the elements of order $p_2^{k_2}$ is adjacent to the elements of order $p_2^{k_2}$ and $p_1p_2^{k_2}$ with $k_2=2,3,\dots, n_2-1$.
\end{proof}

\begin{corollary}\label{adjacency_4}
	The elements of order $p_1^{k_1}p_2$ is adjacent to the elements of order $p_1^{k_1}$ and $p_1^{k_1}p_2$, where $k_1\geq 2$. Also, the elements of order $p_1p_2^{k_2}$ are adjacent to the elements of order $p_2^{k_2}$ and $p_1p_2^{k_2}$, where $k_2\geq 2$. 
\end{corollary}

\begin{proof}
	Following the \autoref{adjacency_lemma_3}, if an element $(x_1,y_1)$ of order $p_1^{k_1}p_2$ is adjacent to the element $(x_2,y_2)$ of order $p_1^{k_1}$, then $(x_2,y_2)=(x_1^{-1},e_2)$. Here, $o((x_1,y_1)(x_2,y_2))=p_2$. Let $o(x_1,y_1)=p_1^{k_1}p_2$ and $(x_1,y_1)\sim (x_2,y_2)$ with $o((x_1,y_1)(x_2,y_2))=p_1$. Thus,
	\begin{equation}
		\begin{split}
			& o((x_1,y_1)(x_2,y_2))=p_1 \\
			\implies & o(x_1x_2,y_1y_2)=p_1\\
			\implies & o(x_1x_2)=p_1, y_1y_2=e_2\\
			\implies & o(x_1)=o(x_2) , y_2=y_1^{-1}\\
			\implies & o(x_1,y_1) = o(x_2,y_2).\\
		\end{split}
	\end{equation}
	It is clear that, for an element $(x_1,y_1)$ of order $p_1^{k_1}p_2$ is adjacent to the elements of order $p_1^{k_1}$ and $p_1^{k_1}p_2$. 
	
	Similarly, the elements of order $p_1p_2^{k_2}$ are adjacent to the elements of order $p_2^{k_2}$ and $p_1p_2^{k_2}$. 
\end{proof}

By the \autoref{adjacency_lemma_3} and its \autoref{adjacency_4}, for $k_1\geq 2$ the elements of order $p_1^{k_1}$ and $p_1^{k_1}p_2$ belong to the same component according to their adjacency condition. Following the \autoref{Disconnected_Lemma_2}, the graph $\Gamma(\mathbb{Z}_{p_1^{n_1}p_2^{n_2}})$ is disconnected, we can denote the connected components of the graph as $C^{(1)},C^{(2)}, \dots, C^{(l)}$. Let $C^{(1)}$ be a connected component containing the elements of order $p_1^{k_1}$ and $p_1^{k_1}p_2$, where $k_1\geq 2$. 

\begin{lemma}\label{elements_count_1}
	The component $C^{(1)}$ contains at least $2p_1p_2$ vertices with degree $(p_1 + p_2 - 2)$.
\end{lemma} 

\begin{proof}
	Recall from \autoref{distance-lemma} that an element and its inverse belong to the same connected component for the POE graph of an Abelian group.
	
	Let $(x_1,y_1)\in C^{(1)}$ such that $o(x_1,y_1)=p_1^{k_1}$. Then, $(x_1,y_1)=(x_1,e_2)$ and by \autoref{degree-lemma-2}, each element of order $p_1^{k_1}$ is adjacent to $(p_1-1)$ elements of order $p_1^{k_1}$. Since $(x_1,y_1) = (x_1,e_2)$, the elements of order $p_1^{k_1}$ can be treated as the vertices of the graph $\Gamma(\mathbb{Z}_{p_1^{n_1}})$. Thus, following \autoref{degree-lemma-2}, and \autoref{structure theorem 1}, we say $C^{(1)}$ contains $2p_1$ elements of order $p_1^{k_1}$. Also, by \autoref{adjacency_lemma_3}, if an element $(x_1,e_2)$ of order $p_1^{k_1}$ is adjacent to an element $(x_2,y_2)$ of order $p_1^{k_1}p_2$ then $(x_2,y_2)=(x_1^{-1},y_2)$ with $o(y_2)=p_2$. Since the group $\mathbb{Z}_{p_1^{n_1}p_2^{n_2}}$ contains $(p_2-1)$ elements of order $p_2$, then for each vertex in $C^{(1)}$ associated with the elements of order $p_1^{k_1}$, we get $(p_2-1)$ vertices associated with the elements of order $p_1^{k_1}p_2$. By \autoref{same-order-lemma} and \autoref{degree-lemma-2}, each vertex associated with the element of order $p_1^{k_1}$ in $C^{(1)}$ has degree $(p_1+p_2-2)$. Thus, the number of vertices in $C^{(1)}$ is at least $2p_1+2p_1(p_2-1)=2p_1p_2$. 
	\end{proof} 
	 
	 \begin{lemma}\label{break_lemma_p_1p_2}
	 	The set of elements of order $p_1^{k_1}$ in the component $C^{(1)}$ are adjacent to $2p_1(p_2-1)$ number of elements of order $p_1^{k_1}p_2$ to form the $p_2-$edges.
	 \end{lemma}
	\begin{proof}
	Now, we need to prove that these are the only $2p_1p_2$ number of elements in $C^{(1)}$. By \autoref{adjacency_4}, the elements of order $p_1^{k_1}p_2$ is adjacent to the elements of order $p_1^{k_1}$ and $p_1^{k_1}p_2$. Let $(x_1,e_2)\sim (x_2,y_2)$, where $o(x_1,e_2)=p_1^{k_1}$ and $o(x_2,y_2)=p_1^{k_1}p_2$. Clearly, $x_2=x_1^{-1}$ and $o(y_2)=p_2$. Thus, an element $(x_2,y_2)$ of order $p_1^{k_1}p_2$ can be adjacent to only one element of order $p_1^{k_1}$.
	
	 Let us denote the elements of order $p_1^{k_1}(k_1\geq 2)$ as $$S_1=\{(x_1,e_2),(x_2,e_2),\dots, (x_{p_1},e_2),\dots, (x_1^{-1},e_2), (x_2^{-1},e_2), \dots, (x_{p_1}^{-1},e_2)\}.$$ Let the elements $(x_i,e_2)$ of order $p_1^{k_1}(k_1\geq 2)$ are adjacent to $$ S_2=\{(x_i^{-1},y_1),(x_i^{-1},y_2),\dots, (x_i^{-1},y_{\frac{p_2-1}{2}})\}(i=1,2,\dots , p_1)$$ and $$S_3=\{(x_i^{-1},y_1^{-1}), (x_i^{-1},y_2^{-1}),\dots,
	 (x_i^{-1},y_{\frac{p_2-1}{2}}^{-1}) (i=1,2,\dots , p_1)\}$$ of order $p_1^{k_1}p_2(k_1\geq 2)$.

	  Similarly, the elements of order $p_1^{k_1}(k_1\geq 2)$ of the form $(x_i^{-1},e_2)$ are adjacent to the elements
	  $$S_4=\{(x_i,y_1),(x_i,y_2),\dots, (x_i,y_{\frac{p_2-1}{2}})\}(i=1,2,\dots , p_1)$$ and $$S_5=\{(x_i,y_1^{-1}),(x_i,y_2^{-1}),\dots, (x_i,y_{\frac{p_2-1}{2}}^{-1})\} (i=1,2,\dots , p_1)$$ of order $p_1^{k_1}p_2(k_1\geq 2)$. Here, $|S_2|=|S_3|=|S_4|=|S_5|=\frac{p_1(p_2-1)}{2}$. Hence, $|S_2|+|S_3|+|S_4|+|S_5|=2p_1(p_2-1)$.
	   \end{proof}
	   
	   \begin{lemma}\label{break_lemma_p_1p_2(a)}
	   	The connected component $C^{(1)}$ contains only $2p_12p_2$ number of vertices.
	   \end{lemma}
	   
	   \begin{proof}
	   	Following the \autoref{break_lemma_p_1p_2}, the connected component $C^{(1)}$ contains all the elements from $S_2,S_3,S_4,S_5$.
	 Now, we need to prove that $V(C^{(1)})=S_1\cup S_2 \cup S_3 \cup S_4 \cup S_5$.
	
	Here, the elements of order $p_1^{k_1}p_2(k_1\geq 2)$ in $C^{(1)}$ are $S_2 \cup S_3 \cup S_4 \cup S_5$. We have already calculated the degree of the elements of order $p_1^{k_1}(k_1\geq 2)$, and thus, we are left with the total number of elements of order $p_1^{k_1}p_2(k_1\geq 2)$ and their degrees.
	
	Each element from $S_2$ is adjacent to an element of order $p_1^{k_1}(k_1\geq 2)$. Let $(x_i,y_m)\in S_2$, then by \autoref{adjacency_4}, the only element of order $p_1^{k_1}(k_1\geq 2)$ is $(x_i^{-1},e_2)\in S_1$. Similarly, if $(x_i^{-1},y_m^{-1})\in S_3$, then, so, it is adjacent to the element $(x_i,e_2)\in S_1$; if $(x_i,y_m)\in S_4$, then it is adjacent to $(x_i^{-1},e_2)\in S_1$; if $(x_i,y_m^{-1})\in S_5$, it is adjacent to $(x_i^{-1},e_2)\in S_1$. Following the \autoref{adjacency_4}, it is clear that, each element from $S_2,S_3,S_4,S_5$ can be adjacent to only one element of order $p_1^{k_1}(k_1\geq 2)$ in $S_1$.
	
	 Let $(x_s^{-1},y_m)\in S_2$, then $(x_s^{-1},y_m)\sim (x_s,e_2)$ and $((x_s^{-1},y_m)(x_s,e_2))=y_m$. This composition gives an element of order $p_2$. Following the  \autoref{adjacency_4}, $(x_s^{-1},y_m)$ is adjacent to $(x_s,y_j)\in S_4 (j=1,2,\dots , \frac{p_2-1}{2})$ and $(x_s,y_j^{-1})\in S_5 (j=1,2,\dots, \frac{p_2-1}{2}) \text{and} j\neq m$. Composition with the elements of $S_4$ and $S_5$ and $(x_s,e_2)$ gives the total $p_2-1$ number of elements of order $p_2$. 
	
	The element $(x_s^{-1},y_m)\sim (x,y_m^{-1})$ in $S_3$ and $S_5$  and by the \autoref{structure theorem 1}, these are $p_1-1$ in total. Thus, the elements from $S_2$ is adjacent to $p_1+p_2-2$ number of elements in $C^{(1)}$.   
	 
	Similarly, we can prove that the elements from $S_3,S_4,S_5$ are adjacent to $p_1+p_2-2$ number of elements and they are adjacent to the elements from $S_1,S_2,S_3,S_4,S_5$. Hence, $V(C^{(1)})=S_1\cup S_2 \cup S_3 \cup S_4 \cup S_5$. Thus, $|V(C^{1})|=2p_1p_2$.
\end{proof}

Let $C^{(2)}$ be a connected component of elements of order $p_2^{k_2}$ and $p_1p_2^{k_2}$ where $k_2\geq 2$.

\begin{corollary}
	The connected component $C^{(2)}$ contains $2p_1p_2$ number of vertices with degree $(p_1+p_2-2)$.
\end{corollary}

\begin{lemma}\label{same_order_lemma_3}
	For an element $(x_1,y_1)\in \mathbb{Z}_{p_1^{n_1}p_2^{n_2}}$ such that $o(x_1,y_1)=p_1^{k_1}p_2^{k_2}$ is adjacent to $(x_2,y_2)$ such that $o(x_2,y_2)=p_1^{k_1}p_2^{k_2}$ where $k_1\geq 2$, and $k_2\geq 2$.
\end{lemma}

\begin{proof}
	Let $o(x_1,y_1)=p_1^{k_1}p_2^{k_2}$, that is $o(x_1)=p_1^{k_1}, o(y_1)=p_2^{k_2}$ and $(x_1,y_1)\sim (x_2,y_2)$. Then, we have the following cases:
	
	\textbf{Case-I:} Let $(x_1,y_1)\sim (x_2,y_2)$ and $o((x_1,y_1)(x_2,y_2))=p_1$. This implies 
	\begin{equation}
		\begin{split}
			& o(x_1x_2,y_1y_2)=p_1\\
			\implies & o(x_1x_2)=p_1, y_1y_2=e_2\\
			\implies & o(x_1x_2)=p_1, y_1=y_2^{-1}.\\
			\end{split}
	\end{equation}
 Since the order of $((x_1,y_1)(x_2,y_2))$ depends on the order of $x_1x_2$ only, then, we can treat the elements $x_1,x_2$ as the elements of the group $\mathbb{Z}_{p_1^{n_1}}$. Following \autoref{same-order-lemma} and by the adjacency condition of two vertices associated to the two elements of order $p_1^{k_1}$. In the graph $\Gamma(\mathbb{Z}_{p_1^{n_1}})$, we have, $o(x_1)=o(x_2)$ that is $o(x_2)=p_1^{k_1}(k_1\geq 2)$. Also, $y_1=y_2^{-1} \implies o(y_1)=o(y_2)=p_2^{k_2} \implies o(x_2,y_2)=p_1^{k_1}p_2^{k_2}$. Hence, $o(x_1,y_1)=o(x_2,y_2)=p_1^{k_1}p_2^{k_2}$.
 
\textbf{Case-II:} Let $(x_1,y_1)\sim (x_2,y_2)$ and $o((x_1,y_1)(x_2,y_2))=p_2$. This implies 
\begin{equation}
	\begin{split}
 		& o(x_1x_2,y_1y_2)=p_2\\
 		\implies & x_1x_2=e_1, o(y_1y_2)=p_2\\
 		\implies & x_1=x_2^{-1}, o(y_1y_2)=p_2.\\
 	\end{split}
\end{equation}
Following the same manner in Case-I, we can conclude that $o(y_1)=o(y_2)=p_2^{k_2}(k_2\geq 2)$ and $o(x_1)=o(x_2)=p_1^{k_1}(k_1\geq 2)$. Therefore, $o(x_1,y_1)=o(x_2,y_2)=p_1^{k_1}p_2^{k_2}$.
\end{proof}

By the above \autoref{same_order_lemma_3}, if $o(x_1,y_1)=p_1^{k_1}p_2^{k_2}$, and $(x_1,y_1)\sim (x_2,y_2)$, then either $(x_2,y_2)=(x_2,y_1^{-1})$ with $o(x_1)=o(x_2)$ or $(x_2,y_2)=(x_1^{-1},y_2)$ with $o(y_1)=o(y_2)$ where $k_1,k_2\geq 2$. Let $C^{(3)}$ be a connected component of the elements of order $p_1^{k_1}p_2^{k_2}(k_1,k_2\geq 2)$.

\begin{lemma}\label{structure lemma 2}
The connected component $C^{(3)}$ contains at least $2p_1p_2$ number of vertices with degree $(p_1+p_2-2)$.
\end{lemma}

\begin{proof}
	Following the \autoref{same_order_lemma_3}, elements of order $p_1^{k_1}p_2^{k_2}(k_1,k_2\geq 2)$ are adjacent to the elements of same order. Let $(x_i,y_i)\in C^{(3)}$ and $o((x_i,y_i)(x_j,y_j))=p_1$, then $(x_j,y_j)=(x_j,y_i^{-1})$. Then, by the \autoref{same_order_lemma_3}, $(x_i,y_i)\sim (x_j,y_i^{-1})$ when their composition gives an element of order $p_1$. Since the group $\mathbb{Z}_{p_1^{n_1}p_2^{n_2}}$ contains $p_1-1$ number of elements of order $p_1$, then $(x_i,y_i)$ is adjacent to $(x_j,y_i^{-1})$, where $o(x_ix_j)=p_1$. Following the \autoref{adjacency_lemma_3}, it is found that, $(x_i,y_i)$ is adjacent to $p_1-1$ number of such elements. 
	
	Let $(x_1,y_1)\sim (x_2,y_2)$ and $o((x_1,y_1)(x_2,y_2))=p_i(i=1,2)$, then we can say there is a $p_i-edge$ between the two elements $(x_1,y_1)$ and $(x_2,y_2)$.
	
	 Let $(x,y)\in C^{(3)}$ be an element and it is adjacent to the set of elements $S_1$. The element $(x,y)$ has $p_1-edge$ with every element of $S_1$. The set $S_1$ is 
	 \begin{equation}
	 	S_1=\{(x_i,y^{-1}): i=1,2,\dots p_1-1\}.
	 \end{equation}
	By the \autoref{distance-lemma}, $(x,y)$ and its inverse $(x^{-1},y^{-1})$ belong to the same component $C^{(3)}$, $(x^{-1},y^{-1})$ has $p_1-edge$ with the set of elements 
	\begin{equation}
		S_1^{-1}=\{(x_i^{-1},y): i=1,2,\dots p_1-1\}.
	\end{equation}
	Here, $|S_1|=|S_1^{-1}|=p_1-1$. 
	By the \autoref{adjacency-lemma}, $(x_i,y^{-1})\sim (x_j^{-1},y)$, where $i,j=1,2,\dots, p_1-1$, $o(x_ix_j^{-1})=p_1$ and $i\neq j$. Therefore, the elements of $S_1\cup S_1^{-1}\cup \{(x,y),(x^{-1},y^{-1})\}$ has their $p_1-edge$ inside the set. Also, $|S_1\cup S_1^{-1}\cup \{(x,y),(x^{-1},y^{-1})\}|=2p_1$. 
	
	Now, every element from the set $S_1\cup S_1^{-1}\cup \{(x,y),(x^{-1},y^{-1})\}$  is adjacent to $p_2-1$ number of elements with whom they have their $p_2-edge$. We have $2p_1+2p_1(p_2-1)=2p_1p_2$ number of elements in $C^{(3)}$. We only need to show these $2p_1p_2$ elements are the only elements in the component $C^{(3)}$.
	\end{proof}
	
	\begin{lemma}\label{break_lemma_3}
		The connected component $C^{(3)}$ contains only $2p_1p_2$ vertices.
	\end{lemma}
	
	\begin{proof}
	the vertices $(x,y)$ and $(x^{-1},y^{-1})$ makes $p_2$-edge with the elements in $T$ and $T^{-1}$, where 
	\begin{equation}
		T=\{(x^{-1},y_j):j=1,2,\dots, p_2-1\}, \text{and}~ 
		T^{-1}=\{(x,y_j^{-1}):j=1,2,\dots, p_2-1\}.
	\end{equation}
	Also, for $i = 1, 2, \dots, (p_1 - 1)$ the vertices $(x_i,y^{-1})$ and $(x_i^{-1},y)$ are adjacent to the set of elements $T_i$ and $T_i^{-1}$ to make the $p_2$-edge, where	
	\begin{equation}
		T_i=\{(x_i^{-1},y_j^{-1}):j=1,2,\dots, p_2-1\}, ~\text{and}~ 
		T_i^{-1}=\{(x_i,y_j):j=1,2,\dots, p_2-1\}.
	\end{equation}
	
	We need to show that $V(C^{(3)}) = S_1\cup S_2\cup T\cup T^{-1}\cup \{{\cup}_{i=1}^{p_2-1}T_i\} \cup \{{\cup}_{i=1}^{p_2-1}T_i^{-1}\} \cup \{(x,y),(x^{-1},y^{-1})\}$. The elements in $T$, $(x^{-1},y_j)$, are adjacent to $(x,y_k^{-1})$ in $T^{-1}$, where $j\neq k$ and $o((x^{-1},y_j)(x,y_k^{-1}))=p_2$. Also, the elements $(x^{-1},y_j)\sim (x_i^{-1},y_j^{-1})$ and $(x,y_j^{-1})\sim (x_i,y_j)$ for $i=1,2,\dots, p_1-1$ and for each $j$ make $p_1$-edges. The elements $(x_i^{-1},y_j^{-1})$ in $T_i$ are adjacent to $(x_i,y_k)$ in $T_i^{-1}$ for each $i$ and $j\neq k$ create $p_2$-edges. In addition, $(x_i,y_j) \sim (x,y_j^{-1})$ as well as $(x_i^{-1},y_j^{-1})\sim (x_k,y_j)$ and $(x_i^{-1},y_j^{-1})\sim (x^{-1},y_j)$,  for each $j$, where $i\neq k$ to build up $p_1$-edges. 
	From the above calculations, we can see that the elements from $S_1\cup S_2\cup T\cup T^{-1}\cup \{{\cup}_{i=1}^{p_2-1}T_i\} \cup \{{\cup}_{i=1}^{p_2-1}T_i^{-1}\} \cup \{(x,y),(x^{-1},y^{-1})\}$ form all their $p_1$-edges and $p_2$-edges inside the set. Hence, there are no other elements in the component $C^{(3)}$. Since, each one of them has $(p_1+p_2-2)$ edges, the degree of each of the $2p_1p_2$ elements has degree $(p_1+p_2-2)$. The \autoref{adjacency-table} indicates all the adjacency in the component $C^{(3)}$.
	\begin{table}[]
		\centering
		
		\begin{tabular}{|c|c|c|c|c|c|c|c|c|}
			\hline
			. & $(x,y)$ & $(x^{-1},y^{-1})$ & $(x_i,y^{-1})$ & $(x_i^{-1},y)$ & $(x^{-1},y_j)$ & $(x,y_j^{-1})$ & $(x_i,y_j)$ & $(x_i^{-1},y_j^{-1})$\\
			\hline 
			$(x,y)$ & $\nsim$ & $\nsim$ & $p_1$ & $\nsim$ & $p_2$ & $\nsim$ & $\nsim$ & $\nsim$\\
			\hline
			$(x^{-1},y^{-1})$ & $\nsim$ & $\nsim$ & $\nsim$ & $p_1$ & $\nsim$ & $p_2$ & $\nsim$ & $\nsim$\\
			\hline
			$(x_k,y^{-1})$ & $p_1$ & $\nsim$ & $\nsim$ & $p_1(k\neq i)$ & $\nsim$ & $\nsim$ & $\nsim$ & $p_2(k=i)$\\
			\hline 
			$(x_k^{-1},y)$ & $\nsim$ & $p_1$ & $p_1(i\neq k)$ & $\nsim$ & $\nsim$ & $\nsim$ & $p_2$ & $\nsim$\\
			\hline
			$(x^{-1},y_s)$ & $p_2$ & $\nsim$ & $\nsim$ & $\nsim$ & $\nsim$ & $p_2(j\neq s)$ & $\nsim$ & $p_1$\\
			\hline
			$(x,y_s^{-1})$ & $\nsim$ & $p_2$ & $\nsim$ & $\nsim$ & $p_2(j\neq s)$ & $\nsim$ & $p_1$ & $\nsim$\\
			\hline
			$(x_k,y_s)$ & $\nsim$ & $\nsim$ & $\nsim$ & $p_2(k=i)$ & $\nsim$ & $p_1(j=s)$ & $\nsim$ & $\nsim$\\
			\hline
			$(x_k^{-1},y_s^{-1})$ & $\nsim$ & $\nsim$ & $p_2(k=i)$ & $\nsim$ & $p_1(j=s)$ & $\nsim$ & $\nsim$ & $\nsim$\\
			\hline
		\end{tabular}
		\caption{Adjacency relations in the component $C^{(3)}$ in the graph $\Gamma(\mathbb{Z}_{p_1^{n_1}p_2^{n_2}})$ between the elements, which is discussed in \autoref{break_lemma_3}. Here $\nsim$ denotes non-existence of an edge between the vertices. Also, $p_i$ denotes the existence of a $p_i$-edge between the vertices in the corresponding rows and columns. }
		\label{adjacency-table}
	\end{table} 
\end{proof}
 
\begin{thm}
	For two odd primes $p_1$ and $p_2$, $\Gamma(\mathbb{Z}_{p_1^{n_1}p_2^{n_2}})\cong \Gamma(\mathbb{Z}_{p_1^{n_1}}\times\mathbb{Z}_{p_2^{n_2}})$ is the union of $\Gamma(\mathbb{Z}_{p_1p_2})$ and $\frac{p_1^{n_1-1}p_2^{n_2-1}-1}{2}$ number of connected regular components with $2p_1p_2$ vertices and degree $p_1+p_2-2$.
\end{thm}
\begin{proof}
	From \autoref{Disconnected_Lemma_2}, the graph $\Gamma(\mathbb{Z}_{p_1^{n_1}p_2^{n_2}})$ is disconnected and one component is isomorphic to $\Gamma(\mathbb{Z}_{p_1p_2})$. The \autoref{adjacency_lemma_3},  \autoref{elements_count_1}, \autoref{same_order_lemma_3}, \autoref{structure lemma 2} indicate that the other connected components contains $2p_1p_2$ number of elements with each having degree $p_1+p_2-2$. Since one connected component $\Gamma(\mathbb{Z}_{p_1p_2})$ contains $p_1p_2$ number of elements, the remaining components contains $p_1^{n_1}p_2^{n_2}-p_1p_2$. Since each connected components contains $2p_1p_2$ number of elements, $\Gamma(\mathbb{Z}_{p_1^{n_1}p_2^{n_2}})$ has $\frac{p_1^{n_1}p_2^{n_2}-p_1p_1}{2p_1p_2}$ that is $\frac{p_1^{n_1-1}p_2^{n_2-1}-1}{2}$ number of connected components other than $\Gamma(\mathbb{Z}_{p_1p_2})$. Therefore, the graph $\Gamma(\mathbb{Z}_{p_1^{n_1}p_2^{n_2}})$ is the union of $\Gamma(\mathbb{Z}_{p_1p_2})$ and $\frac{p_1^{n_1-1}p_2^{n_2-1}-1}{2}$ number of connected components containing $2p_1p_2$ number of vertices with degree of each vertex $(p_1+p_2-2)$.
\end{proof}

\begin{thm}
	For $k$ odd primes $p_1,p_2,\dots p_k$, $\Gamma(\mathbb{Z}_{p_1^{n_1}p_2^{n_2}\dots p_k^{n_k}})\cong {\Gamma({\mathbb{Z}_{p_1^{n_1}}}\times {\mathbb{Z}_{p_2^{n_2}}} \times \dots {\mathbb{Z}_{p_k^{n_k}}})}$ is the union of $\Gamma(\mathbb{Z}_{p_1p_2\dots p_k})$ and $\frac{p_1^{n_1-1}p_2^{n_2-1}\dots p_k^{n_k-1}}{2}$ number of connected regular components with $2p_1p_2\dots p_k$ vertices and degree $p_1+p_2+\dots p_k-k$.
\end{thm}

\begin{proof}
	The elements of the group ${\Gamma({\mathbb{Z}_{p_1^{n_1}}}\times {\mathbb{Z}_{p_2^{n_2}}} \times \dots {\mathbb{Z}_{p_k^{n_k}}})}$ are of the form $\{(x_1,x_2,\dots x_k): x_i\in \mathbb{Z}_{p_i^{n_i}} \mbox{ for } i=\{1,2,\dots, k\}\}$. The identity is $(e_1,e_2,\dots, e_k)$ where $e_i$ is the identity of $\mathbb{Z}_{p_i^{n_i}}$. There is a path $(e_1,e_2,\dots, e_k)\sim (x_1,e_2,\dots, e_k)\sim (x_1^{-1},x_2,\dots, e_k)\sim (x_1,x_2^{-1}, x_3,\dots, e_k)\sim \dots (x_1^{-1},x_2,\dots, x_k)$ for $k$ even and $(e_1,e_2,\dots, e_k)\sim (x_1,e_2,\dots, e_k)\sim (x_1^{-1},x_2,\dots, e_k)\sim (x_1,x_2^{-1}, x_3,\dots, e_k)\sim \dots (x_1,x_2^{-1},\dots, x_k)$ for $k$ odd, where $o(x_i)=p_i$ for $i=1,2,\dots ,k$. From these paths, we can see that the elements of order $p_i(i=1,2,\dots, k)$, $p_{i_1}p_{i_2}(i_1\neq i_2)$, $p_{i_1}p_{i_2}p_{i_3}(i_1 \neq i_2 \neq i_3)\dots p_{i_1}p_{i_2}\dots p_{i_k}(i_1\neq i_2 \neq \dots \neq i_k)$ belong to the same connected component. Following the same manner as the \autoref{Disconnected_Lemma_2}, we can prove that the graph ${\Gamma({\mathbb{Z}_{p_1^{n_1}}}\times {\mathbb{Z}_{p_2^{n_2}}} \times \dots {\mathbb{Z}_{p_k^{n_k}}})}$ is disconnected and one connected component is isomorphic to $\Gamma(\mathbb{Z}_{p_1p_2\dots p_k})$. Also, by applying the same methods in the  \autoref{adjacency_lemma_3},\autoref{elements_count_1},\autoref{break_lemma_p_1p_2},\autoref{break_lemma_p_1p_2(a)},\autoref{same_order_lemma_3},\autoref{structure lemma 2} and \autoref{break_lemma_3}, we can divide the elements by its order by their respective adjacency criteria, and then we can prove the required result by generalizing the same in mathematical induction method.
\end{proof}

\section*{Acknowledgment}
	The authors are thankful to Dr. Angsuman Das of Presidency University, Kolkata for some discussions at the initial stage of this work.


\begin{thebibliography}{10}
		
		\bibitem{bera2018enhanced}
		Sudip Bera and AK~Bhuniya.
		\newblock
		\href{https://www.worldscientific.com/doi/abs/10.1142/S0219498818501463}{On
			enhanced power graphs of finite groups}.
		\newblock {\em Journal of Algebra and its Applications}, 17(08):1850146, 2018.
		
		\bibitem{magnus2004combinatorial}
		Wilhelm Magnus, Abraham Karrass, and Donald Solitar.
		\newblock {\em
			\href{https://books.google.com/books?hl=en&lr=&id=1LW4s1RDRHQC&oi=fnd&pg=PP1&dq=Combinatorial+group+theory:+Presentations+of+groups+in+terms+of+generators+and+relations&ots=WtAL8-0aDz&sig=fNcVwy0UXEqiG75XyX1665oDRuU}
			{Combinatorial group theory: Presentations of groups in terms of generators
				and relations}}.
		\newblock Courier Corporation, 2004.
		
		\bibitem{cayley1878desiderata}
		Professor Cayley.
		\newblock \href{https://www.jstor.org/stable/2369306}{Desiderata and
			suggestions: No. 2. The Theory of groups: graphical representation}.
		\newblock {\em American journal of mathematics}, 1(2):174--176, 1878.
		
		\bibitem{cameron2011power}
		Peter~J Cameron and Shamik Ghosh.
		\newblock
		\href{https://www.sciencedirect.com/science/article/pii/S0012365X10000531}{The
			power graph of a finite group}.
		\newblock {\em Discrete Mathematics}, 311(13):1220--1222, 2011.
		
		\bibitem{kumar2021recent}
		Ajay Kumar, Lavanya Selvaganesh, Peter~J Cameron, and T~Tamizh Chelvam.
		\newblock
		\href{https://www.sciencedirect.com/science/article/pii/S0024379506000590}{Recent
			developments on the power graph of finite groups--a survey}.
		\newblock {\em AKCE International Journal of Graphs and Combinatorics},
		18(2):65--94, 2021.
		
		\bibitem{banerjee2019new}
		Subarsha Banerjee.
		\newblock \href{https://arxiv.org/abs/1911.02763}{On a new graph defined on the
			order of elements of a finite group}.
		\newblock {\em arXiv preprint arXiv:1911.02763}, 2019.
		
		\bibitem{maslova2016gruenberg}
		Natalia~V Maslova.
		\newblock
		\href{https://www.researchgate.net/profile/Natalia-Maslova/publication/344783992_On_Gruenberg-Kegel_graphs_of_finite_groups/links/5f900efa299bf1b53e37a95b/On-Gruenberg-Kegel-graphs-of-finite-groups.pdf}{On
			the Gruenberg--Kegel graphs of finite groups}.
		\newblock In {\em Proceedings of the 47th International Youth School-Conference
			Modern Problems in Mathematics and its Applications. Ed. by SF Pravdin AA
			Makhnev}, 2016.
		
		\bibitem{biswas2024difference}
		Sucharita Biswas, Peter~J Cameron, Angsuman Das, and Hiranya~Kishore Dey.
		\newblock
		\href{https://www.sciencedirect.com/science/article/pii/S0097316524000712}{On
			the difference of the enhanced power graph and the power graph of a finite
			group}.
		\newblock {\em Journal of Combinatorial Theory, Series A}, 208:105932, 2024.
		
		\bibitem{akbari2017co}
		Saieed Akbari, Babak Miraftab, and Reza Nikandish.
		\newblock
		\href{https://www.cambridge.org/core/journals/canadian-mathematical-bulletin/article/comaximal-graphs-of-subgroups-of-groups/454379FE8CC7F49ECF11DF8FC6C28920}{Co-maximal
			graphs of subgroups of groups}.
		\newblock {\em Canadian Mathematical Bulletin}, 60(1):12--25, 2017.
		
		\bibitem{das2024co}
		Angsuman Das, Manideepa Saha, and Saba Al-Kaseasbeh.
		\newblock
		\href{https://link.springer.com/article/10.1007/s11587-022-00718-0}{On
			co-maximal subgroup graph of a group}.
		\newblock {\em Ricerche di Matematica}, 73(4):2075--2089, 2024.
		
		\bibitem{das2025co}
		Angsuman Das and Manideepa Saha.
		\newblock
		\href{https://link.springer.com/article/10.1007/s11587-023-00836-3}{On
			co-maximal subgroup graph of a group-II}.
		\newblock {\em Ricerche di Matematica}, 74(1):91--104, 2025.
		
		\bibitem{betz2022classifying}
		Alexander Betz and David~A Nash.
		\newblock \href{https://arxiv.org/abs/2006.11315}{Classifying groups with a
			small number of subgroups}.
		\newblock {\em The American Mathematical Monthly}, 129(3):255--267, 2022.
		
		\bibitem{akbari2006diameters}
		S~Akbari, A~Mohammadian, H~Radjavi, and P~Raja.
		\newblock
		\href{https://www.sciencedirect.com/science/article/pii/S0024379506000590}{On
			the diameters of commuting graphs}.
		\newblock {\em Linear algebra and its applications}, 418(1):161--176, 2006.
		
		\bibitem{arunkumar2022super}
		G~Arunkumar, Peter~J Cameron, Rajat~Kanti Nath, and Lavanya Selvaganesh.
		\newblock
		\href{https://link.springer.com/article/10.1007/s00373-022-02496-w}{Super
			graphs on groups, I}.
		\newblock {\em Graphs and Combinatorics}, 38(3):100, 2022.
		
		\bibitem{arunkumar2024super}
		G~Arunkumar, Peter~J Cameron, and Rajat~Kanti Nath.
		\newblock
		\href{https://www.sciencedirect.com/science/article/pii/S0166218X24004025}{Super
			graphs on groups, II}.
		\newblock {\em Discrete Applied Mathematics}, 359:371--382, 2024.
		
		\bibitem{bahrami2019further}
		Zahra Bahrami and Bijan Taeri.
		\newblock \href{https://journals.tubitak.gov.tr/math/vol43/iss5/2/}{Further
			results on the join graph of a finite group}.
		\newblock {\em Turkish Journal of Mathematics}, 43(5):2097--2113, 2019.
		
		\bibitem{huang2018perfect}
		He~Huang, Binzhou Xia, and Sanming Zhou.
		\href{https://epubs.siam.org/doi/10.1137/17M1129532}{Perfect codes in cayley graphs.} 
		\newblock {\em SIAM Journal on Discrete Mathematics}, 32(1):548--559, 2018.
		
		\bibitem{ma2020subgroup}
		Xuanlong Ma, Gary~L Walls, Kaishun Wang, and Sanming Zhou.
		\href{https://epubs.siam.org/doi/10.1137/19M1258013}{Subgroup perfect codes in cayley graphs.}
		\newblock {\em SIAM Journal on Discrete Mathematics}, 34(3):1909--1921, 2020.
		
		\bibitem{wang2023applications}
		Qiuyan Wang, Xiaodan Liang, Rize Jin, and Yang Yan.
		\href{https://ieeexplore.ieee.org/document/10266353/}{Applications of strongly regular cayley graphs to codebooks.} 
		\newblock {\em IEEE Access}, 11:106980--106986, 2023.
		
		\bibitem{bretto2011cayley}
		Alain Bretto and Alain Faisant.
		\href{https://www.sciencedirect.com/science/article/pii/S0747717111001283}{Cayley graphs and g-graphs: Some applications.} 
		\newblock {\em Journal of Symbolic Computation}, 46(12):1403--1412, 2011.
		
		\bibitem{jain2023construction}
		Rupali~S Jain, B~Surendranath Reddy, and Wajid~M Shaikh.
		\href{https://www.worldscientific.com/doi/10.1142/S1793557123502133?srsltid=AfmBOopQIGEPfroa-qIC9c4fPaKcfIx8BT5DllsiRA79znfManPKo-gt}{Construction of linear codes from the unit graph $G(\mathbb{Z}_ n)$.}
		\newblock {\em Asian-European Journal of Mathematics}, 16(11):2350213, 2023.
		
		\bibitem{cameron2021graphs}
		Peter~J Cameron.
		\newblock \href{https://arxiv.org/abs/2102.11177}{Graphs defined on groups}.
		\newblock {\em arXiv preprint arXiv:2102.11177}, 2021.
		
		\bibitem{manna2025prime}
		Tapa Manna, Angsuman Das, and Baby Bhattacharya.
		\newblock
		\href{https://link.springer.com/article/10.1007/s11587-024-00906-0}{Prime
			order element graph of a group: T. Manna et al.}
		\newblock {\em Ricerche di Matematica}, 74(3):1305--1319, 2025.
		
		\bibitem{manna2024forbidden}
		Tapa Manna, Angsuman Das, and Baby Bhattacharya.
		\newblock \href{https://arxiv.org/abs/2412.19905}{Forbidden Subgraphs of Prime
			Order Element Graph}.
		\newblock {\em arXiv preprint arXiv:2412.19905}, 2024.
		
		\bibitem{gallian2021contemporary}
		Joseph Gallian.
		\newblock {\em
			\href{https://api.taylorfrancis.com/content/books/mono/download?identifierName=doi&identifierValue=10.1201/9781003142331&type=googlepdf}{Contemporary
				abstract algebra}}.
		\newblock Chapman and Hall/CRC, 2021.
		
		\bibitem{dummitbasic}
		DS~Dummit, RM~Foote, Upper Saddle~River Prentice-Hall, WA~Adkins, and
		SH~Weintraub.
		\newblock
		\href{http://agt.cie.uma.es/~magomez/Archivos/Asignaturas/%E1lgebra%20homol%F3gica/apuntes%20web/ash/Endmatter.pdf}{Basic
			Group Theory}.
		
		\bibitem{west2001introduction}
		Douglas~Brent West et~al.
		\newblock {\em
			\href{https://www.academia.edu/download/6479067/igtpref.ps}{Introduction to
				graph theory}}, volume~2.
		\newblock Prentice hall Upper Saddle River, 2001.
		
		\bibitem{horn2012matrix}
		Roger~A Horn and Charles~R Johnson.
		\newblock {\em
			\href{https://books.google.com/books?hl=en&lr=&id=O7sgAwAAQBAJ&oi=fnd&pg=PR11&dq=Matrix+analysis&ots=lXRawikui2&sig=SZdK6fUZb8sgn5X4oCO0s76qvEU}{Matrix
				analysis}}.
		\newblock Cambridge university press, 2012.
		
		\bibitem{godsil2013algebraic}
		Chris Godsil and Gordon~F Royle.
		\newblock {\em
			\href{https://books.google.com/books?hl=en&lr=&id=6TasRmIFOxQC&oi=fnd&pg=PP9&dq=Algebraic+graph+theory&ots=lXoY_XTgds&sig=IdV1DVrup6ypyKSucP1wxZUHyTQ}{Algebraic
				graph theory}}, volume 207.
		\newblock Springer Science \& Business Media, 2013.
		
	\end{thebibliography}
\end{document}